\documentclass[reqno, 11pt, a4paper]{amsart}

\usepackage{
		amssymb,
		centernot,
		bbm,
		a4wide,
		libertine
        } 
\usepackage[cal=euler, scr=boondoxo]{mathalfa}
\usepackage[colorlinks=true, linkcolor=blue,citecolor=blue]{hyperref}

\usepackage{comment}
\usepackage{mathtools}

\usepackage{enumitem}
\setlist[enumerate]{leftmargin=*, labelsep=0.5em}
\newtheorem{theorem}{Theorem}[section]
\newtheorem{lemma}[theorem]{Lemma}
\newtheorem{prop}[theorem]{Proposition}

\newtheorem{corollary}[theorem]{Corollary}

\theoremstyle{definition}

\theoremstyle{remark}
\newtheorem{remark}[theorem]{Remark}

\numberwithin{equation}{section}
\DeclareMathOperator*{\esssup}{ess\,sup}
\DeclareMathOperator*{\essinf}{ess\,inf}

\newcommand{\IND}{\mathbbm{1}}

\newcommand{\capa}{\mathrm{Cap}}

\newcommand{\De}{\mathrm{d}}
\newcommand{\homo}{\mathrm{hom}}
\newcommand{\cC}{\ensuremath{\mathcal C}}

\newcommand{\cE}{\ensuremath{\mathcal E}}
\newcommand{\cF}{\ensuremath{\mathcal F}}
\newcommand{\cG}{\ensuremath{\mathcal G}}
\newcommand{\cH}{\ensuremath{\mathcal H}}

\newcommand{\cL}{\ensuremath{\mathcal L}}
\newcommand{\cM}{\ensuremath{\mathcal M}}

\newcommand{\cQ}{\ensuremath{\mathcal Q}}
\newcommand{\cR}{\ensuremath{\mathcal R}}
\newcommand{\cS}{\ensuremath{\mathcal S}}

\newcommand{\cU}{\ensuremath{\mathcal U}}

\newcommand{\cW}{\ensuremath{\mathcal W}}

\newcommand{\bbE}{\ensuremath{\mathbb E}}

\newcommand{\bbG}{\ensuremath{\mathbb G}}

\newcommand{\bbN}{\ensuremath{\mathbb N}}

\newcommand{\bbP}{\ensuremath{\mathbb P}}
\newcommand{\bbQ}{\ensuremath{\mathbb Q}}
\newcommand{\bbR}{\ensuremath{\mathbb R}}

\newcommand{\bbV}{\ensuremath{\mathbb V}}

\newcommand{\bbZ}{\ensuremath{\mathbb Z}}
\newcommand{\bfD}{\ensuremath{\mathbf D}}

\begin{document}

\title{Hard wall repulsion for the discrete Gaussian free field in random environment on $\mathbb{Z}^d$, $d\geq 3$}


\author{Alberto Chiarini}
\address{Dipartimento di Matematica ``Tullio Levi-Civita''}
\curraddr{Universit\`a degli studi di Padova, Via Trieste, 63 -- 35121 Padova}
\email{chiarini@math.unipd.it}
\thanks{}

\author{Emanuele Pasqui}
\address{Dipartimento di Matematica ``Tullio Levi-Civita''}
\curraddr{Universit\`a degli studi di Padova, Via Trieste, 63 -- 35121 Padova}
\email{pasqui@math.unipd.it}
\thanks{}

\begin{abstract} We study the discrete Gaussian free field (harmonic crystal) on $\mathbb{Z}^d$, $d\ge 3$, with uniformly elliptic and bounded random conductances sampled according to a sufficiently mixing environment measure.  We consider the hard wall event that the field is non-negative on the discrete blow-up of a bounded regular domain $V\subseteq\mathbb{R}^d$, and establish a quenched large deviation asymptotic for its probability. The asymptotic rate is characterized by the essential supremum of the on-site variances and the homogenized capacity of $V$, which arises from a quenched invariance principle. 
We then analyze the law of the field conditioned on the hard wall event. We determine the first-order asymptotic profile for its expectation and demonstrate that an entropic push-away of the field from the origin occurs. Furthermore, we characterize the field pathwise behavior under the constraint, showing that, when properly recentered, the field converges weakly to the (quenched) discrete Gaussian free field.
 A major challenge is the lack of translation invariance in the model. We  exploit the superharmonicity of the conditioned expectation of the field to obtain a sharp uniform lower bound for it. We then use this bound to estimate the harmonicity defect of the conditioned expectation. This is a key step that allows us to prove a corresponding sharp upper bound and establish a pathwise entropic repulsion phenomenon.
\end{abstract}

\subjclass[2010]{}
\date{\today}
\dedicatory{}
\maketitle

\section{Introduction}
\label{introduction}

In this article we investigate how   \emph{impurities} affect the behavior of the discrete Gaussian free field (GFF) on $\bbZ^d$, $d \geq 3$, in the presence of a hard wall constraint.
The discrete Gaussian free field, also known as harmonic crystal, has been successfully used as an effective model for the interface of separation between pure phases in statistical mechanics, and also arises as a microscopic description of the fluctuations of atoms in a \emph{homogeneous} crystal at non-zero temperature.
The study of the hard wall event for the Gaussian free field, which consists in the atypical scenario that the field remains positive over a macroscopic domain,  dates back to the eighties with~\cite{maes1987effect}, where it is first proven that when the field is conditioned to be positive over a growing sequence of sets, it gets pushed away to infinity. Later this phenomenon was precisely quantified in~\cite{bolthausen1995entropic} for $d \geq 3$, where the authors establish large deviation asymptotics for the hard wall event and determine the exact distance at which the interface is pushed, discovering that it is at the same order of the expected maximum of the unconditioned field (see also~\cite{DEUSCHEL2000333} for related models). Further, in~\cite{deuschel1999entropic} the authors show that when recentering the law of the conditioned field by such diverging amount, it converges to the law of the GFF itself, establishing a pathwise result for the entropic repulsion phenomenon.  Related to this line of investigation, also the study of the extremes of the field over macroscopic domains has been the subject of recent interest (see~\cite{chiariniCiprianiHazra2015note, chiariniCiprianiHazra2016extremes} for $d \geq 3$ and~\cite{biskup2017extrema,biskup2016extreme, bramsonDingZeitouni2016convergence} for $d=2$).

With the goal of modelling inhomogeneous crystals and understanding how impurities influence their behavior, we are naturally led to consider the GFF with \emph{random conductances}. To our knowledge, the first appearance of this model was in~\cite{caputo2003505} and recently several questions related to it were investigated, like level-set percolation in the supercritical regime (see~\cite{chiarini2021disconnection} and~\cite{drewitzGalloPrevost2024} for results on $\bbZ^d$ and Galton-Watson trees, respectively), stochastic homogenization (see~\cite{andres2025scaling,chiariniRuszel2024stochastic}), and, closer to the aim of this work, the scaling limit of the maximum in the critical dimension $d=2$ for the Gaussian free field on the percolation cluster~(see \cite{schweiger2024}) and for the one with uniformly elliptic and bounded i.i.d.\ conductances. One of the main features of the Gaussian free field with random conductances is that the covariances of the field are characterized by the Green function of the random walk among random conductances, known as the Random Conductance Model. It is well known that when the conductances are uniformly elliptic and bounded, the random walk satisfies a quenched invariance principle (see for instance~\cite{barlow2010invariance,sidoravicius2004quenched}), that is, under diffusive rescaling the random walk converges to a Brownian motion with a deterministic covariance matrix $a^{\rm hom}$ for almost all realizations of the environment.

In this article we focus on the hard wall repulsion phenomenon for the discrete Gaussian free field with uniformly elliptic and bounded random conductances on $\bbZ^d$, $d \geq 3$. That is, we look at the event that the field is non-negative over a region $(N V)\cap \bbZ^d$, $V\subseteq \bbR^d$ of volume of order $N^d$. We first obtain \emph{quenched} large deviation asymptotics for the hard wall event as in~\cite{bolthausen1995entropic} under sufficient mixing conditions on the environment. We use a strategy of proof inspired by the recent works~\cite{sznitzman2015disconnection, nitzschner2018entropic,chiarini2019entropic,chiarini2021disconnection} to overcome the lack of translation invariance in the model. We show that the probability of the hard wall event decays exponentially fast with speed $N^{d-2} \log N$ and a rate that reflects both a large-scale effect coming from the homogenization of the random walk in the form of the capacity of $V$ for the limiting Brownian motion, and a microscopic effect from the spins with the highest variances. The speed $N^{d-2} \log N$ in the asymptotics is due to the long-range correlation in the field and to a push up of the field to a height of order $\sqrt{\log N}$, when constrained to be positive on a domain of volume $N^d$. In fact, with a novel approach, to overcome the lack of translation invariance, we show that the expectation of the field conditioned on the hard wall event is at leading order $\mathscr{h}_V^{\rm hom}(\cdot/N)\sqrt{4\overline{g}\log N}$, with $\overline{g}$ being the essential supremum of the on-site variances of the field, and $\mathscr{h}_V^{\rm hom}$ the homogenized harmonic potential of $V$, while the leading order for the expected maximum of the unconditioned field on $V_N$ is $\sqrt{2d\,\overline{g} \log N}$. This has the interesting consequence that in the special case of i.i.d.\ conductances the field gets an extra push up from the randomness in the conductances when compared to the homogeneous case with the same scaling limit. This is interesting when compared to the critical dimension $d=2$, where the behavior of the maximum is captured solely by the effective diffusivity coming from the homogenization of the random walk among random conductances, see~\cite{schweiger2024}. Finally, we also prove that, when properly recentered, the field conditioned to be positive converges weakly to the discrete Gaussian free field with random conductances for almost all realizations of the environment, thus establishing a pathwise entropic repulsion phenomenon akin to the one in~\cite{deuschel1999entropic}.
\vspace{\baselineskip}

We now describe the model and our results in more detail. We look at the discrete Gaussian free field on the $d$-dimensional integer lattice for $d \geq 3$, with random conductances, namely the graph with vertices $\bbZ^d$, edges $\bbE_d$ between nearest-neighboring vertices and  a weight $\omega_e$ (the conductance) for each $e \in \bbE_d$, with the conductances being uniformly bounded from above and below, namely $\omega_e \in [\lambda ,\Lambda]$ for some real numbers $0 <\lambda \leq \Lambda$ and all $e\in \bbE_d$. We let $\Omega_{\lambda,\Lambda} = [\lambda,\Lambda]^{\bbE_d}$ be the set of all weights configurations. Given $\omega \in \Omega_{\lambda,\Lambda}$, we denote by $\bbP^\omega$ the law on $\bbR^{\bbZ^d}$ such that
\begin{equation}
    \label{eq:DefGFF}
    \begin{minipage}{0.7\linewidth}
            under $\bbP^\omega$, the canonical field $(\varphi_x)_{x \in \bbZ^d}$ is centered Gaussian with covariances $\bbE^\omega[\varphi_x \varphi_y]= g^\omega(x,y)$ for every $x,y \in \bbZ^d$\,,
    \end{minipage}
\end{equation}
where $g^\omega$ is the Green function of the continuous-time constant-speed simple random walk on $(\bbZ^d, \bbE_d, \omega)$, see~\eqref{eq:GreenFunction}. 
Informally $\bbP^\omega$ can be seen as a Gibbs measure with formal expression

\begin{equation}\label{eq:DefGFFFormal}
        \bbP^\omega(\De\varphi) ``\propto " \exp\bigg\{-\frac{1}{2} \cE^\omega(\varphi,\varphi)\bigg\} \prod_{x \in \bbZ^d} \lambda^1(\De\varphi_x) \,,
\end{equation}
where 
\begin{equation}
    \cE^\omega(\varphi,\varphi) = \sum_{\{x,y\} \in \bbE_d} \omega_{x,y} (\varphi_y - \varphi_x)^2 \,,
\end{equation}
and $\lambda^1$ is the Lebesgue measure on $\bbR$. We will also use that the law $\bbP^\omega$ of the Gaussian free field can be uniquely characterized by the laws of $(\varphi_x)_{x\in\bbZ^d}$ conditioned on the value of the field in their neighborhoods (see~\eqref{eq:DLR_GFF}). 
Consider a probability measure $\bbQ$ on the measurable space of $\Omega_{\lambda,\Lambda}$ endowed with the canonical $\sigma$-algebra of cylinders, and assume $\bbQ$ to be stationary and ergodic with respect to shifts (we refer to Section~\ref{sec:NotationUsefulResults} for precise definitions). To get the discrete Gaussian free field with random conductances we first sample an environment $\omega \in \Omega_{\lambda,\Lambda}$ with law $\bbQ$ and then we define $\bbP^\omega$ as in~\eqref{eq:DefGFF}. We notice that, if $\mu_x^\omega =\sum_{y\sim x}\omega_{x,y}$, then the energy $\cE^\omega(\cdot,\cdot)$ corresponds to the Dirichlet form  of the constant-speed continuous time random walk among random conductances with generator
\begin{equation}\label{eq:generatorCSRW}
    \cL^\omega f(x) = \frac{1}{\mu_x^\omega}\sum_{y \sim x} \omega_{x,y} (f(y) - f(x)) \,,
\end{equation}
on $\ell^2(\bbZ^d,\mu^\omega)$,
which is known as the Random Conductance Model. The last decade has witnessed a number of important advances in the understanding of this model, especially in the context of stochastic homogenization (we refer to~\cite{andres2025homogenizationtheoryrandomwalks,biskup2011recent} for a recent review of the main progresses). The connection between the random conductance model and the Gaussian free field via~\eqref{eq:DefGFF} allows to better understand the latter, as was done for instance in~\cite{andres2025scaling,chiariniRuszel2024stochastic} in the context of the scaling limit of the Gaussian free field in random environment.

In this article we are interested in the study of \emph{quenched} properties for this field, that is, those that hold for $\bbQ$-almost every realization of the random environment. Our attention will focus on the \emph{hard wall event} for the Gaussian free field $\bbP^\omega$ and on the behavior of the field conditioned on it. 
Let $V \subseteq \bbR^d$ be an open and bounded strongly regular set (see Definition~\ref{def:StronglyRegular}). For $N$ positive integer define the discrete blow-up of $V$ as
\begin{equation}
    V_N = (NV) \cap \bbZ^d\,.
\end{equation}
Then, the hard wall event is given by
\begin{equation}
    \cW_N^+ = \{ \varphi_x \geq 0 \text{ for all } x \in V_N \}
\end{equation}
and can be described as the event that the field stays above a hard wall on $V_N$.
The event $\cW_N^+$ is clearly atypical under $\bbP^\omega$,  and the main result of this article concerns quenched large deviation asymptotics for its probability, and the study of the field conditioned on it for $\bbQ$-almost every environment configuration. The asymptotic rate involves
a large-scale effect coming from the homogenization for the random walk among random conductances, and a microscopic effect from the law $\bbQ$ governing the conductances. The former is represented by $\capa^{\rm hom}(V)$, the capacity of $V$ associated with the limiting Brownian motion for the walk with generator \eqref{eq:generatorCSRW}, while the latter is captured by the constant $\overline{g}$ defined as follows:
\begin{equation}
    \overline{g}:= \esssup_{\omega \in \Omega_{\lambda,\Lambda}} g^\omega (0,0) \in (0,\infty)\,,
\end{equation}
which, by stationarity and ergodicity of the environment, is $\bbQ$-a.s.\ equal  to $\sup_{x \in \bbZ^d} g^\omega (x,x) $. 
We also show that the field conditioned on the hard wall event gets delocalized and, at first order, is pushed up to a level of limiting profile $\mathscr{h}^{\mathrm{hom}}_V(\cdot/N)\sqrt{4 \overline{g}\log N}$ where $\mathscr{h}^{\mathrm{hom}}_V$ is the harmonic potential of $V$ for the limiting Brownian motion (see Section~\ref{sec:NotationUsefulResults}). The value of $\overline{g}$ is in general difficult to determine, but it can be computed explicitly in some cases, for example when the conductances are i.i.d.\ and such that $\essinf\omega_{e} = \lambda$ under $\bbQ$. In this case, one has $\overline{g} = g(0,0)/\lambda$, where $g(\cdot,\cdot)$ is the Green function of the simple random walk, see Remark~\ref{rem:Discussion}, 2).

The main result of this article which subsumes this brief discussion is the following theorem.

\begin{theorem}
\label{thm:MainTheorem} Assume that under $\bbQ$ the conductances are i.i.d.\ (or more generally that $\bbQ$ satisfies a mixing condition, see Remark~\ref{rem:Discussion}, 1)). Then, there exists $\Omega_{\mathrm{typ}}$ with full $\bbQ$-measure such that for all $\omega\in \Omega_{\mathrm{typ}}$
\begin{itemize}
    
    \item[i)] \begin{equation}
        \lim_{N\to\infty} \frac{1}{N^{d-2} \log N}\log \bbP^\omega[\cW_N^+] = - 2 \overline{g} \,\capa^{\mathrm{hom}} (V)\,,
    \end{equation}
    \item[ii)]
    \begin{equation}
        \lim_{N\to\infty} \sup_{x\in  \bbZ^d} \bigg|\frac{\bbE^\omega[\varphi_x|\cW^+_N]}{\sqrt{4 \overline{g} \log N}} - \mathscr{h}^{\mathrm{hom}}_V(x/N)\bigg| = 0\,,
    \end{equation}
    \item[iii)] setting $a_N = \bbE^{\omega} [\varphi_0 | \cW^+_N]$, we have
\begin{equation}
    \hat{\bbP}^\omega_N  \to \bbP^\omega 
\end{equation}
    weakly as $N \to \infty$,
where $\hat{\bbP}^\omega_N$ is the law of the field $\varphi - a_N$ under $\bbP^\omega[ \, \cdot \,| \cW^+_N]$. 
\end{itemize}
\end{theorem}
Theorem~\ref{thm:MainTheorem}, i), gives the quenched large deviation asymptotics for the hard wall event, which decays exponentially fast with speed $N^{d-2} \log N$. The appearance of the constant~$\overline{g}$ in the rate becomes natural when looking at the first-order behavior of the maximum of the field over $V_N$. Indeed, one has that $\bbQ$-a.s.\
\begin{equation}\label{eq:MaxGFFVN}
    \lim_{N \to \infty} \frac{\bbE^\omega[\max_{x \in V_N} \varphi_x]}{\sqrt{2d \,\overline{g} \log N}} = 1 \,,
\end{equation}
see Remark~\ref{rem:MaxGFFErgodic} for the details. From the proof of~\eqref{eq:MaxGFFVN}, which is an adaptation of the arguments in~\cite{chiariniCiprianiHazra2016extremes}, it is clear that the maximum is attained at points where the variance of the field is close to its essential supremum $ \overline{g}$, since by ergodicity such points are present in a positive proportion in $V_N$, $\bbQ$-a.s.\ In fact,~\eqref{eq:MaxGFFVN} holds under the sole assumption that $\bbQ$ is stationary and ergodic. More refined convergence results on the recentered maximum in $d\geq 3$ like those in~\cite{chiariniCiprianiHazra2015note,chiariniCiprianiHazra2016extremes} are not known in this generality, and we leave them for future work.

In Theorem~\ref{thm:MainTheorem}, ii), one might wonder why the repulsion  in $V_N$ is at height $\sqrt{4 \overline{g} \log N}$ rather than $\sqrt{2 d\, \overline{g} \log N}$. To understand this, we can think of decomposing the field into a ``local'' part which lives at scale $N^{2/d}$, and carries the short-range correlations, and a ``undertow'' part, which carries the long-range correlations. For the Gaussian free field to stay positive over $V_N$, it is sufficient that enough of the local fields at scale $N^{2/d}$ are pushed away from zero by $\sqrt{4 \overline{g} \log N}$, which is the same level as their expected maximum, then the ``undertow'' part which carries the long range correlations ensures that the whole field remains positive on $V_N$. Note that the constant $\overline{g}$ is quite sensitive to the law $\bbQ$ of the conductances, hence the microscopic effect of the conductances is visible in the repulsion height. In particular, in the special case of i.i.d.\ conductances with $\essinf \omega_e = \lambda$ under $\bbQ$, we have $\overline{g} = g(0,0)/\lambda$. In particular, the field gets pushed up to a height which is strictly larger than in the homogeneous case with constant conductances and with the same scaling limit, see Remark~\ref{rem:Discussion}, 2). This shows that the impurities in the crystal have a non-trivial effect on the entropic repulsion phenomenon as well as the maximum of the unconditioned field, which is somewhat different from what happens in the critical dimension $d=2$, see Remark~\ref{rem:Discussion}, 4). Furthermore, the form of $\overline{g} = g(0,0)/\lambda$  in the i.i.d.\ setting, suggests that in the case of degenerate conductances, the repulsion height is possibly much higher. In fact, going in this direction, for the Gaussian free field on the supercritical percolation cluster, in $d\geq 3$, due to the presence of pipes, we expect that the field under the hard wall event is repelled at height of order $\log N$, see discussion in Remark~\ref{rem:Discussion}, 3), so that the speed of convergence in Theorem~\ref{thm:MainTheorem}, i), is of order $N^{d-2}(\log N)^2$.

Finally, we remark that the quenched weak convergence in Theorem~\ref{thm:MainTheorem}, iii), has a local nature, and establishes that locally the field conditioned on the hard wall event behaves like the unconditioned field when recentered by the diverging amount $a_N$. This is in line with the homogeneous case treated in~\cite{deuschel1999entropic}, but the proof is more involved due to the lack of translation invariance. We can also obtain some more global information, see Remark~\ref{rem:corollaryglobal}.
\vspace{\baselineskip}

We briefly comment on the proofs of our results and give some intuition as well
as some further directions. The most challenging part of this work concerns the upper bound in Theorem~\ref{thm:MainTheorem}, i) and in establishing the convergence in ii).

For the proof of the upper bound in Theorem~\ref{thm:MainTheorem}, i), we follow the coarse graining strategy from~\cite{sznitzman2015disconnection} which was also used in~\cite{chiarini2020entropic, chiarini2021disconnection, nitzschner2018entropic} with adapted notions of good and bad boxes. We found this strategy better suited to handle the lack of translation invariance in the model, as compared to~\cite{bolthausen1995entropic}. We consider a covering of $V_N$ with well separated boxes of side-length $N^{2/d}$, and for each box we split the field into a local part and its harmonic average. We declare boxes good if the minimum of the local field concentrates around its expected value, which is $-\sqrt{4 \overline{g} \log N}$. Thanks to the mixing property of $\bbQ$ and Bernoulli concentration, most of the boxes are good, and under the hard wall event, in each good box there will be sites where the harmonic averages of the field are at height $\sqrt{4 \overline{g} \log N}$. Incidentally, this is the only place where quantitative ergodicity of the environment is used. We then make use of certain Gaussian bounds, which will bring into play the capacity of the union of the good boxes. These form a porous interface that is hard for the random walk to avoid. A crucial ingredient is therefore a quenched \emph{solidification estimate} for such porous interfaces, which allows us to replace their discrete capacity with the discrete capacity of $V_N$. This is close to the solidification estimates developed in \cite{nitzschner2018entropic, nitzschner2017solidification}, even though it is less involved in our case. The proof is completed with the homogenization of the discrete capacity to the Brownian one.

For what concerns Theorem~\ref{thm:MainTheorem}, ii), the proof is more involved. In~\cite{bolthausen1995entropic}, the authors rely on the translation invariance in the model, together with FKG inequality to replace the conditioned expectation $u^\omega_N(\cdot) = \bbE^{\omega}[\varphi_\cdot|\cW_N^+]$ with a an average over a macroscopic box. To overcome the lack of translation invariance, we exploit the superharmonicity of $u^\omega_N$ to obtain a lower bound as a weighted average over the \emph{boundary} of a macroscopic box. We then show that weighted averages of the field over the boundary of macroscopic boxes contained in $V_N$ concentrate around $\sqrt{4 \overline{g} \log N}$ exponentially fast, which allows us to conclude the lower bound. The proof of this fact is reminiscent of the arguments in~\cite{nitzschner2018entropic,chiarini2019entropic,chiarini2021disconnection} where concentration around macroscopic ball averages under disconnection events was shown. However, our setting is slightly more delicate as the concentration is for averages over a lower dimensional set, namely the boundary of a \emph{macroscopic} box. Strikingly, once we have a sharp lower bound for $u^\omega_N$ in $V_N$, we can use it to get an estimate on the harmonicity defect of $u^\omega_N$, which is the key step in obtaining a matching upper bound, by estimating $u^\omega_N$ from above rewriting it as a weighted average on the boundary of a \emph{mesoscopic} ball, which turns out to be enough with the help of a entropy bound. We hope that this new approach can be useful to obtain pathwise results in other contexts, such as when dealing with disconnection events for level-set percolation as in~\cite{chiarini2021disconnection}.

It turns out that the Gaussian free field is closely related to random interlacements, via certain isomorphism theorems, see \cite{sznitman2012isomorphism}. It would be interesting to investigate atypical events of hard wall type for the occupation time field of random interlacement in random environment, which then should be helpful to obtain some fine information on the occupation time field of random walks in random environments. Hopefully, some ideas developed in this article, together with those in \cite{chiarini2020entropic,LiSznitman2014lower,Li2017lower,nitzschner2017solidification}, will be useful.
In a different direction, we remark that the Gaussian free field with random conductances also has implications for more general gradient fields with non-convex potentials, as studied in~\cite{biskup2007phase,biskupSpohn2011scaling}. For such models, the study of extremes and of the hard wall repulsion phenomenon is quite intriguing, and the results of this article may provide some insight in this direction as well.

\vspace{\baselineskip}

We now explain the organization of this article.
In Section~\ref{sec:NotationUsefulResults} we introduce some notation, and certain results about random walk in random environment, heat kernel bounds, potential theory, the Gaussian free field, FKG and relative entropy inequalities. We also present some key results from stochastic homogenization theory.
Section~\ref{sec:PrelimGaussianFunctional}, similarly to~\cite[Section 6]{chiarini2021disconnection}, is devoted to obtaining  certain quenched Gaussian bounds in Corollary \ref{cor:UpperBoundProbZf}. These estimates are crucially used both in Proposition~\ref{prop:UpperBoundHW} and Proposition~\ref{lem:DecayScalProd}.
Section~\ref{sec:ProbHW} is devoted to the proof of the quenched asymptotic large deviation result for the hard wall event as in Theorem~\ref{thm:MainTheorem}, i). This is attained by showing lower and upper bounds in Proposition~\ref{prop:LowerBoundHW} and Proposition~\ref{prop:UpperBoundHW}, respectively. In order to achieve the latter, we also rely on a quenched solidification estimate for sets with a special covering property, which is the content of Corollary~\ref{cor:SolidificationCapacity}.
Section~\ref{sec:ExpConditionedField} is aimed at the study of the expectation of the field conditionally on the hard wall event. The goal is to obtain Theorem~\ref{thm:MainTheorem}, ii). The proof consists in showing the result for the discrete equilibrium potential first and then replacing it with the Brownian one. The lower bound is contained in Proposition~\ref{prop:LBCF}, and  is pivotal for obtaining in Lemma~\ref{lem:defectharmonicity} an estimate for the harmonicity defect of the conditioned expectation. We then exploit such estimate for the proof of the upper bound in Proposition~\ref{prop:UBCF}.
In Section~\ref{sec:ProofEntropicRepulsion} we derive the limiting pathwise behavior for the conditioned field as in Theorem~\ref{thm:MainTheorem}, iii). This is done in two steps: In the first, we prove that the result holds when we recenter the field by its expectation $\bbE^\omega[\varphi_x |\cW^+_N]$, $x \in \bbZ^d$. In the second, we show that the difference between such recentering and the one constantly equal to $\bbE^\omega[\varphi_0|\cW^+_N]$ is $o(1)$ as $N \to \infty$. This is achieved by showing a H\"older regularity estimate for the conditioned field in Lemma~\ref{lem:RegularityCF}. We conclude Section~\ref{sec:ProofEntropicRepulsion} with an extensive discussion. Finally, in Appendix~\ref{secapp:ProofPositivePortionOfHighVariance} we get a generalization of the results assuming just sufficient mixing conditions on the conductances, while Appendix~\ref{secapp:ProofEstimatesExpectationCF} is devoted to the proof of some technical results on the behavior of the field of weighted averages of neighboring entries of the Gaussian free field.
\vspace{\baselineskip}

We use the following convention for constants. $c$, $c'$, $C$, $C'$, $\ldots$ denote positive constants whose value may change from place to place. Numbered constants $c_1$, $c_2$, $C_1$, $C_2$, $\ldots$ remain fixed to the value with which they are defined the first time they appear in the text. All constants may implicitly depend on the dimension and on the parameters $0<\lambda\leq\Lambda$. When we want to underline that a constant $c$ depends on other quantities, we explicitly write them in the notation.

\section{Notation and useful results}
\label{sec:NotationUsefulResults}

This section establishes further notation and classical results on random walks among inhomogeneous conductances, potential theory, and the Gaussian free field with bond disorder. We also recall standard inequalities, such as FKG, Brascamp-Lieb and entropy inequalities, that are required for our analysis. In this article we always assume that $d\geq 3$.
\vspace{\baselineskip}

To begin, we introduce some basic notation. We denote by $\bbN  = \{1,2,3,...\}$ the set of natural numbers. Given two real numbers $s, t$, we write $s \wedge t$ and $s \vee t$ for the minimum and maximum of $s$ and $t$, respectively, and, when $s\geq 0$, the integer part of $s$ is represented by $\lfloor s \rfloor$. For $s=(s_1,\dots,s_d) \in \bbR^d$ we set $\lfloor s \rfloor = (\lfloor s_1 \rfloor,\dots, \lfloor s_d \rfloor)$. We denote the Euclidean and $\ell^\infty$-norms on $\bbR^d$ by $| \cdot |$  and $| \cdot |_\infty$, respectively. We write $B(x,r) = \{y \in \bbZ^d \, : \, |x-y|_\infty \leq r\} \subseteq \bbZ^d$ for the (closed) $\ell^\infty$-ball of radius $r \geq 0$ centered at $x\in\bbZ^d$. We call $x,y \in \bbZ^d$ neighbors, and write $x \sim y$, when $|x - y| = 1$. For a subset $K \subseteq \bbZ^d$, $|K|$ denotes the cardinality of $K$, and we write $K \subset \subset \bbZ^d$ if $|K| < \infty$. We let $K^c = \bbZ^d \setminus K$ be the complement of $K$, $\partial K = \{ y \in \bbZ^d \setminus K \, : \exists \, x \in K \text{ such that } y \sim x \}$ the external boundary of $K$, and $\partial_{\text{in}}K = \{y \in K \, : \exists \, x \in \bbZ^d \setminus K \text{ such that } y \sim x \}$ the internal boundary of $K$. For $K, U \subseteq \bbZ^d$ we write $d_\infty(K,U)$ for the $\ell^\infty$-distance between $K$ and $U$. Given a set $V \subseteq \bbR^d$, we let $V_N = (NV) \cap \bbZ^d$ be the discrete blow-up of $V$, and, for $\varepsilon \in (0,1)$, we define the $\varepsilon$-bulk of $V$ as $V^\varepsilon = \{ x \in V : |x-y| > \varepsilon  , \, y \in V^c \}$, and we set $V^\varepsilon_N = (N V^\varepsilon) \cap \bbZ^d$. For $\varphi, \psi \in \bbR^{\bbZ^d}$ we write $\varphi \leq \psi$ to say that $\varphi_x \leq \psi_x$ for every $x \in \bbZ^d$. We call a function $f : \bbR^{\bbZ^d} \to \bbR$ increasing if $f(\varphi) \leq f(\psi)$ any time $\varphi \leq \psi$, and decreasing if $-f$ is increasing. We also say that $f$ is supported on $K \subseteq \bbZ^d$ if $f(\varphi) = f(\psi)$ whenever $(\varphi_x)_{x \in K} = (\psi_x)_{x \in K}$. For a measure $\mu$ on $\bbZ^d$ and functions $u,\,v:\bbZ^d\to\bbR$ belonging to $\ell^2(\mu)$ we also write $\langle u,v \rangle_{\mu}=\sum_{x\in \bbZ^d} \mu_x u(x)v(x)$, and use the shorthand $\langle  u,v\rangle$ when $\mu$ is constantly equal to one. We use the conventions $\sup_{x \in \varnothing}u(x)=-\infty$ and $\inf_{x \in \varnothing}u(x)=\infty$.

\vspace{\baselineskip}

We consider the integer $d$-dimensional lattice $\bbZ^d$ with set of undirected edges $\bbE_d = \{\{x,y\}: \,x,y\in \bbZ^d,\, x\sim y \}$. For fixed $0<\lambda\leq\Lambda<\infty$, we define the collection of configurations of weights on the edges of the graph $(\bbZ^d, \bbE_d)$ as $\Omega_{\lambda,\Lambda} = [\lambda,\Lambda]^{\bbE_d}$. For any $\omega\in \Omega_{\lambda,\Lambda}$, the conductance on the edge $\{x,y\} \in \bbE_d$ is denoted by $\omega_{x,y} = \omega_{\{x,y\}} = \omega_{y,x}$.
Given $\omega\in\Omega_{\lambda,\Lambda}$, we define
\begin{equation}
\label{eq:reversibMeasure}
\mu^\omega_x = \sum_{y \, : \, y \sim x} \omega_{x,y}\,, \qquad x \in \bbZ^d\,.
\end{equation}
Given a field of conductances $\omega\in \Omega_{\lambda,\Lambda}$ one can define the operator $\cL^\omega$ acting on functions $f : \bbZ^d \to \bbR$ as
\begin{equation}
\label{eq:GeneratorCSRW}
    \cL^\omega f(x) =  \sum_{y \, : \, y \sim x} \frac{\omega_{x,y}}{\mu^\omega_x}(f(y) - f(x))\,,\qquad x\in \bbZ^d\,,
\end{equation}
which is the generator of a (constant-speed) random walk on $(\bbZ^d, \bbE_d, \omega)$. We call $f : \bbZ^d \to \bbR$ a ($\omega$-)harmonic function in a subset $K \subseteq \bbZ^d$ if $\cL^\omega f(x) = 0$ for every $x \in K$. Further, we say that $f$ is ($\omega$-)superharmonic in $K$ if $\cL^\omega f(x) \leq 0$ for every $x \in K$ and that $f$ is ($\omega$-)subharmonic in $K$ if $\cL^\omega f(x) \geq 0$ for every $x \in K$.
\vspace{\baselineskip}

We now introduce the continuous-time, constant-speed random walk on $(\bbZ^d, \bbE_d, \omega)$ and some potential theoretic notions associated with it. For any given $\omega\in \Omega_{\lambda,\Lambda}$ we denote by $P_x^\omega$ the law of the continuous-time Markov chain $(X_t)_{t \geq 0}$ starting from $x \in \bbZ^d$ with generator $\cL^\omega$, that is, the process jumps from $x$ to $y$, $x\sim y$, at rate $\omega_{x,y}/\mu^\omega_x$. The walk is reversible with respect to the measure $\mu^\omega$ defined in~\eqref{eq:reversibMeasure}.
We denote by $E_x^\omega$ the expectation associated with $P_x^\omega$.

For $U \subseteq \bbZ^d$, we define the stopping times (with respect to the canonical filtration generated by $(X_t)_{t\geq 0}$) $H_U = \inf\{ t \geq 0 \, : \, X_t \in U\}$, $\widetilde{H}_U = \inf\{ t > 0 \, : \, X_t \neq X_0 \, , \, X_t \in U\}$, $T_U = \inf \{ t \geq 0 \, : \, X_t \notin U \}$, which are the entrance, hitting and exit times of $U$, respectively. We define $q^\omega_t$ and  $q^{\omega}_{t,U}$ to be the heat kernel and the killed heat kernel upon leaving $U$, respectively, by setting
\begin{align}
q^\omega_t(x,y) & = P^\omega_x[X_t = y] / \mu^\omega_y\,,  &&t \geq 0\,,\, x,y \in \bbZ^d\,, \text{ and} \\
\label{eq:KilledHeatKernel}
q^{\omega}_{t,U}(x,y) & = P^\omega_x[X_t = y, T_U > t] / \mu^\omega_y\,,  &&t \geq 0\,,\, x,y \in \bbZ^d\,.
\end{align}
We let
\begin{align}
\label{eq:GreenFunction}
g^\omega(x,y) & = \int_0^\infty q^\omega_t(x,y)\, \De t = E^\omega_x\Big[ \int_0^\infty \mathbbm{1}_{\{X_t = y \}}\, \De t \Big] / \mu^\omega_y\,, &&x,y \in \bbZ^d\,, \text{ and} \\
\label{eq:KilledGreenFunction}
g_U^\omega(x,y) & = \int_0^\infty q^{\omega}_{t,U}(x,y)\, \De t = E^\omega_x\Big[ \int_0^{T_U} \mathbbm{1}_{\{X_t = y \}}\, \De t \Big] / \mu^\omega_y\,, &&x,y \in \bbZ^d
\end{align}
be the Green function of the random walk and the Green function of the random walk killed upon leaving $U$, respectively. It is known that, under the assumption of elliptic conductances, $q^\omega_t$, $q^\omega_{t,U}$, $g^\omega$ and $g^\omega_U$ are finite and symmetric (for every $t \geq 0$), and $q^\omega_{t,U}(x,y)=g^\omega_U(x,y)=0$ when $x \in U^c$ or $y \in U^c$.

By the strong Markov property, it also holds that
\begin{equation}
\label{eq:DecompositionGreenKilledGreen}
g^\omega(x,y) = g^\omega_U(x,y) + E^\omega_x[T_U < \infty, g^\omega(X_{T_U},y)]\,, \qquad x,y \in \bbZ^d
\end{equation}
(see Proposition 1.6 in~\cite{sznitman2012topics}), and for $U = \{x\}$ by symmetry we get that 
\begin{equation}
\label{eq:harmonicityOfGreen}
g^\omega(y,\cdot) \text{ is $\omega$-harmonic in } \bbZ^d \setminus \{y\}\,, \qquad y \in \bbZ^d\,.
\end{equation}

We also have the following bounds on the Green function:
\begin{equation}
\label{eq:QuenchedGFEstimate}
\frac{c_1}{|x-y|^{d-2} \vee 1} \leq g^\omega(x,y) \leq \frac{c_2}{|x - y|^{d-2} \vee 1}\,, \qquad x, y \in \bbZ^d\,,\, \omega \in \Omega_{\lambda,\Lambda} \,,
\end{equation}
where all the constants are independent of $\omega\in\Omega_{\lambda,\Lambda}$ (see Theorem 6.28 of~\cite{barlow2017random} for uniform bounds on the heat kernel, that imply the bounds on $g^\omega$).

Moreover, for the killed heat kernel we have that for $\vartheta \in (0,1)$ there exists $R_0(\vartheta) \in \bbN$ such that for every $R \geq R_0(\vartheta)$ it holds that for all $\omega \in \Omega_{\lambda,\Lambda}$ and $x_0 \in \bbZ^d$
\begin{equation}
\label{eq:HK_bound}
q_{t,B(x_0,R)}^{\omega}(x_1,x_2) \geq c_3(\vartheta) t^{-d/2}\,,\qquad x_1,x_2 \in B(x_0, (1-\vartheta)R) \text{ and } c_4(\vartheta)R^2 \leq t \leq R^2\,,
\end{equation} 
with $c_4(\vartheta) < 1$ (see Theorem 5.26 in~\cite{barlow2017random}). 

For later use, we state the following H\"older-regularity property for functions $u$ which are $\omega$-harmonic in $B(x_0,2r)$ for some $x_0 \in \bbZ^d$ and $r>0$, namely that there exists two positive constants $c_5, \tau $, only depending on $\lambda$ and $\Lambda$, such that for every $x,y \in B(x_0,r)$
 \begin{equation}
\label{eq:HoelderRegProperty}
 |u(x) - u(y)| \leq c_5 \left( \frac{|x-y|_\infty}{r} \right)^\tau \sup_{z \in B(x_0,2r)}|u(z)|\, ,
\end{equation}
see Proposition 6.2 in~\cite{Delmotte1997}.
\vspace{\baselineskip}

We now present some potential theory associated with the random walk.
For $\omega \in \Omega_{\lambda,\Lambda}$, we introduce the operator $G^\omega$ acting on $f :\bbZ^d \to \bbR$ as
\begin{equation}
G^\omega f (x) = \sum_{y \in \bbZ^d} g^\omega(x,y) f(y)\mu^\omega_y\,, \qquad x \in \bbZ^d
\end{equation}
when the series on the right converges absolutely. For finite measures $\nu$ on $\bbZ^d$ we also write
\begin{equation}
\nu G^\omega (x) = \sum_{y \in \bbZ^d} \nu_yg^\omega(x,y) \,, \qquad x \in \bbZ^d\,.
\end{equation}
We also define the Dirichlet form associated with $-\cL^\omega$ as
\begin{equation}\label{eq:DefinitionDirichletForm}
\cE^\omega(f,f) = \frac{1}{2} \sum_{x\in\bbZ^d} \sum_{y\,:\,x \sim y} \omega_{x,y} \big( f(y) - f(x) \big)^2
\end{equation}
for $f \in \bbR^{\bbZ^d}$, and by polarization (and symmetry of $\omega_{x,y}$) we define it for $f,g \in \bbR^{\bbZ^d}$ as
\begin{equation}
\cE^\omega(f,g) = \frac{1}{2} \sum_{x\in\bbZ^d} \sum_{y\,:\,x \sim y} \omega_{x,y} \big( f(y) - f(x) \big)\big( g(y) - g(x) \big)
\end{equation}
when the resulting series converges absolutely, in which case we have $\cE^\omega(f,g) = \langle -\cL^\omega f,g \rangle_{\mu^\omega}$. We use the shorthand $\cE^\omega(f) = \cE^\omega(f,f)$.
The following relation between $\cE^\omega$ and $G^\omega$ holds:
\begin{equation}\label{eq:RelationDirichletFormGreenFunction}
\cE^\omega (G^\omega f , g) = \langle f , g \rangle_{\mu^\omega}\,, \qquad f,g : \bbZ^d\to \bbR\,,
\end{equation}
whenever the series converge absolutely
(see Proposition 1.3 of~\cite{sznitman2012topics}), hence we can also write $G^\omega = (-\cL^\omega)^{-1}$ as operators on $\ell^2(\mu^\omega)$. 

For $A \subset \subset \bbZ^d$ we define the equilibrium measure of $A$ as
\begin{equation}
\label{eq:DefEqMeasure}
e^\omega_A(x) = P_x^\omega[\widetilde{H}_A = \infty]\mu^\omega_x \mathbbm{1}_A(x)\,, \qquad x \in \bbZ^d\,,
\end{equation}
and the capacity of $A$ (associated with $X$) as its finite total mass,
\begin{equation}
\capa^\omega(A) = \sum_{x \in A}e_A^\omega(x)\,. 
\end{equation}
For a closed $\ell^\infty$-ball $B(x,L)$, where $x \in \bbZ^d$ and $L \geq 1$, the capacity satisfies the classical bounds
\begin{equation}
\label{eq:QuenchedBoxCapacityEstimate}
cL^{d-2} \leq \capa^\omega(B(x,L)) \leq C L^{d-2}\,,
\end{equation}
where the constants are independent of $\omega \in \Omega_{\lambda,\Lambda}$ (see Lemma 7.21 (a) in~\cite{barlow2017random}).
We introduce the equilibrium potential of $A$,
\begin{equation}
\label{eq:EquilibriumPotential}
h^\omega_A(x) = P^\omega_x[H_A < \infty]\,, \qquad x \in \bbZ^d\,,
\end{equation}
and its relation with the equilibrium measure $e^\omega_A$, 
\begin{equation}
\label{eq:EqPotentialAndMeasure}
h^\omega_A(x) =  e^\omega_{A}G^\omega(x)\,, \qquad x \in \bbZ^d
\end{equation}
(see Corollary 7.4 of~\cite{barlow2017random}). We also note that $\capa^\omega(A)=\cE^\omega(h^\omega_A)$.

We further define the equilibrium potential and the capacity of $A \subseteq V \subset \subset \bbZ^d$ for the random walk killed upon leaving $V$, as
\begin{equation}\label{eq:EquilibriumPotentialKilledAndCapa}
    h^\omega_{A,V}(x)  = P^\omega_x [H_A < T_V]\,, \quad x \in \bbZ^d\,, \qquad \text{and} \qquad \capa_V^\omega(A)  = \cE^\omega (h_{A,V}^\omega),
    \end{equation}
respectively.
\vspace{\baselineskip}

We now consider the (discrete) Gaussian free field on the weighted graph $(\bbZ^d, \bbE_d, \omega)$ for a fixed $\omega \in \Omega_{\lambda,\Lambda}$, which we introduced in~\eqref{eq:DefGFF}, and we denote by $\bbE^\omega$ the expectation associated with $\bbP^\omega$.
For $A \subseteq \bbZ^d$, we write $\cF_A=\sigma(\varphi_x \, : \, x \in A)$ for the $\sigma$-algebra generated by the field on $A$.
 The law of the Gaussian free field $\bbP^\omega$ is uniquely characterized as the centered Gaussian field such that for every $x \in \bbZ^d$
\begin{equation}\label{eq:DLR_GFF}
\bbP^\omega\left(\De\phi|\mathcal{F}_{\{x\}^c}\right)(\varphi)\propto\exp\bigg\{-\frac{1}{2}\sum_{y\,:\,y\sim x}\omega_{x,y}\left(\phi_x-\varphi_y\right)^2\bigg\}\,\De\phi_x \prod_{y\neq x} \delta_{\varphi_y}(\De \phi_y)\,,\qquad\bbP^\omega(\De\varphi)\text{-a.s.}
\end{equation}
(see~\cite[Section 13.1]{georgii1988gibbs}). 
We now state a classical spatial Markov property for the Gaussian free field $\varphi$. For $U \subseteq \bbZ^d$ and $\omega \in \Omega_{\lambda,\Lambda}$, we introduce the ($\omega$-)harmonic average $\xi^{\omega,U}$ of $\varphi$ in $U$ and the ($\omega$-)local field $\psi^{\omega,U}$ as
\begin{align}
\label{eq:HarmonicAverageDef}
\xi_x^{\omega,U}& = E^\omega_x[\varphi_{X_{T_U}}, T_U < \infty] = \sum_{y \in \bbZ^d} P^\omega_x[X_{T_U} = y, T_U < \infty]\varphi_y\,, \qquad x \in \bbZ^d\,; \\
\label{eq:LocalFieldDef}
\psi_x^{\omega,U} &= \varphi_x - \xi_x^{\omega,U}\,, \qquad x \in \bbZ^d\,.
\end{align}
By definition, $\varphi_x = \xi_x^{\omega,U} + \psi_x^{\omega,U}$ and for $x \in U^c$ we have $\psi^{\omega,U}_x = 0$ and $\xi^{\omega,U}_x = \varphi_x$. The Markov property is the following:
\begin{equation}
\label{eq:DomainMP}
\begin{minipage}{0.75\linewidth}
    \begin{center}
    $(\psi^{\omega,U}_x)_{x \in \mathbb{Z}^d}$ is independent of $\cF_{U^c}$ (in particular of  $(\xi^{\omega,U}_x)_{x \in \mathbb{Z}^d}$),\\
    and is distributed as a centered Gaussian field with covariance $g^\omega_U(\cdot,\cdot)$\,,
    \end{center}
\end{minipage}
\end{equation}
where $g^\omega_U$ is the Green function of the random walk among conductances $\omega$ killed upon exiting $U$ as in~\eqref{eq:KilledGreenFunction} (see Proposition 2.3 of~\cite{sznitman2012topics}).
\vspace{\baselineskip}

Given $f:\bbZ^d\to \bbR$, we define the translation operator $T_f:\bbR^{\bbZ^d} \to \bbR^{\bbZ^d}$ as
\begin{equation}
\label{eq:TranslationOperator}
T_f \varphi = \varphi + f\,, \qquad \varphi \in \bbR^{\bbZ^d}\,.
\end{equation}
For a probability measure $\nu$ on $\bbR^{\bbZ^d}$ and a map $T$ from $\bbR^{\bbZ^d}$ to $\bbR^{\bbZ^d}$, we write $\nu T^{-1}(\cdot)$ for the push-forward measure $\nu(T^{-1}\cdot).$
If $\varphi$ is a Gaussian free field under $\bbP^\omega$ and $f$ is finitely supported, then the law $\bbP^\omega T^{-1}_f$ can be characterized by its Radon-Nikodym derivative with respect to $\bbP^\omega$ as
\begin{equation}
\label{eq:RadonNikodymShift}
\frac{\De(\bbP^\omega T^{-1}_f)}{\De\bbP^\omega}(\varphi) = \exp \left\{ \cE^\omega(f,\varphi) - \frac{1}{2} \cE^\omega (f,f) \right\}\,, \qquad \varphi \in \bbR^{\bbZ^d}\,.
\end{equation}
This is equivalent to say that
\begin{equation}
    \begin{minipage}{0.8\linewidth}\begin{center}
        $\varphi$ under $\bbP^\omega T^{-1}_f$ has the same law as $\varphi + f$ under $\bbP^\omega$\,,\end{center}
    \end{minipage}
\end{equation}
see (2.4) of~\cite{sznitzman2015disconnection}.
\vspace{\baselineskip}

We now state some key inequalities that will be useful in the sequel. For $\bbP^\omega$ we will make use of three inequalities. The first one is the FKG inequality for positively correlated Gaussian fields which states that (see e.g.~\cite{Pitt82})
\begin{equation}
    \label{eq:FKGproduct}
     \bbE^\omega[ u(\varphi) v(\varphi) ] \geq \bbE^\omega[ u(\varphi)] \bbE^\omega[ v(\varphi) ]\,, \qquad u,v\in L^2(\bbR^{\bbZ^d},\bbP^\omega), \text{ increasing}
\end{equation}
and of course the same inequality still holds true if $u$ and $v$ are both decreasing, while it is reversed if one is increasing and the other is decreasing.

For any subset $U \subseteq \bbZ^d$ we write 
\begin{equation}
    \bbP^{\omega,+}_U [\,\cdot\,] = \bbP^\omega [\, \cdot\, | \varphi_x \geq 0, \, \forall x \in U ]
\end{equation}
for the measure of the Gaussian free field $\varphi$ conditioned on being non-negative on $U$. 
Given two subsets $A, A' \subset \subset \bbZ^d$ with $A' \subseteq A$, the second inequality states that (see e.g. Proposition 5.22 and Definition 5.21 of~\cite{biskup2017extrema})
\begin{equation}\label{eq:FKG2}
    \bbE^{\omega,+}_A[ u(\varphi) ] \geq \bbE^{\omega,+}_{A'}[ u(\varphi) ]\,, \qquad u: \bbR^{\bbZ^d} \to \bbR \text{ increasing}\, ,
\end{equation}
where $\bbE^{\omega,+}_A$ and $\bbE^{\omega,+}_{A'}$ are the expectation associated with $\bbP^{\omega,+}_A$ and $\bbP^{\omega,+}_{A'}$, respectively.

The third inequality states that for any constant field with value $a \geq 0$
\begin{equation}\label{eq:FKG3}
    \bbE^{\omega,+}_A[ u(\varphi) ] \leq \bbE^{\omega,a}[ u(\varphi) | \varphi_x \geq 0, \, \forall x \in A]\,, \qquad u: \bbR^{\bbZ^d} \to \bbR \text{ increasing}\,,
\end{equation}
where $\bbE^{\omega,a}$ is the expectation associated with the law $\bbP^\omega T_a^{-1}$ of the field shifted by height $a$ (see e.g. (B.13) of~\cite{giacomin2001aspects}).

Another key tool we use is the Brascamp-Lieb inequality, which is a concentration inequality. We use the following version for the Gaussian free field $\varphi$ (see~\cite{deuschel1999entropic}, and~\cite{BRASCAMP1976366}): 
\begin{equation}\label{eq:BrascampLieb}
    \bbE^{\omega,+}_A \Big[ u ( \langle f,\varphi \rangle - \bbE^{\omega,+}_A [ \langle f, \varphi \rangle] ) \Big] \leq \bbE^{\omega}[ u( \langle f, \varphi \rangle ) ]
\end{equation}
for every $f :\bbZ^d\to \bbR$ compactly supported, and where $u : \bbR \to \bbR$ is either $u(r) =  |r|^\alpha$ for some $\alpha \geq 1$, or $u(r) = \mathrm{e}^r$.

The last inequalities we recall are relative entropy inequalities to compare the measure of a set under two different probabilities. Let $\bbP$ and $\widetilde{\bbP}$ be two probability measures on the same measurable space $(E,\cE)$ such that $\widetilde{\bbP}$ is absolutely continuous with respect to $\bbP$. The relative entropy of $\widetilde{\bbP}$ with respect to $\bbP$ is defined as
\begin{equation}
\label{eq:DefRelEntropy}
H(\widetilde{\bbP} | \bbP) = \widetilde{\bbE} \Big[ \log \frac{\De \widetilde{\bbP}}{\De \bbP} \Big] = \bbE\Big[\frac{\De \widetilde{\bbP}}{\De \bbP}\log \frac{\De \widetilde{\bbP}}{\De \bbP} \Big] \in [0,\infty]\,,
\end{equation}
where $\widetilde{\bbE}$ and $\bbE$ are the expectation with respect to $\widetilde{\bbP}$ and $\bbP$, respectively. For any $\widetilde{\bbP}$-integrable function $U: E \to \bbR$ with $\bbE[\mathrm{e}^{U}]<\infty$, the following inequality holds:
\begin{equation}\label{eq:RelEntropyInequal1}
\log \bbE[\mathrm{e}^{U}] + H(\widetilde{\bbP} | \bbP) \geq \widetilde{\bbE}[U]\, ,
\end{equation}
see Lemma 3.2.13 of~\cite{deuschel2001large}.
Moreover, for any event $F\in \cE$ with positive $\widetilde{\bbP}$-probability,
\begin{equation}\label{eq:RelEntropyInequal2}
    \log \frac{\bbP[F]}{\widetilde{\bbP}[F]} \geq -\frac{1}{\widetilde{\bbP}[F]} \left[ H(\widetilde{\bbP}|\bbP) + e^{-1} \right]\, ,
\end{equation}
see p. 76 of~\cite{deuschel2001large}.
\vspace{\baselineskip}

We now introduce the random environment and a homogenization result for the random walk. We define a group of environment shifts $\tau_x: \Omega_{\lambda,\Lambda} \to \Omega_{\lambda,\Lambda}$, $x\in \bbZ^d$, as
\begin{equation}
    ( \tau_x \, \omega )_{y,z} = \omega_{y+x, z+x}\,, \qquad y,z \in \bbZ^d\,, \omega \in \Omega_{\lambda,\Lambda}\,.
\end{equation}
We endow $\Omega_{\lambda,\Lambda}$ with the canonical $\sigma$-algebra of cylinders $\cG$ and consider on the resulting measurable space a probability measure $\bbQ$ (the environment measure) which we assume stationary and ergodic with respect to the group of shifts $(\tau_x)_{x \in \bbZ^d}$, namely that $\bbQ [\tau_x (A)] = \bbQ [A]$ for every $A \in \cG$ and $x \in \bbZ^d$, and that every measurable $f : \Omega_{\lambda,\Lambda} \to \bbR$ such that $f \circ \tau_x = f$, $\bbQ$-a.e.\ and for every $x \in \bbZ^d$, is $\bbQ$-a.e.\ constant. This is equivalent to say that all events $A\in\cG$ that are invariant under the shifts satisfy $\bbQ [A] \in \{0,1\}$.

It is know that when $\bbQ$ is a stationary and ergodic probability measure on $\Omega_{\lambda,\Lambda}$, for $\bbQ$-a.e.\ $\omega \in \Omega_{\lambda,\Lambda}$ the diffusively scaled random walk $(n^{-1} X_{tn^2})_{t \geq 0}$ under $P^\omega_0$ converges in law to a Brownian motion $Z=(Z_t)_{t \geq 0}$ which has a deterministic and non-degenerate covariance matrix $a^{\mathrm{hom}} \in \bbR^{d \times d}$. Denoting the law of $Z$ started at $x \in \bbR^d$ by $W_x$, we can define the Dirichlet form $\bfD: W^{1,2}(\bbR^d) \times W^{1,2}(\bbR^d) \to \bbR$ associated with $a^{\mathrm{hom}}$ (where $W^{1,2}(\bbR^d)$ is the Sobolev space of functions in $L^2(\bbR^d)$ that are weakly differentiable with weak derivatives in $L^2(\bbR^d)$) as
\begin{equation}
\bfD (f,f) = \frac{1}{2} \int_{\bbR^d} \nabla f (x) \cdot a^{\mathrm{hom}} \nabla f (x) \De x, \qquad f \in W^{1,2}(\bbR^d).
\end{equation}
Given $A \subseteq \bbR^d$ closed or open bounded, we define the harmonic potential of $A$ associated with $Z$ as
\begin{equation}
\mathscr{h}^{\mathrm{hom}}_A(x) = W_x [(Z_t)_{t \geq 0} \text{ hits } A]\,, \qquad x \in \bbR^d\,,
\end{equation}
and for $A \subseteq B \subseteq \bbR^d$ both bounded and closed or open, we let
\begin{equation}
    \mathscr{h}^{\mathrm{hom}}_{A,B}(x) = W_x [(Z_t)_{t \geq 0} \text{ hits } A \text{ before exiting } B ]\,, \qquad x \in \bbR^d\, ,
\end{equation} 
be the harmonic potential (associated with $Z$) of $A$ with respect to $B$.
Then $\mathscr{h}^{\mathrm{hom}}_{A,B} \in W^{1,2}( \bbR^d )$. Moreover, for every $A \subseteq \bbR^d$ bounded and closed or open $\mathscr{h}^{\mathrm{hom}}_{A}$ belongs to the extended Dirichlet space of $(\bfD, W^{1,2}(\bbR^d))$ (this can be seen combining Theorem 4.3.3, p. 171 of~\cite{FukushimaOshimaTakeda+2010} with Theorem 2.1.5, p. 72 of the same reference, while in Example 1.5.3 in~\cite{FukushimaOshimaTakeda+2010} there is a characterization of the extended Dirichlet space).
Thus we can define the capacity (associated with the Brownian motion $Z$) of $A$, and the capacity of $A$ with respect to $B$, as
\begin{equation}
\label{eq:CapacityContinuumDef}
    \capa^\homo(A) = \bfD(\mathscr{h}^{\mathrm{hom}}_A,\mathscr{h}^{\mathrm{hom}}_A)\,, \qquad \capa^{\mathrm{hom}}_B (A) = \bfD ( \mathscr{h}^{\mathrm{hom}}_{A,B} , \mathscr{h}^{\mathrm{hom}}_{A,B} )\,,
\end{equation}
respectively.
Moreover, it is known that for every compact Lipschitz domain $V \subseteq \bbR^d$, letting $B$ be an open Euclidean ball such that $B \supseteq \overline{V}$, it holds for $\bbQ$-a.e.\ $\omega \in \Omega_{\lambda,\Lambda}$
\begin{equation}\label{eq:hom_capacityKilled}
    \frac{1}{N^{d-2}}\capa^\omega_{B_N}(V_N) \to \capa^{\mathrm{hom}}_{B}(V), \qquad \text{as } N \to \infty\,,
\end{equation}
and for every $V \subseteq \bbR^d$ bounded Lipschitz domain, for $\bbQ$-a.e.\ $\omega \in \Omega_{\lambda,\Lambda}$
\begin{equation}\label{eq:hom_capacity}
    \frac{1}{N^{d-2}}\capa^\omega(V_N) \to \capa^{\mathrm{hom}}(V), \qquad \text{as } N \to \infty\,.
\end{equation}
(see Proposition 5.3 and Corollary 5.4 of~\cite{chiarini2021disconnection}, whose proofs are based on the argument of Corollary 3.4 of~\cite{Neukamm_2017}).
In fact, the convergence in~\eqref{eq:hom_capacityKilled} and~\eqref{eq:hom_capacity} can be extended to every open bounded \emph{strongly regular} set (see Definition 1.8 in \cite{andres2025scaling} for a similar concept).
\begin{equation}\label{def:StronglyRegular}
    \begin{minipage}{0.85\linewidth}
        A set $V \subseteq \bbR^d$ is said to be strongly regular if for all $x \in \partial V$ one has $W_x[H_{\mathring{V}}=0]=1$, where for $A\subseteq \bbR^d$, $H_{A} = \inf\{t > 0\,: Z_t\in A\}$ is the first hitting time of $A$.
    \end{minipage}
\end{equation}
Examples of strongly regular sets are open bounded sets with Lipschitz boundary.

\begin{lemma} Let $V\subseteq\bbR^n$ be an open bounded strongly regular set. Let $B$ be an open ball such that $B \supseteq \overline{V}$. Then,    
\begin{equation}
    \mathscr{h}^{\mathrm{hom}}_{V, B} (x) = \mathscr{h}^{\mathrm{hom}}_{\overline{V}, B}(x)\quad\text{and}\quad\mathscr{h}^{\mathrm{hom}}_V (x) = \mathscr{h}^{\mathrm{hom}}_{\overline{V}}(x), \quad \text{for all } x \in \bbR^d\,.
\end{equation}
In particular,  $\capa_B^{\mathrm{hom}}(V) = \capa_B^{\mathrm{hom}}(\overline{V})$ and $\capa^{\mathrm{hom}}(V) = \capa^{\mathrm{hom}}(\overline{V})$.
\end{lemma}
\begin{proof} Clearly by the strong Markov property one has $\mathscr{h}^{\mathrm{hom}}_V (x) = E^W_x[ W_{Z_{H_{\overline{V}}}}[H_V<\infty]\IND_{H_{\overline{V}}<\infty}]$, $E^W_x$ being the expectation associated with $W_x$. By strong regularity of $V$ and the fact that $V$ is open, for all $y\in \partial V$ one has that $W_{y}[H_V<\infty] = 1$ as $W_{y}[H_{V} = 0] = 1$. Hence, $W_{Z_{H_{\overline{V}}}}[H_V<\infty] =1$ on the event $\{H_{\overline{V}}<\infty\}$ and thus  $\mathscr{h}^{\mathrm{hom}}_V (x) = \mathscr{h}^{\mathrm{hom}}_{\overline{V}}(x)$. The other equality is proven analogously.
\end{proof}

We can now extend the convergence in~\eqref{eq:hom_capacityKilled} and~\eqref{eq:hom_capacity} to open bounded strongly regular sets. We briefly discuss only~\eqref{eq:hom_capacity} as the other is analogous. One can approximate from inside and outside the set $V$ with increasing, respectively decreasing, sequences $(U_n)_{n\in \bbN}$ and $(W_n)_{n\in \bbN}$ of open sets with smooth boundaries.
Then, by the monotonicity of the capacity, one has for all $N \geq 1$
\begin{equation}
    N^{d-2}\capa^\omega((U_n)_N) \leq N^{d-2}\capa^\omega(V_N) \leq N^{d-2}\capa^\omega((W_n)_N)\,,
\end{equation}
which taking the limit $N \to \infty$ and using~\eqref{eq:hom_capacity} for $U_n$ and $W_n$ leads to \begin{equation}
    \capa^{\mathrm{hom}}(U_n) \leq \varliminf_{N \to \infty} \frac{1}{N^{d-2}} \capa^\omega(V_N) \leq \varlimsup_{N \to \infty} \frac{1}{N^{d-2}} \capa^\omega(V_N) \leq \capa^{\mathrm{hom}}(W_n)\,,
\end{equation}
for $\bbQ$-a.e.\ $\omega \in \Omega_{\lambda,\Lambda}$.
Noting that $\capa^{\mathrm{hom}}(U_n) \to \capa^{\mathrm{hom}}(V)$ and $\capa^{\mathrm{hom}}(W_n) \to \capa^{\mathrm{hom}}(\overline{V}) = \capa^{\mathrm{hom}}(V)$ as $n \to \infty$ (the last equality follows from the previous lemma), we conclude that~\eqref{eq:hom_capacity} holds true for $V$.
\vspace{\baselineskip}

In Section \ref{sec:ExpConditionedField}, we will also need the following  homogenization result for the  uniform convergence of the discrete harmonic potential to the continuous one, in order to obtain Theorem~\ref{thm:MainTheorem}, ii), with $\mathscr{h}^{\mathrm{hom}}_V(x/N)$ in place of $h^\omega_{V_N}(x)$.

\begin{lemma}
\label{lem:UnifHarmonicPotential}
Let $V \subseteq \bbR^d$ be a bounded open strongly regular set. Then, for $\bbQ$-a.e.\ $\omega \in \Omega_{\lambda,\Lambda}$,
\begin{equation}
    \lim_{N\to\infty}\sup_{x \in \bbZ^d} \Big| h^\omega_{V_N}(x) - \mathscr{h}^{\mathrm{hom}}_V(x/N) \Big| = 0\,.
\end{equation}
\end{lemma}
\begin{proof} Consider the closed Euclidean ball $\overline{B}_M$ of radius $M$, then 
    \begin{equation}
        \begin{aligned}
            \sup_{x \in \bbZ^d} | h^\omega_{V_N}(x) - \mathscr{h}^{\mathrm{hom}}_V(x/N) | &\leq \sup_{x \in N \overline{B}_{MN}\cap \bbZ^d} | h^\omega_{V_N}(x) - \mathscr{h}^{\mathrm{hom}}_V(x/N) | \\&+ \sup_{x \in N \overline{B}_{MN}^c \cap \bbZ^d} | h^\omega_{V_N}(x) - \mathscr{h}^{\mathrm{hom}}_V(x/N) |\,.
        \end{aligned}
    \end{equation}
    We first handle the second term on the right-hand side and see that by the decay of the Green function (see~\eqref{eq:QuenchedGFEstimate}), together with~\eqref{eq:QuenchedBoxCapacityEstimate} and \eqref{eq:EqPotentialAndMeasure}, (the same estimates hold also in the continuum) it is bounded from above by
    \begin{equation}\label{eq:UnifConvTail}
        \sup_{x \in N \overline{B}_{MN}^c \cap \bbZ^d} [| h^\omega_{V_N}(x)| +| \mathscr{h}^{\mathrm{hom}}_V(x/N) |] \leq \frac{c}{M^{d-2}}\,,
    \end{equation}
    which can be made arbitrarily small by choosing $M$ large. We now focus on the first term on the right-hand side. By~\cite[Appendix A]{chencroydonkumagai2015}, since the conductances are uniformly elliptic and bounded, an arbitrary starting point invariance principle holds. That is, for $\bbQ$-a.e.\ $\omega \in \Omega_{\lambda,\Lambda}$ it holds that for every sequence $(x_N)_{N \geq 1}$ such that $x_N \in (N \overline{B}_{MN}) \cap \bbZ^d$ and $x_N/N \to x \in \overline{B}_M$ the random walk $(N^{-1}X_{t N^2})_{t\geq 0}$ converges in law under $P^\omega_{x_N}$ to the Brownian motion $(Z_t)_{t \geq 0}$ started at $x$ with deterministic non-degenerate covariance matrix $a^\mathrm{hom}$. 

    Since $V$ is strongly regular, it follows that for $\bbQ$-a.e.\ $\omega \in \Omega_{\lambda,\Lambda}$, we have that for every sequence $(x_N)_{N \geq 1}$ such that $x_N \in (N \overline{B}_{MN}) \cap \bbZ^d$ and $x_N/N \to x \in \overline{B}_M$
    \begin{equation}
        h^\omega_{V_N}(x_N) = P^\omega_{x_N}[H_{V_N} < \infty] \to W_x[H_V < \infty] = \mathscr{h}^{\mathrm{hom}}_V(x)\,.
    \end{equation}
    By the arbitrariness of the sequence $(x_N)_{N \geq 1}$, we conclude that for $\bbQ$-a.e.\ $\omega \in \Omega_{\lambda,\Lambda}$ for all $M\in \bbN$
    \begin{equation}\label{eq:UnifConvIntermediate}
        \lim_{N \to \infty} \sup_{x \in N \overline{B}_{MN} \cap \bbZ^d} | h^\omega_{V_N}(x) - \mathscr{h}^{\mathrm{hom}}_V(x/N) | = 0\,,
    \end{equation}
    Indeed, if this were not the case, there would exist $\varepsilon >0$ and a sequence $(x_{N_k})_{k \geq 1}$ such that $x_{N_k} \in (N_k \overline{B}_{MN_k}) \cap \bbZ^d$ and $| h^\omega_{V_{N_k}}(x_{N_k}) - \mathscr{h}^{\mathrm{hom}}_V(x_{N_k}/N_k) | > \varepsilon$ for all $N_k \geq 1$. By compactness, we can assume without loss of generality that $x_{N_k}/N_k \to x \in \overline{B}_M$, which leads to a contradiction, as by strong regularity of $V$, $\mathscr{h}^{\mathrm{hom}}_V(x)$ is continuous (see~\cite[Proposition 3.6, p. 31]{port2012brownian}). By putting together \eqref{eq:UnifConvTail} and \eqref{eq:UnifConvIntermediate}, we conclude the proof taking $M\to \infty$.

\end{proof}

\section{Preliminary quenched asymptotic estimates for Gaussian functionals}
\label{sec:PrelimGaussianFunctional}
In this section we obtain bounds on some Gaussian functionals that hold uniformly over all weights configurations $\omega \in \Omega_{\lambda,\Lambda}$. These estimates yield the key Corollary~\ref{cor:UpperBoundProbZf}, that we later use in Section~\ref{subsec:UpperBoundHW} to obtain the asymptotic upper bound on the probability of the hard wall event $\cW^+_N = \{ \varphi_x \geq 0 \,, \forall x \in V_N \}$ as in Theorem~\ref{thm:MainTheorem}, i).  Corollary~\ref{cor:UpperBoundProbZf} will also be crucial in Section~\ref{sec:ExpConditionedField} for the characterization of the asymptotic law of the field conditioned on $\cW^+_N$, as stated in Theorem~\ref{thm:MainTheorem}, ii).

To define the relevant Gaussian functionals, we now divide $\bbZ^d$ into boxes and use them to decompose $\varphi$ into a harmonic average and a local field. For $L \geq 1$, $ K \geq 100 $ and $z \in \bbZ^d$ we define the family of boxes
\begin{equation}
    \label{eq:BoxDefinition}
        B_z  = z+ [-L,L)^d \cap \bbZ^d
        \subseteq U_z = z+ [-KL,KL)^d \cap \bbZ^d
\end{equation}
and we introduce a subset $\cS$ of $\bbZ^d$ such that
\begin{equation}
    \label{eq:CollectionC}
        \cS \subseteq \bbZ^d \text{ is a collection of vertices with mutual $| \cdot |_\infty$-distance at least }4KL\,,
\end{equation}
and we let 
\begin{equation}
    \label{eq:DefinitionC}
    S = \bigcup_{z \in \cC} B_z \,,
\end{equation}
be a union of disjoint boxes.

Given $\omega \in \Omega_{\lambda,\Lambda}$, for every $z \in \bbZ^d$ we decompose the Gaussian free field into a harmonic average $\xi^{\omega,U_z}$ on $U_z$, and the resulting local field $\psi^{\omega,U_z}$ (see~\eqref{eq:HarmonicAverageDef} and~\eqref{eq:LocalFieldDef} for their definition). For $z \in \bbZ^d$ we use the shorthands $\xi^{\omega,z} = \xi^{\omega,U_z}$ and $\psi^{\omega,z} = \xi^{\omega,U_z}$. By~\eqref{eq:DomainMP}, we have that $(\psi^{\omega,z})_{z \in \cS}$ is a family of independent centered Gaussian fields and it is independent from the collection of Gaussian fields $(\xi^{\omega,z}_x)_{x \in U_z}$, $z \in \cS$.

Following the construction in~\cite{chiarini2021disconnection} (which was inspired by~\cite{sznitzman2015disconnection}), we now define the functionals that capture how the harmonic averages associated with the set of vertices $\cS$ behave. For $\cS$ as in~\eqref{eq:CollectionC}, we introduce the collection of functions that select a point from each box $B_z$, $z\in\cS$. Formally, 
\begin{equation}
    \label{eq:DefinitionF}
    \cF = \{ f \in (\bbZ^d)^{\cS} : f(z) \in B_z \text{ for all } z \in \cS \}\,.
\end{equation}
We also consider the coefficients
\begin{equation}
    \label{eq:DefinitionNu}
    \nu^\omega (z) = \frac{e^\omega_S (B_z)}{\capa^\omega (S)}\,, \qquad z \in \cS\,,
\end{equation}
and, for $x \in \bbZ^d$ and $r >0 $, the distribution $\eta^\omega_{x,r}$ of the first visit of the random walk to the complement of the ball $B(x,r)$,
\begin{equation}
    \label{eq:DefinitionEta}
    \eta^\omega_{x,r}(y) = P^\omega_x [ X_{T_{B(x,r)}} = y ] \,, \qquad y \in \bbZ^d \,.
\end{equation}
The Gaussian functionals associated with $f \in \cF$ are the convex combinations
\begin{equation}
    \label{eq:ZfDefinition}
    Z^\omega_f = \sum_{z \in \cC} \nu^{\omega}(z) \xi^{\omega,z}_{f(z)}
\end{equation}
and, for $x \in \bbZ^d$ and $r>0$,
\begin{equation}\label{eq:ZfbetabDefinition}
    Z^\omega_{f,\beta,b} = b Z^\omega_f - \beta \langle \varphi , \eta^\omega_{x,r} \rangle \,, \qquad b>0\,, \beta \geq 0 \,,
\end{equation}
whose dependence on $x$ and $r$ is suppressed in the notation.
Finally, we define the variable
\begin{equation}
    \label{eq:ZDefinition}
    Z^\omega = \sup_{f \in \cF} Z^\omega_f\,.
\end{equation}
The next theorem provides the upper bounds on the variance of $Z^\omega_{f,\beta,b}$ and the expectation of $Z^\omega$ which are uniform in the choice of the environment and that we need for the proof of Corollary~\ref{cor:UpperBoundProbZf}. It is similar in spirit to Theorem 6.1 in~\cite{chiarini2021disconnection} and Theorem 4.2 in~\cite{sznitzman2015disconnection}. 
\begin{theorem}
    \label{thm:ZfbetabBounds}
    Recall $L \geq 1$, $K \geq 100$ are integers. Then, for all $r=r(L)>0$ such that $r(L)\to \infty$ as $L \to \infty$, all $b=b(L)>0$ bounded sequences, and for every $\beta \geq 0$,
    \begin{equation}\label{eq:VarZfbetabBound}
        \limsup_{L \to \infty } \sup_{\omega \in \Omega_{\lambda,\Lambda}} \sup_{\cS} \sup_{f \in \cF} \sup_{x \in \bbZ^d} \Big\{ r^{d-2}\bbV\mathrm{ar}^\omega[Z^\omega_{f,\beta,b}] - \frac{r^{d-2}}{\capa^\omega (S)}\alpha^\omega_{K,L,\beta,b}\Big\} \leq c_6 \beta^2  \,,
    \end{equation}
    where $\bbV\mathrm{ar}^\omega$ is the variance associated with $\bbP^\omega$, and
    \begin{equation}\label{eq:DefAlpha}
        \alpha_{K,L,\beta,b}^\omega = b^2 -2\beta b \langle h^\omega_S ,\eta^\omega_{x,r} \rangle + \cR_{K,L,\beta,b} \,,
    \end{equation}
    with
    \begin{align}
        \label{eq:REstimate}
        &0\leq \cR_{K,L,\beta,b} \leq b^2 U(K,L) + 2b\beta \Big[U(K,L) + \langle \IND_{\cU} ,\eta^\omega_{x,r} \rangle \Big] \,,\\
        \label{eq:UDecay}
        &\cU = \bigcup_{z\in \cS} U_z\,,\qquad\lim_{K \to \infty} \limsup_{L \to \infty} U(K,L) = 0 \,,\qquad U(K,L) \geq 0 \,.
    \end{align}
    Moreover,
    \begin{equation}\label{eq:ExpZBound}
        \sup_{L} \sup_{\omega \in \Omega_{\lambda,\Lambda}} \sup_{\cS} |\bbE^\omega[Z^\omega]| \left( \frac{\capa^\omega (S)}{|\cS|} \right)^{1/2} \leq \frac{c_7}{K^{c_8}} \,.
    \end{equation}
\end{theorem}
\begin{proof}
    The result~\eqref{eq:ExpZBound} on the expectation of $Z^\omega$ can be proven as (6.14) of Theorem 6.1 in~\cite{chiarini2021disconnection}, which is a straightforward adaptation of the proof of Theorem 4.2 in~\cite{sznitzman2015disconnection}, that exploits Theorem 1.3.3 of~\cite{adler2007random}.

    We now start the proof of~\eqref{eq:VarZfbetabBound} and~\eqref{eq:DefAlpha}.
    For any $\omega \in \Omega_{\lambda,\Lambda}$ we define the sequences $\gamma^\omega$ and $\widetilde{\gamma}^\omega$ as
    \begin{equation}
        \begin{split}
            \gamma^\omega(K,L) & = \sup_{\substack{z,z' \in \bbZ^d \\ |z-z'|_\infty \geq KL}} \, \sup_{\substack{x,y \in B_z \\ x', y' \in B_{z'}} }  \frac{g^\omega(y,y')}{g^\omega(x,x')} \,(\geq 1) \,, \\
            \widetilde{\gamma}^\omega(K,L) & = \inf_{\substack{z,z' \in \bbZ^d \\ |z-z'|_\infty \geq KL}} \, \inf_{\substack{x,y \in B_z \\ x',y' \in B_{z'}} }  \frac{g^\omega(y,y')}{g^\omega(x,x')} \,(\leq 1)\,.
        \end{split}
    \end{equation}
    By the arguments from (6.23) to (6.28) in the proof of Theorem 6.1 in~\cite{chiarini2021disconnection} it can be proven that 
    \begin{align}
        \label{eq:ClaimGamma}
        \lim_{K \rightarrow \infty} \limsup_{L \rightarrow \infty} \sup_{\omega \in \Omega_{\lambda,\Lambda}} \gamma^\omega(K,L) & = 1 \,, \text{ and }\\
        \label{eq:ClaimTildeGamma}
        \lim_{K \rightarrow \infty} \liminf_{L \rightarrow \infty} \inf_{\omega \in \Omega_{\lambda,\Lambda}} \widetilde{\gamma}^\omega(K,L) &= 1 \,.
    \end{align}
    For any $\cS$ as in~\eqref{eq:CollectionC}, $f \in \cF$, $x \in \bbZ^d$, $r>0$, $b>0$ and $\beta \geq 0$, the field $(Z^\omega_{f,\beta,b})_{f \in \cF}$ is centered Gaussian, hence
    \begin{equation}
        \label{eq:DecompositionVarZfbetab}
        \bbV\mathrm{ar}^\omega [Z^\omega_{f,\beta,b}] = b^2 \bbV\mathrm{ar}^\omega [Z^\omega_f] + \beta^2 \bbV\mathrm{ar}^\omega [\langle \varphi,\eta^\omega_{x,r} \rangle] -2b\beta \bbE^\omega [Z^\omega_f \langle \varphi,\eta^\omega_{x,r} \rangle] \,.
    \end{equation}
    For the first addendum, by the arguments leading up to (6.31) in the proof of Theorem 6.1 in~\cite{chiarini2021disconnection} we have that uniformly over $f \in \cF$
    \begin{equation}
        \label{eq:VarZfBound}
        \bbV\mathrm{ar}^\omega[Z^\omega_f] \leq \capa^\omega (S)^{-1} [ cK^{2-d} + \gamma^\omega(K,L) ] \,.
    \end{equation}
    For the second addendum we have that
    \begin{equation}
        \begin{aligned}
            \label{eq:VarDotProdBound}
            \bbV\mathrm{ar}^\omega [\langle \varphi, \eta^\omega_{x,r} \rangle] & \stackrel{\eqref{eq:DefinitionEta}}{=} \sum_{y \in \partial B(x,r)} \eta^\omega_{x,r}(y) E^\omega_x[ T_{ B(x,r)} < \infty, g^\omega(y,X_{T_{ B(x,r)}}) ] \\
            & \, \,= \sum_{y \in \partial B(x,r)} \eta^\omega_{x,r}(y) g^\omega (y,x) \leq c_6 \, r^{2-d}
        \end{aligned}
    \end{equation}
    with $c_6$ independent of $\omega\in\Omega_{\lambda,\Lambda}$ and $x \in \bbZ^d$, having used in the second step~\eqref{eq:DecompositionGreenKilledGreen} and the fact that $y \in \partial B(x,r)$, while in the third step we used the quenched bounds~\eqref{eq:QuenchedGFEstimate} on $g^\omega$ and the fact that $\eta^\omega_{x,r}$ is a probability measure.
    We now estimate the third addendum in the right-hand side of~\eqref{eq:DecompositionVarZfbetab}.
    Defining the set
    \begin{equation}
        \label{eq:DefinitionU}
        \cU = \bigcup_{z \in \cS} U_z \,,
    \end{equation}
    it holds
    \begin{equation}
        \label{eq:ExpZfScalProdEstimate1}
        \bbE^\omega [Z^\omega_f \langle \varphi,\eta^\omega_{x,r} \rangle] \geq \sum_{z \in \cS} \sum_{\substack{y \in \partial B(x,r)\\ y \notin \cU}} \nu^\omega(z) \eta^\omega_{x,r}(y) \bbE^\omega [\xi^{\omega,z}_{f(z)} \varphi_y ] \,.
    \end{equation}
    By the random walk representation of the harmonic average in~\eqref{eq:HarmonicAverageDef}, we have that for $z \in \cC$ and $y \in \partial B(x,r)$
    \begin{equation}
        \bbE^\omega [\xi^{\omega,z}_{f(z)} \varphi_y ] = \sum_{\substack{y' \in \partial U_z}} P^\omega_{f(z)} [X_{T_{U_z}} = y'] g^\omega (y',y) = E^\omega_{f(z)} [ T_{U_z} < \infty , g^\omega (X_{T_{U_z}},y) ] = g^\omega (f(z),y) \,,
    \end{equation}
    having used in the last step~\eqref{eq:DecompositionGreenKilledGreen} again and the fact that $y \in \cU^c$.
    Then by the definition~\eqref{eq:DefinitionNu} of $\nu^\omega$ the inequality~\eqref{eq:ExpZfScalProdEstimate1} becomes
    \begin{equation}
        \label{eq:ExpZfScalProdEstimate2}
        \begin{aligned}
            \bbE^\omega [Z^\omega_f \langle \varphi,\eta^\omega_{x,r} \rangle] & \geq \frac{1}{\capa^\omega (S)} \sum_{\substack{y \in \partial B(x,r) \\ y \notin \cU}} \eta^\omega_{x,r}(y) \sum_{z \in \cS} e^\omega_S (B_z) g^\omega (f(z),y) \\
            & \!\! \stackrel{\eqref{eq:EqPotentialAndMeasure}}{\geq} \frac{\widetilde{\gamma}^\omega(K,L)}{\capa^\omega(S)} ( \langle \eta^\omega_{x,r}, h^\omega_S \rangle - \langle \IND_{\cU} , \eta^\omega_{x,r} \rangle ) \,.
        \end{aligned}
    \end{equation}
    Defining
    \begin{equation}
        U(K,L) = \Big( cK^{2-d} + \sup_{\omega \in \Omega_{\lambda,\Lambda}} \gamma^\omega(K,L) -1 \Big) \vee \Big( 1-\inf_{\omega \in \Omega_{\lambda,\Lambda}} \widetilde{\gamma}^\omega(K,L) \Big) \,,
    \end{equation}
    the result~\eqref{eq:VarZfbetabBound} is implied by the decomposition~\eqref{eq:DecompositionVarZfbetab} and the three estimates~\eqref{eq:VarZfBound},~\eqref{eq:VarDotProdBound} and~\eqref{eq:ExpZfScalProdEstimate2}. 
    \end{proof}
    
As a consequence of the controls on the variance of $Z^\omega_{f,\beta,b}$ and on the expectation of $Z^\omega$, we get the following corollary which states a quenched exponential bound on the tail probabilities of $\sup_{f \in \cF}Z^\omega_{f,\beta,b}$. Its proof is immediate by applying in order~\eqref{eq:ExpZBound}, Borell-TIS inequality and~\eqref{eq:VarZfbetabBound}. We use this corollary with $\beta = 0$ in Section~\ref{subsec:UpperBoundHW} for the proof of the upper bound for Theorem~\ref{thm:MainTheorem}, i), and with $\beta >0$ in Section~\ref{sec:ExpConditionedField} for the proof of ii).
\begin{corollary}
    \label{cor:UpperBoundProbZf}
    Recall $L \geq 1$, $K \geq 100$. Let $r=r(L)>0$ such that $r(L)\to \infty$ as $L \to \infty$, $b=b(L)>0$ bounded sequence, and $\beta \geq 0$. Then, for every $\Delta, \Theta > 0$ it holds that
    \begin{equation}
        \begin{aligned}
        \label{eq:UpperBoundZf}
        & \varlimsup_{L \to \infty} \sup_{\omega\in\Omega_{\lambda,\Lambda}} \sup_{\cS} \sup_{x\in\bbZ^d} \Bigg\{ \bbP^\omega \Big[ \sup_{f \in \cF} Z^\omega_{f,\beta,b} \geq b \Delta -\beta \Theta \Big] \\
        & \quad - 2   \exp \Bigg[ -\frac{1}{2} \bigg( b \bigg( \Delta - \frac{c_7}{K^{c_8}}\sqrt{\frac{|\cS|}{\capa^\omega(S)}} \bigg) -\beta \Theta \bigg)_+^2 \frac{\capa^\omega(S)}{c_6\beta^2 r^{2-d}\capa^\omega(S)+\alpha^\omega_{K,L,\beta,b}} \Bigg] \Bigg\} \leq 0 \,,
        \end{aligned}
    \end{equation}
    where $\alpha^\omega_{K,L,\beta,b}$ is defined in Theorem~\ref{thm:ZfbetabBounds}. When $\beta = 0$, and $b = 1$ this bound simplifies to
      \begin{equation}
        \label{eq:UpperBoundZf_beta0}
        \begin{aligned}
            & \varlimsup_{L \to \infty} \sup_{\omega\in\Omega_{\lambda,\Lambda}} \sup_{\cS} \sup_{x\in\bbZ^d} \Bigg\{ \bbP^\omega \Big[ \sup_{f \in \cF} Z^\omega_{f} \geq \Delta \Big] \\
            & \quad - 2 \exp \Bigg[ -\frac{1}{2} \bigg( \Delta - \frac{c_7}{K^{c_8}}\sqrt{\frac{|\cS|}{\capa^\omega(S)}} \bigg)_+^2 \frac{\capa^\omega(S)}{1+U(K,L)} \Bigg] \Bigg\} \leq 0\,,
        \end{aligned}
    \end{equation}
    with $U(K,L)$ as in Theorem~\ref{thm:ZfbetabBounds}.
\end{corollary}

\section{Quenched decay of the probability of the hard wall event}\label{sec:ProbHW}
\subsection{Lower bound}
In this section we prove the quenched asymptotic lower bound, as $N \to \infty$, for the rate of decay for the probability of the hard wall event $\cW_N^+ = \{ \varphi_x \geq 0, \, \forall x \in V_N \}$, and get the optimal constant as in Theorem~\ref{thm:MainTheorem}, i).

We follow a classical change of measure approach, namely we tilt the Gaussian free field measure in such a way that $\cW_N^+$ becomes typical under the resulting new measure. Then we use the relative entropy inequality~\eqref{eq:RelEntropyInequal2} to compare the probability of $\cW_N^+$ under the original and the tilted measure. This yields an estimate in which $N^{2-d}\capa^\omega(V_N)$ appears, that converges $\bbQ$-a.s.\ to the capacity $\capa^{\mathrm{hom}} (V)$ of the continuum domain $V$ by the homogenization result~\eqref{eq:hom_capacity}.

\begin{prop}
    \label{prop:LowerBoundHW}
    For $V\subseteq \bbR^d$ an open and bounded strongly regular
set and for $\bbQ$-a.e.\ $\omega \in \Omega_{\lambda,\Lambda}$ it holds that
    \begin{equation}
        \liminf_{N \to \infty} \frac{1}{N^{d-2}\log N} \log \bbP^\omega[\cW_N^+] \geq - 2 \overline{g}\, \capa^{\mathrm{hom}}(V)\,.
    \end{equation}
  \end{prop}
  \begin{proof} We follow the idea in~\cite[Section 3.5]{giacomin2001aspects} that dealt with the homogeneous case.
    Let $\alpha_N = \sqrt{4\overline{g}\log N}$ and $B \subseteq \bbR^d$ an open ball of radius $R$ such that $B \supseteq V$. We define the probability measure of the field shifted by a height $f_N = \alpha_N h^\omega_{V_N,B_N}$ as $\bbP_2 = \bbP^\omega T_{f_N}^{-1}$ (recall the definition~\eqref{eq:EquilibriumPotentialKilledAndCapa} of the equilibrium potential $h^\omega_{V_N,B_N}$ for the killed random walk). We also define a third probability measure $\bbP_3 [\,\cdot\,] = \bbP_2 [\,\cdot\, | \cW_N^+]$, for which the hard wall event is typical.
  
    We want to use the relative entropy inequality~\eqref{eq:RelEntropyInequal2} in order to compare the measure of the hard wall event under $\bbP^\omega$ and $\bbP_3$.    
    Since $\bbP_3$-a.s.\ it holds that $\De \bbP_3/\De \bbP^\omega = (\De \bbP_3 / \De \bbP_2)(\De \bbP_2 / \De \bbP^\omega )$, we have that
    \begin{equation}\label{eq:EntropyDecomposition}
        H (\bbP_3 | \bbP^\omega) = H(\bbP_3 | \bbP_2 ) + \bbE_3 \Big[ \log \frac{\De \bbP_2}{\De \bbP^\omega} \Big]
    \end{equation}
    where $\bbE_3$ is the expectation associated with $\bbP_3$.
    We now show that the first term on the right-hand side is negligible as $N \to \infty$, and that the cost of the change of measure in the limit is therefore given just by the second term. By the FKG inequality~\eqref{eq:FKGproduct} we have that
    \begin{equation}\label{eq:LBHWFirstTerm}
        \begin{aligned}
            H(\bbP_3 | \bbP_2 ) &= - \log \bbP^\omega [ T_{f_N}^{-1} \cW^+_N ] = -\log \bbP^\omega [ \varphi_x \geq -\alpha_N,\, \forall x \in V_N ] \\
            & \leq - \sum_{x \in V_N} \log \bbP^\omega [ \varphi_x \geq -\alpha_N ] \leq - |V_N| \log \phi (\alpha_N/\sqrt{\overline{g}})
        \end{aligned}
    \end{equation}
    where $\phi$ is the standard normal cumulative distribution function. Using that \\ $\phi(x) \geq 1 - \tfrac{1}{\sqrt{2\pi} x} \exp(-x^2/2)$ for $x > 0$, we get that for all $N$ sufficiently large
    \begin{equation}
            H(\bbP_3 | \bbP_2 ) \leq - c\, N^{d} \log \bigg( 1- \frac{c'}{\alpha_N/\sqrt{\,\overline{g}}} \exp \left(-\alpha_N^2/2\overline{g}\right)\bigg) \leq C\, \frac{N^{d-2}}{\sqrt{4 \log N }}\,.
    \end{equation}
    Hence
    \begin{equation}
        \lim_{N\to\infty} \frac{1}{N^{d-2}}H(\bbP_3 | \bbP_2 ) = 0
    \end{equation}
    and we can therefore neglect the first term in the right-hand side of~\eqref{eq:EntropyDecomposition}.
    For the second term, we have that
    \begin{equation}\label{eq:LBHWSecondTerm}
        \begin{aligned}
          \bbE_3 \Big[ \log \Big(\frac{\De \bbP_2}{\De \bbP^\omega} (\varphi)\Big) \Big] & = \bbE^{\omega} \Big[ \log \Big( \frac{\De \bbP_2}{\De \bbP^\omega} ( T_{f_N} \varphi)\Big) \Big| T_{f_N}^{-1} \cW_N^+ \Big]\\
          & = \alpha_N \, \bbE^\omega \Big[ \cE^\omega(h^\omega_{V_N,B_N} , \varphi) \Big| T^{-1}_{f_N} \cW_N^+ \Big] + \frac{\alpha_N^2}{2}\, \capa^{\omega}_{B_N}(V_N)
        \end{aligned}
    \end{equation}
    where in the second step we used the characterizations of the shifted measure~\eqref{eq:RadonNikodymShift}.
    We now want to prove that the first term in the right-hand side of~\eqref{eq:LBHWSecondTerm} is negligible.
    Note that $\cE^\omega (h^\omega_{V_N,B_N} , \varphi) = \langle -\cL^\omega h^\omega_{V_N,B_N} , \varphi \rangle_{\mu^\omega}$ is a centered Gaussian with variance $\capa^\omega_{B_N}(V_N)$. Applying the relative entropy inequality~\eqref{eq:RelEntropyInequal1} with $u (\varphi) = \cE^\omega (h^\omega_{V_N,B_N} , \varphi)$ to relate $\bbP^\omega$ and $\bbP^\omega[\,\cdot\, | T_{f_N}^{-1} \cW^+_N]$, we get
    \begin{equation}
        \Big|\alpha_N\, \bbE^\omega [ \left.\cE^\omega(h^\omega_{V_N,B_N} ,\varphi)\right| T_{f_N}^{-1}\cW^+_N ]\Big| \leq \alpha_N \left\{ \capa^{\omega}_{B_N}(V_N) / 2 - \log \bbP^\omega [ T_{f_N}^{-1} \cW^+_N ] \right\}\,.
    \end{equation}
    When dividing by $N^{d-2} \log{N}$, in the limit for $N \to \infty$ the first addendum in the right-hand side goes to $0$, since $\limsup_{N\to\infty} N^{2-d}\capa^{\omega}_{B_N}(V_N) < \infty$ by~\eqref{eq:QuenchedBoxCapacityEstimate}, and the second addendum also goes to $0$ since it is equal to $H(\bbP_3 | \bbP_2)$ by~\eqref{eq:LBHWFirstTerm}.
    Plugging this into~\eqref{eq:LBHWSecondTerm} yields
    \begin{equation}
    \lim_{N\to\infty}\frac{1}{N^{d-2}\log N} \, \bbE_3 \Big[ \log \Big(\frac{\De \bbP_2}{\De \bbP^\omega} (\varphi)\Big) \Big] = 2 \overline{g}\, \capa^{\mathrm{hom}}_B(V)\,.
    \end{equation}
    Coming back to~\eqref{eq:EntropyDecomposition}, we have shown that
    \begin{equation}
        \lim_{N\to\infty}\frac{1}{N^{d-2}\log N} H(\bbP_3 | \bbP^\omega) =
        2 \overline{g}\, \capa^{\mathrm{hom}}_B(V)\,,
    \end{equation}
    thus, the claim follows from~\eqref{eq:RelEntropyInequal2} since $\bbP_3[\cW^+_N]=1$ and upon taking $R \to \infty$.
  \end{proof}

\subsection{Upper bound}
\label{subsec:UpperBoundHW}
In this section we prove the quenched asymptotic upper bound, as $N \to \infty$, for the rate of decay for the probability of the hard wall event $\cW^+_N = \{ \varphi_x \geq 0, \, \forall x \in V_N \}$, and get the optimal constant as in Theorem~\ref{thm:MainTheorem}, i).

We consider the division of $V_N$ into boxes (whose size depends on $N$) and the associated decomposition of $\varphi$ into harmonic averages and local fields, as in Section~\ref{sec:PrelimGaussianFunctional}. We show that, under $\cW^+_N$, for $\bbQ$-a.e.\ environment $\omega \in \Omega_{\lambda,\Lambda}$, with high $\bbP^\omega$-probability, on the set $V_N$ there is a large number of ``good" boxes where there is at least one vertex in which the harmonic average is above the threshold $\sqrt{4\overline{g}\log N}$ (note that by Theorem~\ref{thm:MainTheorem}, ii), this is the first-order asymptotic behavior of $\bbE^\omega[\varphi_x|\cW^+_N]$ when $x \in V_N$). We then encode this event into the supremum of the Gaussian functionals we defined in Section~\ref{sec:PrelimGaussianFunctional}, whose tail probabilities we estimate via Corollary~\ref{cor:UpperBoundProbZf}. This yields a bound in terms of the capacity, associated with the random walk, of the union of good boxes. As $N \to \infty$, we show that the union of good boxes has solidification properties (akin to those studied in~\cite{nitzschner2018entropic},~\cite{chiarini2021disconnection} and~\cite{sznitzman2015disconnection}) that let us control such capacity by $\capa^\omega(V_N)$, hence capturing the exact decay rate of $\bbP^\omega[\cW^+_N]$ through the homogenization result~\eqref{eq:hom_capacity}.

We divide the proof into three parts. In Section~\ref{subsec:MaxLocalField} we characterize the asymptotic behavior of the expectation of the maximum of the local field on mesoscopic boxes in $V_N$ as their size grows to infinity. In Section~\ref{subsec:Solidification} we prove the solidification property of the union of good boxes. Finally, in Section~\ref{subsec:ProofUBHW} we exploit Corollary~\ref{cor:UpperBoundProbZf} and the results of Section~\ref{subsec:MaxLocalField} and Section~\ref{subsec:Solidification} to complete the proof for the upper bound.

\subsubsection{Quenched maximum of the local field on boxes}
\label{subsec:MaxLocalField}

From now on we consider $L = L(N)$ to be a growing sequence of integers. We set
\begin{equation}
    \label{eq:DefinitionL}
    L = \left\lceil \frac{N^{2/d}}{(\log N)^{1/2d}} \right\rceil \,, \qquad N \in \bbN \,,
\end{equation}
and consider, for $K \geq 100d$ integer,
\begin{equation}
    \label{eq:DefinitionCN}
    \cC_N = ((4KL) \bbZ^d) \cap V_N \,,
\end{equation}
which satisfies~\eqref{eq:CollectionC}.
Note that the cardinality of $\cC_N$ is of order $(N/ (KL))^d$, hence
\begin{equation}
    \label{eq:CardinalityCN}
    c\,K^{-d} N^{d-2}\sqrt{\log N} \leq |\cC_N| \leq C\, K^{-d} N^{d-2}\sqrt{\log N} \,.
\end{equation}

In the next proposition we show that, for $\bbQ$-a.e.\ $\omega \in \Omega_{\lambda,\Lambda}$, on every box $B_z$, $z \in \cC_N$, there is a (uniformly in $z\in\cC_N$) positive proportion of points where the local field $\psi^{\omega,z}$ has high variance. To have such uniformity, quantitative ergodicity for the environment is needed. Although in the Proposition below conductances are assumed to be i.i.d.\ to ease the exposition,  the result can be extended to environments with enough mixing  (see Remark~\ref{rem:Discussion}, 1) below and Proposition~\ref{prop:AlphaMixPositiveOfPortionHighVariance} for a precise statement in the case of polynomially mixing environments).

\begin{prop}
    \label{prop:PositivePortionOfHighVariance}
    Let $\varepsilon > 0$.
    Then, there exists a constant $q>0$  such that for $\bbQ$-a.e.\ $\omega \in \Omega_{\lambda,\Lambda}$
    \begin{equation}
        \label{eq:PositivePortionOfHighVariance}
        \liminf_{N \to \infty} \inf_{z \in \cC_N}\frac{1}{|B_z|}\sum_{y \in B_z} \IND_{\{ g_{U_z}^\omega(y,y) \geq \overline{g} - \varepsilon \}} \geq q \,.
    \end{equation}
\end{prop}
\begin{proof}
    We first find, by means of~\eqref{eq:DecompositionGreenKilledGreen} and the quenched bounds~\eqref{eq:QuenchedGFEstimate}, $R = R(\varepsilon)$ sufficiently large such that $g^\omega(y,y) \leq g^\omega_{B(y,R)}(y,y) + \varepsilon/2$ for all $y \in \bbZ^d$ and $\omega \in \Omega_{\lambda,\Lambda}$.
    By stationarity of $\bbQ$, we have 
    \begin{equation}\label{eq:parameter_q}
        \bbQ[g_{B(y,R)}^\omega(y,y)\geq \overline{g} - \varepsilon] \geq \bbQ[g^\omega(0,0)\geq \overline{g} - \varepsilon/2] =: 2q > 0 \,,
    \end{equation}
    where the last inequality follows from  $\overline{g} = \esssup_{\omega \in \Omega_{\lambda,\Lambda}} g^\omega(0,0)$.
    We aim to show that $\bbQ$-a.s.
\begin{equation}
        \label{eq:PositivePortionOfHighVariance2}
        \liminf_{N \to \infty} \inf_{z \in \cC_N}\frac{1}{|B_z|}\sum_{y \in B_z} \IND_{\{ g_{B(y,R)}^\omega(y,y) \geq \overline{g} - \varepsilon \}} \geq q.
    \end{equation}
    from which the claim follows as $g_{U_z}^\omega(y,y) \geq g_{B(y,R)}^\omega(y,y)$ for $N$ large enough.
    To that end, we introduce the event
    \begin{equation}
        \label{eq:DefinitionJz}
        J^\omega_N(z) = \Big\{ \frac{1}{|B_z|}\sum_{y \in B_z} \IND_{\{ g^\omega_{B(y,R)}(y,y) \geq \overline{g} - \varepsilon \}} < q \Big\} \,,
    \end{equation}
whose $\bbQ$-probability is independent of $z$, and provide a summable upper bound on $\bbQ[\cup_{z \in \cC_N} J^\omega_N(z)]$.
    By Borel-Cantelli lemma, this will imply that for $\bbQ$-a.e.\ $\omega \in \Omega_{\lambda,\Lambda}$ there exists $N_0 = N_0(\omega)$ such that for all $N \geq N_0$ and all $z \in \cC_N$ the event $J^\omega_N(z)$ does not occur, which is equivalent to~\eqref{eq:PositivePortionOfHighVariance2}.

    We start estimating $\bbQ[J^\omega_N(z)]$. We take $\ell = 4 R$ and denote the box of side length $\ell$ by $\square_\ell = [0,\ell)^d \cap \bbZ^d$. Then 
\begin{equation}
     J^\omega_N(z) \subseteq \Big\{ \sum_{x \in \square_\ell} \sum_{y \in B_{z,x}} \IND_{ \{ g^\omega_{B(y,R)}(y,y) \geq \overline{g} - \varepsilon \} } < q|B_z|  \Big\} \,,
\end{equation}
where $B_{z,x} = x + (\ell\bbZ^d \cap B_z)$.
Hence, $J^\omega_N(z)$ is also contained in the event that there exists a $x \in \square_\ell$ such that the average of the indicator functions on the set $B_{z,x}$ is smaller than $q$.
Formally,
\begin{equation}
        \bbQ[J^\omega_N(z)]  \leq \bbQ \Big[ \bigcup_{x \in \square_\ell} \Big\{ \frac{1}{|B_{z,x}|}\sum_{y \in B_{z,x}} \IND_{\{ g^\omega_{B(y,R)}(y,y) \geq \overline{g} - \varepsilon \}} < q  \Big\} \Big] \,.
\end{equation}
For every $x \in \square_\ell$ the collection $\{ \IND_{\{ g^\omega_{B(y,R)}(y,y) \geq \overline{g} - \varepsilon \}} : \, y \in B_{z,x} \}$ is composed of independent variables, since $\ell = 4R$ and for each $y\in B_{z,x}$ the random variable $g^\omega_{B(y,R)}(y,y)$ depends only on  $\{\omega_{e}\, :\,e\cap B(y,R) \neq \varnothing,\,e\in \bbE_d \}$ and the conductances are i.i.d. Then, by Chernoff's inequality for the lower tail of Binomial random variables (see Exercise 2.3.5 in~\cite{vershynin2018high}), together with~\eqref{eq:parameter_q}, we have that
\begin{equation}
    \bbQ[J^\omega_N(z)] \leq \ell^d \mathrm{e}^{-c (L/\ell)^d} \,.
\end{equation}
Hence~\eqref{eq:DefinitionL} yields
\begin{equation}
    \sum_{N \in \bbN} \bbQ \Big[\bigcup_{z \in \cC_N} J^\omega_N(z) \Big] \stackrel{\eqref{eq:CardinalityCN}}{\leq} C \sum_{N \in \bbN} N^{d-2} \sqrt{\log N} \mathrm{e}^{-c L^d} < \infty \,.
\end{equation}
By Borel-Cantelli lemma we have the claim~\eqref{eq:PositivePortionOfHighVariance2}.
\end{proof}

The following proposition characterizes the first order behavior of the maximum of the local field $\psi^{\omega,z}$ uniformly over all the boxes $B_z$, $z \in \cC$. We later use this result in the asymptotic upper bound for the probability of the hard wall event $\cW^+_N$, in order to show that, conditionally on $\cW^+_N$, with high probability there is a large number of ``good boxes'' where the maximum of the harmonic average $\xi^{\omega,z}$ is above the threshold $\sqrt{4\overline{g}\log N}$.
\begin{prop}
    \label{prop:MaxLocalField}
    For $\bbQ$-a.e.\ $\omega \in \Omega_{\lambda,\Lambda}$
    \begin{equation}\label{eq:MaxLocalField}
        \lim_{N \to \infty} \sup_{z \in \cC_N} \left| \frac{\bbE^\omega[ \max_{y \in B_z} \psi^{\omega,z}_y ]}{\sqrt{2\overline{g} \log |B_z|}} - 1 \right| = 0 \,.
    \end{equation}
\end{prop}
\begin{proof}
    We first prove the upper bound. By Sudakov-Fernique inequality we can dominate $(\psi^{\omega,z}_y)_{y \in B_z}$ with a field $(\beta_y)_{y \in B_z}$ of independent centered Gaussians with variance $\overline{g}$ for $\bbQ$-a.e.\ $\omega \in \Omega_{\lambda,\Lambda}$. Then by classical estimates for i.i.d.\ Gaussians we get that $\bbQ$-a.s.\ for every $z \in \bbZ^d$
    \begin{equation}
        \bbE^\omega[ \max_{y \in B_z} \psi^{\omega,z}_y ] \leq \bbE^\omega[ \max_{y \in B_z} \beta_y ] \leq \sqrt{2 \overline{g} \log |B_z|} \,.
    \end{equation}
    We now turn to the lower bound. Fix $\varepsilon >0$ and consider the set
    \begin{equation}
        A^{\omega,\varepsilon}_N (z) = \{ y \in B_z : g^\omega_{U_z}(y,y) \geq \overline{g} - \varepsilon \} \,, \qquad z \in \bbZ^d \,.
    \end{equation}
    By Proposition~\ref{prop:PositivePortionOfHighVariance} $\bbQ$-a.s.\ for every $z \in \cC_N$
    \begin{equation}
        \liminf_{N \to \infty} \frac{|A^{\omega,\varepsilon}_N(z)|}{|B_z|} \geq q/2 >0 \,.
    \end{equation} 
    We now introduce a further parameter $J > 0$ large enough and consider a set $\mathscr{A}^{\omega,\varepsilon}_{J,N}(z)$ of points in $A^\omega_{N,\varepsilon}(z)$ with mutual distance at least $J$ and such that
    \begin{equation}
        |\mathscr{A}^{\omega,\varepsilon}_{J,N}(z)| \geq c |A^{\omega,\varepsilon}_N(z)| J^{-d} \,.
    \end{equation}
    One such set always exists. Note that for every $y,y' \in \mathscr{A}^{\omega,\varepsilon}_{J,N}(z)$ with $y \neq y'$
    \begin{equation}
        \bbE^\omega[ (\psi^{\omega,z}_y - \psi^{\omega,z}_{y'})^2 ] \stackrel{\eqref{eq:QuenchedGFEstimate}}{\geq} 2(\overline{g}-\varepsilon) - 2 c_2 J^{2-d} \,.
    \end{equation}
    Let $(\upsilon_y)_{y \in \mathscr{A}^{\omega,\varepsilon}_{J,N}(z)}$ be a family of independent centered Gaussians with variance $\overline{g}-\varepsilon - c_2 J^{2-d}$. Using again Sudakov-Fernique inequality we have that
    \begin{equation}
        \bbE^\omega [ \max_{y \in B_z} \psi^{\omega,z}_y ] \geq \bbE^\omega [\max_{y \in \mathscr{A}^{\omega,\varepsilon}_{J,N} (z)} \upsilon_y ] \,.
    \end{equation} 
    Dividing by $\sqrt{2\overline{g} \log |B_z|}$ we get
    \begin{equation}
        \label{eq:loweerBoundMaxLocalField}
        \begin{aligned}
            \liminf_{N \to \infty} \frac{ \bbE^\omega[ \max_{y \in \mathscr{A}^{\omega,\varepsilon}_{J,N}(z)} \upsilon_y ] }{\sqrt{2\overline{g} \log |B_z|}} & \geq \liminf_{N \to \infty} \frac{ \bbE^\omega[ \max_{y \in \mathscr{A}^{\omega,\varepsilon}_{J,N}(z)} \upsilon_y ] }{\sqrt{2\overline{g} \log |A^{\omega,\varepsilon}_N(z)|}} \\
            & \geq \liminf_{N \to \infty} \frac{ \bbE^\omega[ \max_{y \in \mathscr{A}^{\omega,\varepsilon}_{J,N}(z)} \upsilon_y ] }{\sqrt{2\overline{g} \log |\mathscr{A}^{\omega,\varepsilon}_{J,N}(z)|}} = \sqrt{\,\overline{g}-\varepsilon - c_2 J^{2-d}} \,,
        \end{aligned}
    \end{equation} 
    having used that for every $z \in \cC_N$
    \begin{equation}
        \lim_{N \to \infty} \frac{\log |B_z|}{\log |A^{\omega,\varepsilon}_N(z)|} =1  \,, \qquad \lim_{N \to \infty} \frac{\log |\mathscr{A}^{\omega,\varepsilon}_{J,N}(z)|}{\log |A^{\omega,\varepsilon}_N(z)|} \geq \lim_{N \to \infty} \frac{ \log |A^{\omega,\varepsilon}_N(z)| - d \log J }{\log |A^{\omega,\varepsilon}_N(z)|} = 1 \,,
    \end{equation}
    where the first limit follows from $d \log L + c \leq \log |A^{\omega,\varepsilon}_N(z)| \leq d \log L + C$. Letting $J \to \infty$ and $\varepsilon \to 0$ in~\eqref{eq:loweerBoundMaxLocalField} we get the claim.
\end{proof}
\begin{remark}\label{rem:MaxGFFErgodic}
    In general, for every stationary and ergodic probability measure $\bbQ'$ on $\Omega_{\lambda,\Lambda}$, on every sequence of sets that grows to $\bbZ^d$ as $N \to \infty$, we have the same behavior as the one in Proposition~\ref{prop:MaxLocalField} for the expectation of the maximum of the Gaussian free field $\varphi$ , namely for $\bbQ'$-a.e.\ $\omega \in \Omega_{\lambda,\Lambda}$ we have the following pointwise convergence:
    \begin{equation}
        \lim_{N \to \infty} \frac{\bbE^\omega[ \max_{y \in D_N} \varphi_y ]}{\sqrt{2\overline{g} \log |D_N|}} =1\,,
    \end{equation}
    where $D_N = (ND)\cap \bbZ^d$ for some open and relatively compact $D\subseteq \bbR^d$.
    The proof is analogous to the one of Proposition~\ref{prop:MaxLocalField}, using that $\overline{g}=\esssup_{\omega} g^\omega(0,0)$.
\end{remark}

\begin{remark}\label{rem:ChoiceL}
    In~\eqref{eq:DefinitionL} we define $L(N)$ strategically so that, for $z \in \cC_N$, since $|B_z|=(2L)^d$, the first order behavior of the maximum of the local field is given by $\sqrt{2d\, \overline{g}\log L}$ (see Proposition~\ref{prop:MaxLocalField}), which matches the asymptotic level $\sqrt{4 \overline{g} \log N}$ of the expectation of the field conditioned on the hard wall event,  $\bbE^\omega[\varphi_\cdot | \cW^+_N]$, on $V_N$, see Theorem~\ref{thm:MainTheorem}, ii).
\end{remark}

\subsubsection{Quenched solidification estimates}
\label{subsec:Solidification}
In this Section we consider a special family of sets $\cS$ that satisfy~\eqref{eq:CollectionC} and have a certain covering property on $V_N$, and provide in Corollary~\ref{cor:SolidificationCapacity} a solidification result for such sets, namely that the capacity (associated with the random walk) of the union of boxes $B_z$ for $z \in \cS$, is asymptotically bounded from below by $\capa^\omega(V^\varepsilon_N)$, as $N \to \infty$, for any $\varepsilon > 0$. We recall that $V^\varepsilon_N = \{ x \in V_N : |x-y| > \varepsilon N,\, y \in V_N^c   \}$ is the $\varepsilon$-bulk of $V_N$. 

Fix $\varepsilon>0$. We consider a length scale $\hat{L}$ with the property that $N/\hat{L}\to \infty$, $\hat{L}/L \to \infty$,
and consider the family $\mathscr{F}$ of boxes $F$ such that
\begin{equation}
    \label{eq:BoxesToCover}
        F = x + [-\hat{L}  ,  \hat{L}]^d \,\text{ for some } x \in V^{\varepsilon/2}_N \cap  \hat{L}\bbZ^d \,.
\end{equation}
Let $\rho\in [0,1)$ be a fixed parameter.
We consider sets $\cS$ with the following covering property:
\begin{equation}
    \label{eq:CoveringProperty}
    \begin{minipage}{0.7\linewidth}
            $\cS \subseteq \cC_N$ satisfies $|F \cap \cS| \geq (1-\rho)|F \cap \cC_N|$, for every $F\in\mathscr{F}$.
    \end{minipage}
\end{equation}
For each such $\cS$ we also define the union of pairwise disjoint boxes
\begin{equation}
    S = \bigcup_{z \in \cS} B_z \,.
\end{equation}

We are now ready to state the two following quenched results on the solidification properties of the collection of all interfaces $S$ with $\cS$ satisfying~\eqref{eq:CoveringProperty}.
\begin{prop}
    \label{prop:SolidificationHittingTime}
    For every $\varepsilon \in (0,1)$ and $\rho \in [0,1)$
    \begin{equation}
        \lim_{N \to \infty} \sup_{\omega \in \Omega_{\lambda,\Lambda}} \sup_{\cS} \sup_{z \in V_N^\varepsilon}P^\omega_z [H_S = \infty] = 0 \,.
    \end{equation}
\end{prop}
\begin{proof}
    First we note that for $\cS$ as in~\eqref{eq:CoveringProperty}, for every box $F\in \mathscr{F}$ with center $x_F$, by the heat kernel estimates for the killed random walk in~\eqref{eq:HK_bound} we have that for all $y$ such that $|y-x_F|_\infty \leq (3/4) \hat{L}$
    \begin{equation}\label{eq:ProbHitSBeforeLeavingBall}
        \begin{aligned}
            P^\omega_y [H_S \leq T_F] &\geq P^\omega_y [X_{\hat{L}^2} \in S \cap F , \hat{L}^2 \leq T_F] = \sum_{z \in S \cap F}  q^\omega_{\hat{L}^2,F}(y,z)\, \mu^\omega_z \\ & \geq  c |S \cap F| \hat{L}^{-d}\geq c|\cS\cap F| (L/\hat{L})^d \\ &\!\! \stackrel{\eqref{eq:CoveringProperty}} {\geq} c (1-\rho_N) |\cC_N \cap F| (L/\hat{L})^{d}\geq c' (1-\rho) \,,
        \end{aligned}
    \end{equation}
    having used that $\mu^\omega_z$ is uniformly bounded from below, and where $c' \in (0,1)$ is independent of the choices of $F \in \mathscr{F}$, $\omega \in \Omega_{\lambda,\Lambda}$ and $y$ such that $|y-x_F|_\infty \leq (3/4) \hat{L}$.

    The proof of the proposition is now based on the fact that, in order to exit $V_N^{\varepsilon/2}$ starting from inside $V_N^{\varepsilon}$, the random walk must cross and get at $|\cdot|_\infty$-distance $(3/4)\hat{L}$ from the centers of at least $k_N = c (N/\hat{L})$ boxes $F\in\mathscr{F}$.
    As such by a repeated application of the strong Markov property and~\eqref{eq:ProbHitSBeforeLeavingBall} we get that for every $z \in V_N^\varepsilon$
    \begin{equation}
         P^\omega_z [H_S = \infty] \leq [1-c'(1-\rho)]^{k_N} \,.
    \end{equation}
    The claim now follows by taking the supremum over all $\omega \in \Omega_{\lambda,\Lambda}$ and $\cS$ as in~\eqref{eq:CoveringProperty}, and noting that the right-hand side goes to $0$ as $N \to \infty$, since $\rho\in [0,1)$ and $k_N \to \infty$.
    \end{proof}
The proof of the following corollary is analogous to the one of (4.8) of Corollary 4.2 in~\cite{chiarini2021disconnection} and uses the result of Proposition~\ref{prop:SolidificationHittingTime}.
\begin{corollary}
    \label{cor:SolidificationCapacity}
    For every $\varepsilon \in (0,1)$, $\rho \in [0,1)$
    \begin{equation}
        \liminf_{N \to \infty} \inf_{\omega \in \Omega_{\lambda,\Lambda}} \inf_{\cS} \frac{\capa^\omega (S)}{\capa^\omega (V_N^\varepsilon)} \geq 1 \,.
    \end{equation}
\end{corollary}

\subsubsection{Proof of the quenched upper bound for the decay of the hard wall event probability}
\label{subsec:ProofUBHW}
In this section we finish the proof of the upper bound for the asymptotic decay rate, as $N \to \infty$, of the probability of the hard wall event $\cW^+_N$.
We now sketch the main steps of the proof.
We first define the concept of bad and good boxes, and show in Proposition~\ref{prop:NegligibleEvent} that we can $\bbQ$-a.s.\ neglect the event on which the set of vertices indexing good boxes does not satisfy the covering property~\eqref{eq:CoveringProperty}, hence we can assume that it is satisfied for $\bbQ$-a.e.\ environment. The proof is concluded in Proposition~\ref{prop:UpperBoundHW}, where we note that on the intersection of this event with $\cW^+_N$, in each good box there is at least one point where the harmonic average of the field is above the threshold $\sqrt{4\overline{g} \log N}$. We encode this event into the supremum of some linear combinations, which are the objects we considered in Section~\ref{sec:PrelimGaussianFunctional} (see~\eqref{eq:ZfDefinition} and~\eqref{eq:ZDefinition}). Via the results we previously obtained on these variables, we get an estimate for the exact constant for the decay of $\bbP^\omega[\cW^+_N]$ in terms of the capacity of the union of good boxes. Lastly, we use the solidification results of Section~\ref{subsec:Solidification} and the homogenization result to complete the proof.

We say that a vertex $z \in \bbZ^d$ is \textit{good} at level $a$ (or, equivalently, that $B_z$ is good at level $a$) if
\begin{equation}
    \label{eq:DefinitionBad}
    \min_{y \in B_z} \psi_y^{\omega,z} \leq -\sqrt{a \log N}
\end{equation}
and we say that $z$ is \textit{bad} at level $a$ (or, equivalently, that $B_z$ is bad at level $a$) otherwise.
Note that if $z$ is good at level $a$ and $\varphi$ is non-negative on $B_z$, then $\max_{y \in B_z} \xi^{\omega,z}_y \geq \sqrt{a \log N}$.
For $\delta \in (0,1)$ we define
\begin{equation}
    \label{eq:DefinitionA}
    a_\delta = 4 \delta \overline{g}  \,.
\end{equation}

\begin{prop}
    \label{prop:NegligibleEvent}
    Recall $\cC_N$ as in~\eqref{eq:DefinitionCN}. For every $\delta \in (0,1)$ we have that for $\bbQ$-a.e.\ $\omega \in \Omega_{\lambda,\Lambda}$
    \begin{equation}
        \label{eq:NegligibleEvent}
        \limsup_{N \to \infty} \frac{1}{N^{d-2}\log N} \log \bbP^\omega \left[
            \begin{minipage}{0.35\textwidth}
                \begin{center}
                $\{ z \in \cC_N: z $ is good at level $ a_\delta \}$\\ does not satisfy~\eqref{eq:CoveringProperty} with $\hat{L} = N/(\log N)^{1/4d}$ and $\rho = 1/2$
            \end{center}
        \end{minipage}
        \right] = - \infty \,,
    \end{equation}
\end{prop}
\begin{proof}
    Fix $\delta \in (0,1)$, and let $\hat{L} = N/(\log N)^{1/4d}$ and $\rho = 1/2$.
    We first provide a uniform upper bound for the probability that a vertex $z \in \bbZ^d$ is bad at level $a_\delta$.
    Let
    \begin{equation}
        \label{eq:DefinitionProbBad}
        \eta^\omega_\delta = \sup_{z\in \cC_N} \bbP^\omega [ \min_{y \in B_z} \psi^{\omega,z}_y > -\sqrt{a_\delta \log N} ] \,.
    \end{equation}
    We claim that
    \begin{equation}
        \label{eq:ClaimProbBad}
        \eta^\omega_\delta \leq 2 N^{-\delta'}\,,
    \end{equation}
    with  $\delta'(\delta) > 0$.
    To show this, we first note that by the definition of $L$ in~\eqref{eq:DefinitionL} we have for $N$ large enough
    \begin{equation}\label{eq:LowerBoundExpectationMaxLocalField}
        \sqrt{2 \overline{g} \log |B_z|} \geq \frac{1+\sqrt{\delta}}{2}\sqrt{4  \overline{g} \log N} \,,
    \end{equation}
 therefore by the uniform behavior of the expected maximum of the local field over boxes $B_z$, $z\in \cC_N$, in Proposition~\ref{prop:MaxLocalField} we get for $\bbQ$-a.e.\ $\omega \in \Omega_{\lambda,\Lambda}$ and $N$ large enough
    \begin{equation}
        \begin{aligned}
            \eta^\omega_\delta & = \sup_{z\in \cC_N} \bbP^\omega \Big[ \max_{y \in B_z} \psi^{\omega,z}_y - \bbE^\omega [\max_{y \in B_z} \psi^{\omega,z}_y] \leq \sqrt{4\delta \overline{g} \log N} - \bbE^\omega [\max_{y \in B_z} \psi^{\omega,z}_y]\Big]  \\&\!\! \stackrel{\eqref{eq:MaxLocalField}}{\leq}\sup_{z\in \cC_N}\bbP^\omega \Big[\big|\max_{y \in B_z} \psi^{\omega,z}_y - \bbE^\omega [\max_{y \in B_z} \psi^{\omega,z}_y]\big| \geq \frac{1 - \sqrt{\delta}}{2} \sqrt{4  \overline{g} \log N} \Big]
        \end{aligned}
    \end{equation}
    and we have~\eqref{eq:ClaimProbBad} by Borell-TIS inequality.
    Note that the number $|\mathscr{F}|$ of boxes as in~\eqref{eq:BoxesToCover} satisfies
    \begin{equation}
        \lim_{N \to \infty} \frac{\log |\mathscr{F}|}{N^{d-2}\log N} =0 \,.
    \end{equation}
    Then, by a union bound the claim~\eqref{eq:NegligibleEvent} readily follows if we show that
    \begin{equation}
        \limsup_{N \to \infty} \frac{1}{N^{d-2}\log N} \, \log \, \widetilde{\sup} \, \bbP^\omega \Big[ \sum_{z \in F\cap \cC_N} \IND_{ \{ z \text{ is bad at level }a_\delta \} } \geq \rho |F\cap \cC_N| \Big] = - \infty \,,
    \end{equation}
    where $\widetilde{\sup}$ is the supremum on the set of boxes $F\in \mathscr{F}$ as in~\eqref{eq:BoxesToCover}. We note that  $\IND_{\{ z \text{ is bad at level }a_\delta \}}$, $z\in F\cap \cC_N$ are independent Bernoulli random variables with parameter at most $\eta_\delta^\omega$, as the local fields $\psi^{\omega,z}$, $z\in F\cap \cC_N$ are independent. Using standard exponential estimates on the sum of independent Bernoulli random variables we get
    \begin{equation}
        \begin{aligned}
        \bbP^\omega \Big[ \sum_{z \in F\cap \cC_N} &\IND_{ \{ z \text{ is bad at level }a_\delta \} } \geq \rho |F\cap \cC_N| \Big] \\ &\leq \exp \Big\{ - |F\cap \cC_N| \Big[ \rho \log \frac{\rho}{\eta^\omega_\delta} + (1-\rho) \log \frac{1-\rho}{1-\eta^\omega_\delta} \Big] \Big\} \,.
        \end{aligned}
    \end{equation}
    This yields 
    \begin{equation}
        \begin{aligned}
            \limsup_{N \to \infty} & \frac{1}{N^{d-2}\log N} \, \log \, \widetilde{\sup} \, \bbP^\omega \Big[ \sum_{z \in F\cap \cC_N} \IND_{ \{ z \text{ is bad at level }a_\delta \} } \geq \rho |F\cap \cC_N| \Big] \\
            & \leq - \liminf_{N \to \infty} \widetilde{\inf} \frac{c |F\cap \cC_N|}{N^{d-2}\log N}   \log (1/{\eta^\omega_\delta}) \\
            &\leq  - \liminf_{N \to \infty} \frac{c (\hat{L}/L)^d}{N^{d-2} \log N} \log N = -\infty \,,
        \end{aligned}
    \end{equation}
   having used that $|F\cap \cC_N| \geq c (\hat{L}/L)^d$, $\hat{L} = N/(\log N)^{1/4d}$ and \eqref{eq:ClaimProbBad}.
\end{proof}
We are now ready to use the results of Section~\ref{sec:PrelimGaussianFunctional}, Section~\ref{subsec:Solidification} and the homogenization property~\eqref{eq:hom_capacity} to prove in the following proposition the quenched asymptotic upper bound for the decay rate of the probability of the hard wall event as $N \to \infty$.
\begin{prop}
    \label{prop:UpperBoundHW}
    For $\bbQ$-a.e.\ $\omega \in \Omega_{\lambda,\Lambda}$
    \begin{equation}
        \label{eq:UpperBoundHWClaim}
        \limsup_{N \to \infty} \frac{1}{N^{d-2}\log N} \log \bbP^\omega [\cW^+_N] \leq -2 \overline{g} \,\capa^{\mathrm{hom}}(V) \,.
    \end{equation}
\end{prop}
\begin{proof}   
    Let $\delta \in (0,1)$ and recall $a_\delta$ as in~\eqref{eq:DefinitionA}. Consider the sets
    \begin{equation}
        \label{eq:DefinitionS}
        \cS^\omega_\delta = \{ z \in \cC_N: z \text{ is good at level } a_\delta \} \,, \qquad S^\omega_\delta = \bigcup_{z \in \cS^\omega} B_z \,.
    \end{equation}
    We start noticing, in view of~\eqref{eq:DefinitionBad}, that 
    \begin{equation}\label{eq:InclusionHardWallGoodBoxes}
        \cW^+_N \subseteq \bigcap_{z\in \cS^\omega_\delta}\{ \varphi_x \geq 0 \text{ for all } x \in B_z \} \subseteq \bigcap_{z\in \cS^\omega_\delta}\Big\{ \max_{y \in B_z} \xi^{\omega,z}_y \geq \sqrt{a_\delta \log N} \Big\} \,,
    \end{equation}
    As a consequence of Proposition~\ref{prop:NegligibleEvent}, by discarding an event of negligible probability, we can assume that $\cS^\omega_\delta$ satisfies~\eqref{eq:CoveringProperty} with $\hat{L} = N/(\log N)^{1/4d}$ and $\rho = 1/2$, for $\bbQ$-a.e.\ $\omega \in \Omega_{\lambda,\Lambda}$.
    The number of possible choices $\cS\subseteq \cC_N$ for $\cS^\omega_\delta$ with this property is at most the number of subsets of $\cC_N$, which is $2^{|\cC_N|}$. By~\eqref{eq:CardinalityCN} we have that this number is of order $\exp\{o(N^{d-2}\log N)\}$.

    Thus, with a union bound,  we can write
    \begin{equation}\label{eq:UpperBoundHWStart}
        \varlimsup_{N \to \infty} \frac{1}{N^{d-2}\log N} \log \bbP^\omega [\cW^+_N] \leq \varlimsup_{N \to \infty} \sup_{\cS}\frac{1}{N^{d-2}\log N} \log  \bbP^\omega \left[ \bigcap_{z\in \cS}\Big\{ \max_{y \in B_z} \xi^{\omega,z}_y \geq \sqrt{a_\delta \log N} \Big\} \right] \,.
    \end{equation}
    where the supremum is over all $\cS \subseteq \cC_N$ satisfying~\eqref{eq:CoveringProperty} with $\hat{L} = N/(\log N)^{1/4d}$ and $\rho = 1/2$.
    Consider the Gaussian functional $Z_f^\omega$ defined as in~\eqref{eq:ZfDefinition} for a set of vertices $\cS$ as above, that is,
    \begin{equation}
        Z_f^\omega = \sum_{z \in \cS} \nu^\omega(z) \xi^{\omega,z}_{f(z)} \,, \qquad f \in \cF \,,
    \end{equation}
    and $Z^\omega$ as in~\eqref{eq:ZDefinition}. 
    We notice that the event on the right-hand side of \eqref{eq:UpperBoundHWStart} is contained in the event $\{ Z^\omega \geq \sqrt{a_\delta \log N} \}$, and thus we can further bound the right-hand side of~\eqref{eq:UpperBoundHWStart} as
    \begin{equation}\label{eq:UpperBoundHWEncodingInZ}
        \varlimsup_{N \to \infty} \frac{1}{N^{d-2}\log N} \log \bbP^\omega [\cW^+_N] \leq \varlimsup_{N \to \infty} \sup_{\cS}\frac{1}{N^{d-2}\log N} \log \bbP^\omega \left[ Z^\omega \geq \sqrt{a_\delta \log N} \right] \,.
    \end{equation}
    Using Corollary~\ref{cor:UpperBoundProbZf} (note that for $\beta=0$ and $b=1$ it holds that $\sup_{f \in \cF}Z^\omega
    _{f,\beta,b}=Z^\omega$), the right-hand side of the above display is bounded  by
    \begin{equation}\begin{aligned}
          - \varliminf_{N\to\infty}\inf_{\cS} \frac{1}{2} \left( \sqrt{a_\delta } - \frac{c_7}{K^{c_8}} \sqrt{\frac{|\cS|}{\capa^\omega(S)\log N}} \right)_+^2 \frac{\capa^\omega(S)}{N^{d-2}} \frac{1}{1+U(K,L)} \,,
    \end{aligned}
    \end{equation}
    with $\lim_{K \to \infty} \limsup_{L \to \infty} U(K,L) = 0$. Note that for every $\cS$ satisfying~\eqref{eq:CoveringProperty} we have by Corollary~\ref{cor:SolidificationCapacity} and \eqref{eq:CardinalityCN} that $\bbQ$-a.s.
    \begin{equation}\label{eq:RelationcSCapS}
        \varlimsup_{N \to \infty} \sup_{\omega \in \Omega_{\lambda,\Lambda}} \sup_{\cS} \frac{|\cS|}{\capa^\omega(S) \log N} \leq  \varlimsup_{N \to \infty} \frac{c N^{d-2}}{\capa^\omega(V_N^{\varepsilon}) \sqrt{\log N}} = 0 \,.
    \end{equation}

    Plugging~\eqref{eq:RelationcSCapS} into the previous estimate and using in order the solidification estimate in Corollary~\ref{cor:SolidificationCapacity} and the homogenization result in~\eqref{eq:hom_capacity}, we get that for every $\varepsilon \in (0,1)$
    \begin{equation}
        \begin{aligned}
        \varlimsup_{N \to \infty} \frac{1}{N^{d-2}\log N} \log \bbP^\omega [ \cW^+_N ] &\leq - \varliminf_{N \to \infty}  \frac{1}{2}  \frac{a_\delta \, \capa( V_N^\varepsilon)}{N^{d-2}} \frac{1}{1 + U(K,L)} \,, \\
        & \leq - 2 \delta \overline{g}\, \capa^{\mathrm{hom}}(V^\varepsilon)\, \big(1+\varlimsup_{L\to \infty} U(K,L)\big)^{-1}\,,
    \end{aligned}
    \end{equation} 
    where we used that $a_\delta = 4 \delta \overline{g}$ and where we assumed without loss of generality that $V^\varepsilon$ has Lipschitz boundary in the application of the homogenization~\eqref{eq:hom_capacity}.
    We have the claim~\eqref{eq:UpperBoundHWClaim} by letting $K \to \infty$, $\varepsilon \to 0$ and $\delta \to 1$. 
\end{proof}

\begin{proof}[Proof of Theorem~\ref{thm:MainTheorem}, i)]
    The claim  now simply follows by combining the lower bound in Proposition~\ref{prop:LowerBoundHW} and the upper bound in Proposition~\ref{prop:UpperBoundHW}.
\end{proof}

\section{Expectation of the conditioned field}
\label{sec:ExpConditionedField}
In this section we get some key estimates for the expectation of the field conditioned on $\cW^+_N$, that are also used in Section~\ref{sec:ProofEntropicRepulsion} for the proof of the pathwise entropic repulsion, and prove the characterization of its asymptotic behavior, as in Theorem~\ref{thm:MainTheorem}, ii). The pivotal estimates for the conditioned expectation are provided in Section~\ref{subsec:PreliminariesExpConditionedField}, while Section~\ref{subsec:ExpCF} is devoted to the proof of its asymptotics. 
\subsection{Preliminary estimates for the expectation of the conditioned field}
\label{subsec:PreliminariesExpConditionedField}
Given $\phi \in \bbR^{\bbZ^d}$, we define the field of weighted averages of its neighboring entries
\begin{equation}\label{eq:DefM}
    m^\omega_x(\phi) = \frac{1}{\mu^\omega_x} \sum_{y \,: \, y \sim x} \omega_{x,y} \phi_y \,, \qquad x \in \bbZ^d \,.
\end{equation}
Note that when $\varphi$ is the Gaussian free field we can rewrite \eqref{eq:DLR_GFF} for $\bbP^\omega$ as
    \begin{equation}\label{eq:DLR2_GFF}
        \bbP^\omega(\De \phi | \cF_{\{x\}^c})(\varphi) = \frac{1}{\sqrt{2\pi / \mu^\omega_x}} \exp \left\{ - \frac{\mu^\omega_x}{2} \big(\phi_x - m^\omega_x(\varphi)\big)^2 \right\} \De \phi_x \prod_{y \neq x}\delta_{\varphi_y}(\De \phi_y) \,,\quad \bbP^\omega(\De \varphi)\text{-a.s.}
    \end{equation}
Consider the law of the conditioned field
\begin{equation}\label{eq:LawCF}
    \bbP^{\omega,+}_N [\,\cdot\,] = \bbP^\omega [ \,\cdot\, | \cW^+_N]
\end{equation}
and its associated expectation $\bbE^{\omega,+}_N $.
The following lemma implies that the field of conditioned expectations $\bbE^{\omega,+}_N[\varphi_x]$, $x \in \bbZ^d$, is a \emph{superharmonic} function. In addition to its use in this section, this will be instrumental also in Section~\ref{sec:ProofEntropicRepulsion} for the proof of Theorem~\ref{thm:MainTheorem}, iii). The proof of the lemma exploits~\eqref{eq:DLR2_GFF} and is a simple adaptation of the one of Lemma 2.3 in~\cite{deuschel1999entropic} to our setting. We write it in Appendix~\ref{secapp:ProofEstimatesExpectationCF} for the convenience of the reader.
\begin{lemma}\label{lem:SuperharmonicConditionedExpectation}
    For every $\omega \in \Omega_{\lambda,\Lambda}$, $x \in \bbZ^d$ and $N \in \bbN$
    \begin{equation}
        \bbE^{\omega,+}_N [\varphi_x - m^\omega_x(\varphi)] = \begin{cases}
            \frac{1}{\sqrt{2\pi \mu^\omega_x}} \bbE^{\omega,+}_N \bigg[ \frac{\exp \left(-\frac{\mu^\omega_x}{2} m^\omega_x(\varphi)^2 \right)}{\phi(m^\omega_x(\varphi) \sqrt{\mu^\omega_x})} \bigg] \,, &\mbox{ for } x \in V_N \,, \\
            0 \,, &\mbox{ for } x \in V_N^c \,,
        \end{cases}
    \end{equation}
    where $\phi$ is the standard normal cumulative distribution function.
\end{lemma}
As a consequence of Lemma~\ref{lem:SuperharmonicConditionedExpectation} we have that for all $\omega \in \Omega_{\lambda,\Lambda}$, $x \in \bbZ^d$ and $N \in \bbN$
\begin{equation}\label{eq:PreciseExpression-LCF}
    -\cL^\omega \bbE^{\omega,+}_N[\varphi_\cdot](x) 
    \begin{cases}
        \geq 0 \,, &\mbox{if } x \in V_N \,, \\
        =0 \,, &\mbox{if } x \in V_N^c
    \end{cases}
\end{equation}
(recall that $\cL^\omega$ is the generator of the random walk defined in~\eqref{eq:GeneratorCSRW}) and that for every bounded $B \subseteq \bbZ^d$ with $B\ni x$ it holds 
\begin{equation}\label{eq:SuperharmonicInequality}
    \bbE^{\omega,+}_N[\varphi_x] \geq \sum_{y \in \partial B} P^\omega_x [X_{T_B} = y ] \bbE^{\omega,+}_N[\varphi_y] \,.
\end{equation}
This inequality is pivotal in the proofs of both the lower and the upper bounds for~\eqref{eq:AsymptoticExpectationCF} below, as it allows to bypass the arguments based on the FKG inequality and translation invariance (which we lack) that were used in the proofs of the analogous results in~\cite{bolthausen1995entropic}.
\begin{remark}\label{rem:DecayCondExpectationAtInf}
    For every fixed $N \in \bbN$ the field of conditioned expectations satisfies
    \begin{equation}
        \lim_{R \to \infty} \sup_{|x|_\infty \geq R}\bbE^{\omega,+}_N[\varphi_x] =0\,.
    \end{equation}
    Clearly, it is enough to consider $\bbE^{\omega}_N[\varphi_x \IND_{\cW^+_N}]$ in place of $\bbE^{\omega,+}[\varphi_x]$.  For $x \in V_N^c$ we first write
    \begin{equation}
            \bbE^\omega[\varphi_x \IND_{\cW^+_N}]  =  \bbE^\omega \big[ \bbE^\omega[\varphi_x | \cF_{V_N}] \IND_{\cW^+_N} \big]  = \sum_{y \in \partial_{\mathrm{in}} V_N} P^\omega_x [X_{H_{V_N}} = y , H_{V_N} < \infty] \bbE^\omega[\varphi_y \IND_{\cW^+_N}] \,,
    \end{equation}
    having used in the second step~\eqref{eq:HarmonicAverageDef} and~\eqref{eq:DomainMP}. This implies that 
    \begin{equation}
        0 \leq  \sup_{|x|_\infty \geq R}\bbE^\omega[\varphi_x \IND_{\cW^+_N}] \leq (\overline{g} \,\bbP^\omega[\cW^+_N])^{1/2}  \sup_{|x|_\infty \geq R} h^\omega_{V_N}(x) \,.
    \end{equation}
    Now $ \sup_{|x|_\infty \geq R} h^\omega_{V_N}(x) = \sup_{|x|_\infty \geq R} e^\omega_{V_N} G^\omega (x) \leq c R^{2-d}$ which goes to zero as $R \to \infty$.
\end{remark}

\subsection{Asymptotic behavior of the conditioned expectation}\label{subsec:ExpCF}
In this Section we prove Theorem~\ref{thm:MainTheorem}, ii).
The strategy of the proof goes as follows.
We first show that for $\bbQ$-a.a.\ $\omega \in \Omega_{\lambda,\Lambda}$ it holds
\begin{equation}\label{eq:AsymptoticExpectationCF}
    \limsup_{N \to \infty} \sup_{x \in \bbZ^d} \left| \frac{\bbE^\omega[\varphi_x | \cW^+_N]}{\sqrt{4 \overline{g} \log N}} - h^\omega_{V_N}(x) \right| = 0 \,.
\end{equation}
The lower and upper bound for~\eqref{eq:AsymptoticExpectationCF} are given in Section~\ref{subsubsec:LowerBoundExpCF} and Section~\ref{subsubsec:UpperBoundExpCF}, respectively.
We then conclude the proof by using the homogenization result in Lemma~\ref{lem:UnifHarmonicPotential} that provides $\bbQ$-almost sure uniform convergence, as $N \to \infty$, of $h^\omega_{V_N}(\lfloor N x \rfloor)$  to $\mathscr{h}^\mathrm{hom}_V(x)$.
This allows us to swap $h^\omega_{V_N}(x)$ with $\mathscr{h}^{\mathrm{hom}}_V(x/N)$ in~\eqref{eq:AsymptoticExpectationCF}.
\subsubsection{Lower bound for~\eqref{eq:AsymptoticExpectationCF}}\label{subsubsec:LowerBoundExpCF}
We prove the lower bound for~\eqref{eq:AsymptoticExpectationCF} in Proposition~\ref{prop:LBCF}, using Proposition~\ref{lem:DecayScalProd} below. The main idea is to show that averages of the field over the boundary of large boxes are close to $\sqrt{4 \overline{g} \log N}$  with exponentially high probability under $\bbP^{\omega,+}_N$. We then crucially use the superharmonic property~\eqref{eq:SuperharmonicInequality} to bound $\bbE^{\omega,+}_N[\varphi_\cdot]$ from below. This strategy allows to bypass the lack translation invariance of the field with inhomogeneous conductances that was exploited in~\cite{bolthausen1995entropic}.  The proof of Proposition~\ref{lem:DecayScalProd} is inspired by the proof in \cite[Theorem 7.2]{chiarini2021disconnection}, with the crucial difference that therein averages over boxes were considered, while here we need to work with averages over the boundary of boxes, which is more delicate.

Given $\varepsilon \in (0,1)$ and $x \in V_N^\varepsilon$, let us define 
\begin{equation}\label{eq:DefRNEps}
    r_N^\varepsilon =  \frac{\varepsilon}{2\sqrt{d}}N\,, \qquad \text{so that } B(x,r_N^\varepsilon) \subseteq V^{\varepsilon/2}_N \text{ for all } x \in V_N^\varepsilon \,,
\end{equation} 
\begin{prop}\label{lem:DecayScalProd}
    For every $\delta \in (0,1)$ and $\varepsilon > 0$, 
    \begin{equation}
        \label{eq:ScalProdBulk}
        \limsup_{N \to \infty} \sup_{x \in V_N^\varepsilon}  \frac{1}{{N^{d-2}\log N}} \log \bbP^{\omega,+}_N [ \langle \varphi \,,\, \eta^\omega_{x,r_N^\varepsilon} \rangle \leq (1-\delta) \sqrt{4 \overline{g} \log N} ]  \leq -c \,.
    \end{equation}
\end{prop}
\begin{proof}
    The probability on the left-hand side of~\eqref{eq:ScalProdBulk} can be written as
    \begin{equation}\label{eq:ConditionedProbExpanded}
        \frac{\bbP^\omega[ \{\langle \varphi \,,\, \eta^\omega_{x,r^\varepsilon_N} \rangle \leq (1-\delta) \sqrt{4 \overline{g} \log N} \} \cap \cW^+_N ]}{ \bbP^\omega[\cW^+_N] }\,.
    \end{equation}
    In view of Proposition~\ref{prop:LowerBoundHW}, it suffices to get an asympotic upper bound for the probability in the numerator.
    We start with a preliminary step. For $\cS \subseteq \cC_N$ and $x \in V_N^\varepsilon$, we define the sets
    \begin{equation}
        \cS(x) = \{ z \in \cS: U_z \cap \partial B(x,r^\varepsilon_N) = \varnothing \} \,,\qquad S(x) = \bigcup_{z \in \cS(x)} B_z \,.
    \end{equation}
     For $\cS \subseteq \cC_N$ and $x \in V_N^\varepsilon$, we also consider the Gaussian functionals $Z^\omega_f(x)$ and $Z^\omega_{f,\beta,b}(x)$ defined as in~\eqref{eq:ZfDefinition} and~\eqref{eq:ZfbetabDefinition} for the set $\cS(x)$, that is, for $f \in \cF$,
    \begin{equation}
        \begin{aligned}
            &Z^\omega_f(x) = \sum_{z \in \cS(x)} \nu^{\omega}(z) \xi^{\omega,z}_{f(z)}\,,\\
            &Z^\omega_{f,\beta,b}(x) = b Z^\omega_f(x) - \beta \langle \varphi , \eta^\omega_{x,r^\varepsilon_N} \rangle \,, \qquad b>0\,, \beta \geq 0 \,.
        \end{aligned}
    \end{equation}
   Consider $\hat{L} = N/(\log N)^{1/4d}$, for every $F \subseteq \mathscr{F}$ (recall~\eqref{eq:BoxesToCover})  and $\cS\subseteq \cC_N$ we have
    \begin{equation}\label{eq:EstimateSx}
        \inf_{x \in V_N^\varepsilon} | \cS(x) \cap F| = |\cS \cap F| - \sup_{x \in V_N^\varepsilon}|(\cS \setminus \cS(x)) \cap F| \geq |\cS \cap F| - c (N/L)^{d-1}\,.
    \end{equation}
    Now suppose that $N$ is large enough so that $c (N/L)^{d-1} \leq 1/4 |\cC_N\cap F|$ for all $F \in \mathscr{F}$. Then, if $\cS\subseteq \cC_N$ satisfies the covering property~\eqref{eq:CoveringProperty} with $\rho = 1/2$ and $\hat{L}=N/(\log N )^{1/4d}$, we get from~\eqref{eq:EstimateSx} that also $\cS(x)$ satisfies~\eqref{eq:CoveringProperty} with $\rho = 3/4$ and $\hat{L}=N/(\log N )^{1/4d}$ for every $x \in V_N^\varepsilon$.
   
    As in the proof of Proposition~\ref{prop:UpperBoundHW}, for $\delta \in (0,1)$ we consider the set of vertices
    \begin{equation}
        \cS^\omega_\delta = \{ z \in \cC_N: z \text{ is good at level } a_\delta \}\,.
    \end{equation}
     Analogously to~\eqref{eq:InclusionHardWallGoodBoxes}, we have for every $x \in V_N^\varepsilon$ that
    \begin{equation}
        \cW^+_N \subseteq \bigcap_{z\in \cS^\omega_\delta(x)}\{ \varphi_y \geq 0 \text{ for all } y \in B_z \} \subseteq \bigcap_{z\in \cS^\omega_\delta(x)}\Big\{ \max_{y \in B_z} \xi^{\omega,z}_y \geq \sqrt{a_\delta \log N} \Big\} \,.
    \end{equation}
    By Proposition~\ref{prop:NegligibleEvent}, by discarding an event of negligible probability, we can assume that $S_\delta^\omega$ satisfies the covering property with $\rho=1/2$ and $\hat{L}=N/(\log N)^{1/4d}$ for all $N$ sufficiently large. By the previous discussion, there exists $N_0 \in \bbN$ such that for all $N \geq N_0$ and for $\bbQ$-a.e.\ $\omega \in \Omega_{\lambda,\Lambda}$,
    $\cS^\omega_\delta(x)$ satisfies the covering property~\eqref{eq:CoveringProperty} with $\rho = 3/4$ and $\hat{L}=N/(\log N)^{1/4d}$ for all $x\in V_N^\varepsilon$.

    Similarly to the arguments in the proof of Proposition~\ref{prop:UpperBoundHW}, going from below~\eqref{eq:InclusionHardWallGoodBoxes} to~\eqref{eq:UpperBoundHWEncodingInZ}, we can write that for every $\widetilde{\delta}>0$, for $\bbQ$-a.e.\ $\omega \in \Omega_{\lambda,\Lambda}$
    \begin{equation}\label{eq:EstimateNumeratorConditionedProb}
        \begin{aligned}
            &\varlimsup_{N\to \infty}  \sup_{x\in V_N^\varepsilon}\frac{1}{N^{d-2}\log N} \log \bbP^\omega[ \{\langle \varphi \,,\, \eta^\omega_{x,r^\varepsilon_N} \rangle \leq (1-\delta) \sqrt{4 \overline{g} \log N} \} \cap \cW^+_N ]\\
            &\quad \leq \varlimsup_{N\to \infty} \sup_{\cS} \sup_{x\in V_N^\varepsilon}\frac{1}{N^{d-2}\log N} \log \bbP^\omega \Big[ \sup_{f\in \cF }Z^\omega_{f,\beta,b}(x) \geq b\, \widetilde{\delta}\sqrt{4  \overline{g}\log N} -\beta (1-\delta) \sqrt{ 4 \overline{g} \log N}\Big] \,,
        \end{aligned}
    \end{equation}
    where the supremum is over all $\cS \subseteq \cC_N$ satisfying~\eqref{eq:CoveringProperty} with $\hat{L} = N/(\log N)^{1/4d}$ and $\rho = 1/2$.
    We choose $b = 1+\beta \langle h^\omega_{ S} , \eta^\omega_{x,r^\varepsilon_N} \rangle $ and notice that $b$ forms a bounded sequence in $N$. We now bound the right-hand side of~\eqref{eq:EstimateNumeratorConditionedProb} with the help of Corollary~\ref{cor:UpperBoundProbZf}, to get the upper estimate
    \begin{equation}\label{eq:CFNumeratorEstimate2}
        \begin{aligned}
            &\varlimsup_{N\to \infty}\sup_{\cS}\sup_{x \in V_N^\varepsilon}  \frac{1}{N^{d-2}\log N}\bbP^\omega \Big[ \sup_{f \in \cF} Z^\omega_{f,\beta,b}(x) \geq b\, \widetilde{\delta}\sqrt{4\overline{g}\log N} -\beta (1-\delta) \sqrt{ 4 \overline{g} \log N} \Big] \\
            &\quad \leq- \varliminf_{N\to \infty} \inf_{\cS}\inf_{x \in V_N^\varepsilon}   \frac{1}{2}\,\frac{1}{N^{d-2}\log N}  \left( \sqrt{4 \overline{g} \log N} \left( b\, \widetilde{\delta} - \beta (1-\delta) \right) -b \frac{c_7}{K^{c_8}} \sqrt{\frac{|\cS(x)|}{\capa^\omega(S(x))}} \right)_+^2 \\
            &\qquad \cdot \, \frac{\capa^\omega(S(x))}{c_6 \beta^2 N^{2-d} \capa^\omega(S(x)) + \alpha^\omega_{K,L,\beta,b}(x)} \,.
        \end{aligned}
    \end{equation}
    where
    \begin{equation}
        \alpha^\omega_{K,L,\beta,b}(x) = b^2 -2\beta b \langle h^\omega_{S(x)} , \eta^\omega_{x,r^\varepsilon_N} \rangle + \cR_{K,L,\beta,b}(x)\,,
    \end{equation}
    with $0\leq \cR_{K,L,\beta,b}(x) \leq b^2 U(K,L) + 2b\beta \Big[U(K,L) + \langle \IND_{\cU(x)} ,\eta^\omega_{x,r} \rangle \Big]$, where
    \begin{equation}
        \cU(x) = \bigcup_{z\in \cS(x)} U_z
    \end{equation}
    and $\lim_{K \to \infty} \limsup_{L \to \infty} U(K,L) = 0$, $U(K,L) \geq 0$ (see~\eqref{eq:REstimate} and~\eqref{eq:UDecay}).
    We notice that by Proposition~\ref{prop:SolidificationHittingTime}
    \begin{equation}\label{eq:LimitDotProd}
        \lim_{N \to \infty}  \inf_{\cS}\inf_{x \in V_N^\varepsilon}\langle h^\omega_{S(x)} , \eta^\omega_{x,r} \rangle = 1 \,.
        \end{equation}
    Since $\cU(x)\cap \partial B(x,r_N^\varepsilon)=\varnothing$ by construction, we have $\langle \IND_{\cU(x)},\eta^\omega_{x,r^\varepsilon_N} \rangle = 0$, hence for $\beta\leq 1$ 
    \begin{equation}
        0\leq \limsup_{N\to \infty} \alpha^\omega_{K,L,\beta,b}(x) \leq 1 + \vartheta_{K} \,,\qquad \text{ with } \vartheta_{K} =  \varlimsup_{L\to\infty} 8 U(K,L)  \,.
    \end{equation} 
    Furthermore, for any $\widetilde{\varepsilon}<\varepsilon$, we can apply Corollary~\ref{cor:SolidificationCapacity} to the family $S(x)$, $x\in V_N^{\widetilde{\varepsilon}}$ (as they satisfy a covering property with $\rho=3/4$ and $\hat{L}=N/(\log N)^{1/4d}$) and $V_N^{\widetilde{\varepsilon}}$ (assuming w.l.o.g.\ that it has smooth boundary), and use \eqref{eq:RelationcSCapS} for $\cS(x)$ and $S(x)$ to get that the right-hand side of~\eqref{eq:CFNumeratorEstimate2} is bounded from above by 
    \begin{equation}\label{eq:CFNumeratorEstimate3}
             -  2 \overline{g} \left( (1+\beta) \widetilde{\delta} - \beta (1-\delta) \right)^2 \varliminf_{N\to \infty}    \, \frac{N^{2-d}\capa^\omega(V_N^{\widetilde{\varepsilon}})}{c_6 \beta^2 N^{2-d} \capa^\omega(V_N^{\widetilde{\varepsilon}}) + 1 + \vartheta^\omega_{K}} \,.
    \end{equation}
    Finally, we can use the homogenization result~\eqref{eq:hom_capacity} to replace $N^{2-d}\capa^\omega(V_N^{\widetilde{\varepsilon}})$ with $\capa^{\mathrm{hom}}(V^{\widetilde{\varepsilon}})$ in~\eqref{eq:CFNumeratorEstimate3}, take $ K\to \infty$, $\widetilde{\varepsilon} \to 0$, $\widetilde{\delta} \to 1$ to get that the right-hand side of~\eqref{eq:CFNumeratorEstimate2} is bounded from above by
    \begin{equation}
       -  2 \overline{g} \left(1+\delta \beta  \right)^2   \, \frac{\capa^{\mathrm{hom}}(V)}{c_6 \beta^2 \capa^{\mathrm{hom}}(V) + 1} \,.  
    \end{equation}
    We can now find $\beta >0$ depending on $\delta$ small enough such that this is less than $-2 \overline{g} \,\capa^{\mathrm{hom}}(V) - c$.
    The claim now follows by combining this estimate with Proposition~\ref{prop:LowerBoundHW}.
\end{proof}

We are now ready to prove the lower bound for~\eqref{eq:AsymptoticExpectationCF}.
\begin{prop}\label{prop:LBCF}
    For $\bbQ$-a.e.\ $\omega\in\Omega_{\lambda,\Lambda}$
    \begin{equation}\label{eq:LowerBoundExpectationCFNotHomogenized}
        \liminf_{N \to \infty} \inf_{x \in \bbZ^d} \Big[ \frac{\bbE^{\omega,+}_N[\varphi_x]}{\sqrt{4 \overline{g} \log N}} - h^\omega_{V_N}(x)   \Big] \geq 0 \,.
    \end{equation} 
\end{prop}
\begin{proof}
    The result is readily implied by the fact that $h^\omega_{V_N}$ is bounded if we show that for every $\delta \in (0,1)$, for $\bbQ$-a.a.\ $\omega \in \Omega_{\lambda,\Lambda}$ 
    \begin{equation}\label{eq:LowerBoundExpectationCFNotHomogenizedDelta}
        \liminf_{N \to \infty} \inf_{x \in \bbZ^d} \Big[ \frac{\bbE^{\omega,+}_N[\varphi_x]}{\sqrt{4 \overline{g} \log N}} - (1-\delta)h^\omega_{V_N}(x)   \Big] \geq 0 \,.
    \end{equation}
    Fix $\varepsilon>0$ and take $x \in V_N^\varepsilon$. Then, by the superharmonicity of $\bbE^{\omega,+}_N[\varphi_\cdot]$ and the definition of $\eta^\omega_{x,r_N^\varepsilon}$ in~\eqref{eq:DefinitionEta}, we get
    \begin{equation}
        \bbE^{\omega,+}_N[\varphi_x] \geq \sum_{y \in \partial B(x,r_N^\varepsilon)} P^\omega_x [X_{T_{B(x,r_N^\varepsilon)}} = y ] \bbE^{\omega,+}_N[\varphi_y] = \bbE^{\omega,+}_N [\langle \varphi \,,\, \eta^\omega_{x,r_N^\varepsilon} \rangle] \,.
    \end{equation}
    We set $E_{\delta,x}$ to be the event on the left-hand side of~\eqref{eq:ScalProdBulk}, use that $\langle \varphi \,,\, \eta^\omega_{x,r_N^\varepsilon} \rangle\geq 0$ conditionally on $\cW^+_N$, and deduce that
    \begin{equation}
        \bbE^{\omega,+}_{N} [\langle \varphi \,,\, \eta^\omega_{x,r_N^\varepsilon} \rangle] \geq
        (1-\delta) \sqrt{4 \overline{g} \log N}\, \bbP^{\omega,+}_N(E^c_{\delta,x}) \,,
    \end{equation}
    which implies that
    \begin{equation}\label{eq:LowerBoundExpectationCFBulk}
        \liminf_{N\to \infty} \inf_{x\in V_N^\varepsilon}\Big[\frac{\bbE^{\omega,+}_N[\varphi_x]}{\sqrt{4 \overline{g} \log N}} - (1-\delta)h^\omega_{V_N}(x)\Big] \geq  - \limsup_{N\to \infty} \sup_{x\in V_N^\varepsilon} \bbP^{\omega,+}_N(E_{\delta,x}) = 0 \,,
    \end{equation}
    where we used that $h_{V_N}^\omega(x) = 1$ for $x \in V_N^\varepsilon$ and in the last inequality we applied Proposition~\ref{lem:DecayScalProd}. 
  Let $x\in (V_N^\varepsilon)^c$. Then, by the inequality~\eqref{eq:FKG2} 
    \begin{equation}
        \begin{aligned}
            \bbE^{\omega,+}_N[\varphi_x]  &= \bbE^{\omega,+}_N \big[ \bbE^{\omega,+}_N[\varphi_x | \cF_{V_N^\varepsilon}] \big] \\&\geq \bbE^{\omega,+}_N \big[\bbE[\varphi_x | \cF_{V_N^\varepsilon}]\big] =
             \sum_{y \in \partial_{\mathrm{in}} V_N^\varepsilon} P^\omega_x \big[ X_{H_{V_N^\varepsilon}} = y \big] \bbE^{\omega,+}_N[\varphi_y] \,,
        \end{aligned}
    \end{equation}
    which implies with the help of~\eqref{eq:LowerBoundExpectationCFBulk} that for $\bbQ$-a.a.\ $\omega \in \Omega_{\lambda,\Lambda}$
    \begin{equation}
         \liminf_{N\to \infty} \inf_{x\in (V_N^\varepsilon)^c} \Big[\frac{\bbE^{\omega,+}_N [\varphi_x]}{\sqrt{4\overline{g}\log N} } - (1-\delta) h^\omega_{V_N^\varepsilon}(x)  \Big] \geq 0 \,.
    \end{equation}
    We now consider a sequence $\varepsilon_n \downarrow 0$ and show that for $\bbQ$-a.a.\ $\omega \in \Omega_{\lambda,\Lambda}$
    \begin{equation}\label{eq:SolidificationEquilibriumPotential}
        \lim_{n \to \infty} \limsup_{N \to \infty} \sup_{x \in \bbZ^d}  \big| h^\omega_{V_N^{\varepsilon_n}}(x)  - h^\omega_{V_N}(x) \big|   = 0\,, 
    \end{equation}
    which implies~\eqref{eq:LowerBoundExpectationCFNotHomogenizedDelta}. By the homogenization result in Lemma~\ref{lem:UnifHarmonicPotential} applied to $V^{\varepsilon_n}$ and $V$, we have that for $\bbQ$-a.a.\ $\omega \in \Omega_{\lambda,\Lambda}$
    \begin{equation}
         \lim_{n \to \infty} \limsup_{N \to \infty} \sup_{x \in \bbZ^d} \big| h^\omega_{V_N^{\varepsilon_n}}(x)  - h^\omega_{V_N}(x) \big| \leq  \lim_{n \to \infty} \sup_{x \in \bbR^d } \big|\mathscr{h}^{\mathrm{hom}}_{V^{\varepsilon_n}}(x)- \mathscr{h}^{\mathrm{hom}}_V(x)\big|\,.
    \end{equation}
    To show that the right-hand side is zero, we notice that $\mathscr{h}^{\mathrm{hom}}_{V^{\varepsilon_n}}(x)\uparrow \mathscr{h}^{\mathrm{hom}}_{V}(x)$ pointwise by Proposition 1.13 in \cite{port2012brownian}, that $\mathscr{h}^{\mathrm{hom}}_{V}(x)$ is continuous being $V$ strongly regular, and that $\sup_{\bbR^{d}\setminus B(0,R)}|h^{\mathrm{hom}}_{V^{\varepsilon_n}}(x)| \vee |h^{\mathrm{hom}}_{V}(x)|\leq c/R^{d-2}$, which goes to zero as $R \to \infty$. Therefore, one can apply Dini's theorem to conclude that $\mathscr{h}^{\mathrm{hom}}_{V^{\varepsilon_n}} \to \mathscr{h}^{\mathrm{hom}}_{V}$ uniformly.
\end{proof}
\subsubsection{Upper bound for~\eqref{eq:AsymptoticExpectationCF}}\label{subsubsec:UpperBoundExpCF}
We prove the upper bound for~\eqref{eq:AsymptoticExpectationCF} in Proposition~\ref{prop:UBCF} using Lemma \ref{lem:defectharmonicity} below, that provides a uniform control on $\cL^\omega \bbE^{\omega,+}_N[\varphi_\cdot]$ over $V_N\setminus \partial_{\mathrm{in}}V_N$, which actually holds over all $\bbZ^d$ (see Lemma \ref{lem:ExtensionControlLu}). Such quantitative result on the harmonicity defect of $\bbE^{\omega,+}_N[\varphi_\cdot]$ allows us to compare on the bulk of $V_N$ the conditioned field with a weighted average over a sufficiently small ball contained in $V_N$, with a controlled error that vanishes as $N\to \infty$ by properly tuning the radius of the ball.

\begin{lemma}\label{lem:defectharmonicity}
    For every $\delta>0$ there exists a constant $c_9=c_9(\delta)>0$ such that for $\bbQ$-a.a.\ $\omega \in \Omega_{\lambda,\Lambda}$ we can find $N_0(\omega,\delta) \in \bbN$ such that for all $N \geq N_0(\omega,\delta)$
    \begin{equation}\label{eq:defectharmonicity}
        \sup_{x \in V_N\setminus \partial_\mathrm{in} V_N}|\cL^\omega \bbE^{\omega,+}_N[\varphi_\cdot](x)| \leq \frac{c_9}{N^{2-\delta}} \,.
    \end{equation}
\end{lemma}
\begin{proof}
    Let $x\in V_N\setminus \partial_{\mathrm{in}}V_N$. Since on the event $\cW^+_N$ it holds for all $\omega \in \Omega_{\lambda,\Lambda}$ that $m^\omega_x(\varphi) \geq 0$, Lemma~\ref{lem:SuperharmonicConditionedExpectation} implies
    \begin{equation}\label{eq:EstimateAinVNminusInternalBoundary}
        0 \leq -\cL^\omega \bbE^{\omega,+}_N[\varphi_\cdot](x) \leq \sqrt{\frac{2}{\pi \mu^\omega_x}} \bbE^{\omega,+}_N \Big[ \exp \Big( -\frac{\mu^\omega_x}{2} m^\omega_x(\varphi)^2 \Big) \Big]  
    \end{equation}
    for every $\omega \in \Omega_{\lambda,\Lambda}$.
    We now claim that there exists a constant $c_{10}>0$ such that for $\bbQ$-almost every $\omega \in \Omega_{\lambda,\Lambda}$, $N \in \bbN$ and $x \in V_N$
    \begin{equation}\label{eq:EstimateExpExpectationCF}
        \bbE^{\omega,+}_N \Big[ \exp \Big( -\frac{\mu^\omega_x}{2} m^\omega_x(\varphi)^2 \Big) \Big] \leq \exp \Big( -\frac{1}{2 \overline{g}} m^\omega_x(\overline{\varphi}^\omega)^2 + c_{10} |m^\omega_x(\overline{\varphi}^\omega)| \Big) \,,
    \end{equation}
    where $\overline{\varphi}^\omega_\cdot = \bbE^{\omega,+}_N [\varphi_\cdot]$. The proof of~\eqref{eq:EstimateExpExpectationCF} is a straightforward adaptation of the one of Lemma 3.4 in~\cite{deuschel1999entropic} and is in Appendix~\ref{secapp:ProofEstimatesExpectationCF}, see Lemma~\ref{lem:EstimateExpExpectationCF}.
    This yields, using that $\overline{\varphi}^\omega$ is non-negative,
    \begin{equation}
            0 \leq -\cL^\omega \bbE^{\omega,+}_N[\varphi_\cdot](x) \leq \sqrt{\frac{2}{\pi \mu^\omega_x}}\exp \Big( -\frac{1}{2 \overline{g}} m^\omega_x(\overline{\varphi}^\omega)^2  +  c_{10}\, m^\omega_x(\overline{\varphi}^\omega) \Big) \,.
    \end{equation}
    We now exploit the lower bound for the expectation of the conditioned field in Proposition~\ref{prop:LBCF} to get that for every $\delta'>0$ we can find a constant $c_9=c_9(\delta') > 0$ such that for $\bbQ$-a.a.\ $\omega \in \Omega_{\lambda,\Lambda}$ there exists $N_0(\omega,\delta') \in \bbN$ such that for all $N \geq N_0(\omega,\delta')$ 
    \begin{equation}
        \sup_{x\in V_N\setminus \partial_{\mathrm{in}}V_N}|\cL^\omega \bbE^{\omega,+}_N[\varphi_\cdot](x)| \leq c_9 N^{-2(1-\delta')^2}
    \end{equation}
    and assuming w.l.o.g. $\delta \in (0,1)$ we have~\eqref{eq:defectharmonicity} taking $\delta' = 1-\sqrt{1-\delta/2}$.
\end{proof}
We are now ready to prove the upper bound for~\eqref{eq:AsymptoticExpectationCF}.
\begin{prop}\label{prop:UBCF}
    For $\bbQ$-a.e.\ $\omega \in \Omega_{\lambda,\Lambda}$
    \begin{equation}\label{eq:UpperBoundExpectationCF}
        \limsup_{N \to \infty} \sup_{x \in \bbZ^d} \Big[ \frac{\bbE^{\omega,+}_N[\varphi_x]}{\sqrt{4 \overline{g} \log N}} - h^\omega_{V_N}(x)   \Big] \leq 0 \,.
    \end{equation}
\end{prop}
\begin{proof}
    The result is readily implied by the fact that $h^\omega_{V_N}$ is bounded if we show that for every $\delta \in (0,1)$ for $\bbQ$-a.a.\ $\omega \in \Omega_{\lambda,\Lambda}$ we have
    \begin{equation}\label{eq:UpperBoundExpectationCFNotHomogenized}
        \limsup_{N \to \infty} \sup_{x \in \bbZ^d} \Big[ \frac{\bbE^{\omega,+}_N[\varphi_x]}{\sqrt{4 \overline{g} \log N}} - (1+\delta)h^\omega_{V_N}(x)   \Big] \leq 0 \,.
    \end{equation}
    Fix $\delta \in (0,1)$.
    Let $u^\omega_N(x)=\bbE^{\omega,+}_N[\varphi_x]$, $x \in \bbZ^d$. Then by Lemma~\ref{lem:defectharmonicity} for every $\delta'>0$ for $\bbQ$-a.a.\ $\omega \in \Omega_{\lambda,\Lambda}$ there exists $c_9=c_9(\delta')>0$ such that we can find $N_0(\omega,\delta') \in \bbN$ such that for every $N \geq N_0(\omega,\delta')$
    \begin{equation}\label{eq:LabelDefectHarmonicityU}
        \sup_{x\in V_N\setminus \partial_{\mathrm{in}}V_N} |\cL^\omega u^\omega_N(x)| \leq \frac{c_9}{N^{2-\delta'}} \,.
    \end{equation}
    Consider $r = r(N)= N^{1-\delta'}$, so that $B(x,r)\subseteq V_N$ for every $x \in V_N^\varepsilon$, $\varepsilon \in (0,1)$. Let the operator $G^\omega_{B(x,r)}$ be defined as $G^\omega_{B(x,r)}f(z) = \langle g^\omega_{B(x,r)}(z,\cdot), f(\cdot)\rangle_{\mu^\omega}$ for $f \in \bbR^{\bbZ^d}$, $z\in\bbZ^d$ (recall the definition of the Green function for the walk killed upon exiting $B(x,r)$ in~\eqref{eq:KilledGreenFunction}).
    Since $(-\cL^\omega g^\omega_{B(x,r)}(z,\cdot))(y) = \delta_z(y)/\mu^\omega_z$ for all $z\in B(x,r)$, $y \in \bbZ^d$, denoting by $\IND$ the identity function on $\bbZ^d$, we have that $-\cL^\omega G^\omega_{B(x,r)} \IND=1$ on $B(x,r)$.
    It follows by \eqref{eq:LabelDefectHarmonicityU} that for $x \in V_N^\varepsilon$ the function $u^\omega_N - c_9 N^{-2+\delta'} G^\omega_{B(x,r)} \IND$ is subharmonic on $B(x,r)$ and is equal to $u^\omega_N$ on $\partial B(x,r)$, hence for all $x \in V_N^\varepsilon$
    \begin{equation}\label{eq:SubharmonicEstimateU}
        u^\omega_N(x) - c_9 N^{-2+\delta'} G^\omega_{B(x,r)} \IND(x) \leq \langle \eta^\omega_{x,r}, u^\omega_N \rangle \, 
    \end{equation}
    (recall the definition of $\eta^\omega_{x,r}$ in~\eqref{eq:DefinitionEta}).
    Note that $G^\omega_{B(x,r)} \IND(x) \leq \bbE^\omega_x[T_{B(x,r)}]/2d\lambda \leq c r^2$, hence
    \begin{equation}\label{eq:EstimateCEFromAbove}
        \sup_{x \in V_N^\varepsilon} [u^\omega_N(x) - \langle \eta^\omega_{x,r}, u^\omega_N\rangle ]  \leq c'\frac{r^2}{N^{2-\delta'}} \,.
    \end{equation}

    We now want to estimate
    $\langle \eta^\omega_{x,r}, u^\omega_N \rangle = \bbE^{\omega,+}_N [\langle \eta^\omega_{x,r}, \varphi \rangle ]$.
Let $\alpha = 4 \overline{g}(1+\delta)^2$ and $\alpha_N = \sqrt{\alpha \log N}$, and define the probability measure $\bbP^\omega_\alpha = \bbP^\omega T_\psi^{-1}$ with associated expectation $\bbE^\omega_\alpha$, where $\psi :\bbZ^d \rightarrow \bbR$ is the field with constant value $\alpha_N$ and $T_\psi$ is the translation defined in~\eqref{eq:TranslationOperator}. Let $\bbP^{\omega,+}_{N,\alpha} [\,\cdot\,] = \bbP^{\omega}_\alpha [\,\cdot\, | \cW^+_N]$ and denote by $\bbE^{\omega,+}_{N,\alpha}$ the associated expectation. By the inequality~\eqref{eq:FKG3} we get that for every $\varepsilon \in (0,1)$
    \begin{equation}\label{eq:EstimateExpCF2}
        \bbE^{\omega,+}_N[ \langle \eta^\omega_{x,r} , \varphi \rangle ] \leq \bbE^{\omega,+}_{N,\alpha} [ \langle \eta^\omega_{x,r} , \varphi \rangle ] \,, \qquad x \in V_N^\varepsilon \,.
    \end{equation}
    For $\beta >0$ consider $F(\varphi) = \beta \langle \eta^\omega_{x,r} , \varphi \rangle $. By the relative entropy inequality~\eqref{eq:RelEntropyInequal1} applied to $F$
    \begin{equation}
        \begin{aligned}\label{eq:EstimateWithFEntropyIneq}
            \bbE^{\omega,+}_{N,\alpha} [\langle \eta^\omega_{x,r} , \varphi \rangle] &\leq \frac{1}{\beta} (\log \bbE^\omega_\alpha [\exp (\beta \langle \eta^\omega_{x,r} , \varphi \rangle)] - \log \bbP^\omega_\alpha [\cW^+_N] )\\
            & = \frac{\beta}{2} \bbV\mathrm{ar}^\omega[\langle \eta^\omega_{x,r},\varphi\rangle] - \frac{1}{\beta} \log \bbP^\omega_\alpha [\cW^+_N] + \alpha_N \,,
        \end{aligned}
    \end{equation}
    having used in the last step that $\langle \eta^\omega_{x,f} , \varphi \rangle $ under $\bbP^\omega$ is a centered Gaussian.
    Optimizing over $\beta >0$ we have
    \begin{equation}\label{eq:EstimateDotProdEtaU}
        \langle \eta^\omega_{x,r}, u^\omega_N \rangle \leq \sqrt{-2 \bbV\mathrm{ar}^\omega(\langle \eta^\omega_{x,r},\varphi \rangle) \log(\bbP^\omega_\alpha[\cW^+_N])} + \alpha_N \,.
    \end{equation}
    By the same computation as in~\eqref{eq:VarDotProdBound} we get that
    \begin{equation}\label{eq:VarDotProdBound2}
        \bbV\mathrm{ar}^\omega(\langle \eta^\omega_{x,r},\varphi \rangle) \leq c_6 r^{2-d}
    \end{equation}
    where the constant $c_6$ is independent of $\omega \in \Omega_{\lambda,\Lambda}$ and $x \in \bbZ^d$.
    Furthermore, by the FKG inequality~\eqref{eq:FKGproduct}, denoting by $\phi$ the standard normal cumulative distribution function, we have for $N$ large enough
    \begin{equation}
        \begin{aligned}
            \bbP^\omega_\alpha [\cW^+_N] &\geq \prod_{y \in V_N} \bbP^\omega [\varphi_y \geq -\alpha_N] \geq  \phi (  \alpha_N /\sqrt{\,\overline{g}})^{|V_N|} \\
            & \geq \Big[ 1- \frac{c}{\alpha_N / \sqrt{\, \overline{g}}} \exp(\alpha_N^2/2\overline{g}) \Big]^{cN^d} \geq \exp \Big( - \frac{c'}{\sqrt{\log N}} N^{d-\alpha/2\overline{g}} \Big) \,,
        \end{aligned}
    \end{equation}
    having used in the third step that $\phi(x) \geq 1 - \tfrac{1}{\sqrt{2\pi} x} \exp(-x^2/2)$ for $x > 0$.
    Then, keeping in mind that $r=N^{1-\delta'}$, and \eqref{eq:VarDotProdBound2}, we get from~\eqref{eq:EstimateDotProdEtaU} that for $N$ large enough
    \begin{equation}
        \sup_{x \in V_N^\varepsilon}\Big\{ \langle \eta^\omega_{x,r},u^\omega_N\rangle + C'' \Big(\frac{1}{\sqrt{\log N}} N^{d-\alpha/2\overline{g}-(d-2)(1-\delta')} \Big)^{1/2} \Big\} \leq \alpha_N\,.
    \end{equation}
    Then \eqref{eq:EstimateCEFromAbove} yields that there exists $c=c(\delta')>0$ such that for $\bbQ$-a.a.\ $\omega \in \Omega_{\lambda,\Lambda}$ and $N \geq N_0(\omega,\delta')$
    \begin{equation}
         \sup_{x \in V_N^\varepsilon} \Big[ u^\omega_N(x) - C'' \Big(\frac{1}{\sqrt{\log N}} N^{d-\alpha/2\overline{g}-(d-2)(1-\delta')} \Big)^{1/2} - \alpha_N \Big] \leq   \frac{c'}{N^{\delta'}} \,.
    \end{equation}
    Recalling the definition of $\alpha$, the exponent of $N$ in the middle term is equal to $-2(1+\delta)^2 +2(1-\delta') +\delta' d$ and by taking $\delta' \leq 2\delta(\delta+2)/(d-2)$ it is non-positive. Thus, letting $N \to \infty$ we have for $\bbQ$-a.a.\ $\omega \in \Omega_{\lambda,\Lambda}$
    \begin{equation}\label{eq:EstimateExpCF3}
            \limsup_{N\to\infty} \sup_{x\in V_N^\varepsilon} \left[ \bbE^{\omega,+}_N[\varphi_x]- (1+\delta)\sqrt{4\overline{g}\log N}  \right] \leq 0\,,
    \end{equation}
    hence we get~\eqref{eq:UpperBoundExpectationCFNotHomogenized} for $x \in V_N^\varepsilon$.
    When $x \in V_N \setminus V_N^\varepsilon$ we just use the inequality~\eqref{eq:FKG2} to get
    \begin{equation}\label{eq:EstimateExpCF4}
        \begin{aligned}
        &\limsup_{N \to \infty} \sup_{x\in V_N\setminus V_N^\varepsilon}\Big[\bbE^{\omega,+}_N[\varphi_x] - (1+\delta)\sqrt{4 \overline{g} \log N}\Big]\\
        &\quad \leq \limsup_{N \to \infty} \sup_{x\in V_N\setminus V_N^\varepsilon}\Big[ \bbE^\omega[ \varphi_x | \varphi_y \geq 0 , \, \forall y \in V_{2N}]  - (1+\delta)\sqrt{4 \overline{g} \log N}\Big] \leq 0 \,,
    \end{aligned}
    \end{equation}
    where the last inequality is \eqref{eq:EstimateExpCF3} applied to $V_{2N}$.
    For $x \in V_N^c$, since the function $\bbE^{\omega,+}_N[\varphi_\cdot]$ is harmonic on $V_N^c$ (see~\eqref{eq:PreciseExpression-LCF}) and converges to $0$ as $|x|\to \infty$ (see Remark \ref{rem:DecayCondExpectationAtInf}), we have that
    \begin{equation}
        \bbE^{\omega,+}_N[\varphi_x] = \sum_{y \in \partial_{\mathrm{in}}V_N} P^\omega_x[X_{H_{V_N}} = y, H_{V_N} < \infty] \bbE^{\omega,+}_N[\varphi_y] \,.
    \end{equation} 
    Then \eqref{eq:EstimateExpCF4} yields
    \begin{equation}
        \begin{aligned}
        &\limsup_{N\to\infty}\sup_{x \in V_N^c} \Big( \bbE^{\omega,+}_N[\varphi_x] - (1+\delta)\sqrt{4 \overline{g} \log N}h^\omega_{V_N}(x)   \Big) \\
        &\quad=\limsup_{N\to\infty}\sup_{x \in V_N^c}  \sum_{y \in \partial_{\mathrm{in}}V_N} P^\omega_x[X_{H_{V_N}} = y, H_{V_N} < \infty] \Big( \bbE^{\omega,+}_N[\varphi_y] - (1+\delta)\sqrt{4 \overline{g} \log N}   \Big) \\
        &\quad \leq  \limsup_{N\to\infty}  \sup_{y \in \partial_{\mathrm{in}}V_N} \Big[ \bbE^{\omega,+}_N[\varphi_y] - (1+\delta)\sqrt{4 \overline{g} \log N}\Big]\vee 0 \leq 0 \,,
        \end{aligned}
    \end{equation}
    so that we have~\eqref{eq:UpperBoundExpectationCFNotHomogenized} also for $x \in V_N^c$.
\end{proof}

We are now ready to conclude the proof of Theorem~\ref{thm:MainTheorem}, ii).
\begin{proof}[Proof of Theorem~\ref{thm:MainTheorem}, ii)]
By combining the lower bound in Proposition~\ref{prop:LBCF} and the upper bound in Proposition~\ref{prop:UBCF} we have that 
\begin{equation}
    \lim_{N \to \infty} \sup_{x \in \bbZ^d} \Big| \frac{\bbE^{\omega,+}_N[\varphi_x]}{\sqrt{4 \overline{g} \log N}} - h^\omega_{V_N}(x)   \Big| = 0
\end{equation}
for $\bbQ$-a.e.\ $\omega \in \Omega_{\lambda,\Lambda}$. Then, by Lemma \ref{lem:UnifHarmonicPotential} we have that  $\bbQ$-a.s.\ 
\begin{equation}
    \lim_{N \to \infty} \sup_{x \in \bbZ^d} \Big| \frac{\bbE^{\omega,+}_N[\varphi_x]}{\sqrt{4 \overline{g} \log N}} - \mathscr{h}^{\mathrm{hom}}_{V}(x/N)   \Big| = 0 \,,
\end{equation}
ending the proof of Theorem~\ref{thm:MainTheorem}, ii).
\end{proof}

The upper bound in Proposition~\ref{prop:UBCF} can be used to extend the uniform control on $\cL^\omega \bbE^{\omega,+}_N[\varphi_\cdot]$ in Lemma \ref{lem:defectharmonicity} also over the internal boundary of $V_N$.
The proof proceeds analogously to the computations from (3.32) to (3.36) of~\cite{deuschel1999entropic}, but we put it here for consistency.
\begin{lemma}\label{lem:ExtensionControlLu}
    For every $\delta >0$ there exists $c_{11}=c_{11}(\delta)>0$ such that for $\bbQ$-a.a.\ $\omega \in \Omega_{\lambda,\Lambda}$ there exists $N_0(\omega,\delta) \in \bbN$ such that for every $N \geq N_0(\omega,\delta)$
    \begin{equation}\label{eq:defectharmonicityFull}
        \sup_{x\in\bbZ^d}|\cL^\omega \bbE^{\omega,+}_N[\varphi_\cdot](x)| \leq \frac{c_{11}}{N^{2-\delta}} \,.
    \end{equation}
\end{lemma}
    \begin{proof}
    By Lemma~\ref{lem:defectharmonicity} and Lemma~\ref{lem:SuperharmonicConditionedExpectation}, we just have to prove the claim for $x \in \partial_{\mathrm{in}} V_N$. We define the event $F^\omega_x = \{ \varphi \in \bbR^{\bbZ^d} \,:\, m^\omega_x(\varphi) \leq 0 \}$. Then Lemma~\ref{lem:SuperharmonicConditionedExpectation} yields
    \begin{equation}\label{eq:EstimateAOnBoundary}
        |\cL^\omega \bbE^{\omega,+}_N[\varphi_\cdot](x)| \leq c \Big\{ \bbE^{\omega,+}_N \Big[ \exp \Big( - \frac{\mu^\omega_x}{2} m^\omega_x(\varphi)^2 \Big) \Big] + \bbE^{\omega,+}_N\Big[\frac{ \exp(- \mu^\omega_x m^\omega_x(\varphi)^2/2 ) \IND_{F^\omega_x} }{\phi(m^\omega_x(\varphi)\sqrt{\mu^\omega_x})} \Big] \Big\} \,.
    \end{equation}
    Proceeding as in the proof of Lemma \ref{lem:defectharmonicity}, we can show that the first term in the right-hand side is bounded by $c\,N^{-2+\delta}$ for $N$ large enough and $c$ independent of $\omega$. By $\mathrm{e}^{-r^2/2}/\int_r^\infty \mathrm{e}^{-s^2/2}\De s \leq 2+r$ for $r \geq 0$, and H\"older inequality, the second term is bounded by
    \begin{equation}\label{eq:EstimatorSecondTermRegularityCF}
        \bbP^{\omega,+}_N[F^\omega_x]^{1/q} \, \bbE^{\omega,+}_N [ (2+\sqrt{\mu^\omega_x}|m^\omega_x(\varphi)| )^p ]^{1/p}
    \end{equation}
    for every $p,q \geq 1$ such that $1/p+1/q=1$.
    The factor $\bbP^{\omega,+}_N[F^\omega_x]$ can be estimated in the following way for $\bbQ$-a.a.\ $\omega$:
    \begin{equation}
        \begin{aligned}
            \bbP^{\omega,+}_N[F^\omega_x] &\leq \inf_{\beta > 0} \bbP^{\omega,+}_N \Big[ \mathrm{e}^{-\beta (m^\omega_x(\varphi) - \bbE^{\omega,+}_N[m^\omega_x(\varphi)])} \geq \mathrm{e}^{\beta \bbE^{\omega,+}_N[m^\omega_x(\varphi)] } \Big] \\
            &\leq \inf_{\beta > 0} \mathrm{e}^{-\beta \bbE^{\omega,+}_N[m^\omega_x(\varphi)]} \bbE^\omega [\mathrm{e}^{-\beta m^\omega_x(\varphi)} ] \leq \mathrm{e}^{-\bbE^{\omega,+}_N[m^\omega_x(\varphi)]^2/2 \overline{g}} \,,
        \end{aligned}
    \end{equation}
    having used in the second step Markov inequality and Brascamp-Lieb inequality~\eqref{eq:BrascampLieb} with $u(r)=\mathrm{e}^r$.
    Using now the lower bound on the decay rate for the expectation of the conditioned field in Proposition~\ref{prop:LBCF}, we get that for every $\delta' \in (0,1)$, for $\bbQ$-a.a.\ $\omega \in \Omega_{\lambda,\Lambda}$, for $N$ large enough
    \begin{equation}\label{eq:EstimateFQ}
        \sup_{x \in \partial_{\mathrm{in}}V_N}\bbP^{\omega,+}_N [F^\omega_x] \leq c' \frac{1}{N^{2(1-\delta')^2}} \,.
    \end{equation}
    The second factor in~\eqref{eq:EstimateAOnBoundary} can be estimated using the triangle inequality and the Brascamp-Lieb inequality~\eqref{eq:BrascampLieb} with $u(r)=|r|^p$, that yield
    \begin{equation}\label{eq:EstimateMP}
        \bbE^{\omega,+}_N [ (2+\sqrt{\mu^\omega_x}|m^\omega_x(\varphi)| )^p ]^{1/p} \leq c \Big(1+  \sum_{y \,:\, y \sim x} \omega_{x,y} \bbE^\omega [|\varphi_y|^p]^{1/p} +  \sum_{y \,:\, y \sim x} \omega_{x,y} \bbE^{\omega,+}_N[\varphi_y] \Big) \,.
    \end{equation}
    Since the conductances $\omega_{x,y}$ are uniformly bounded away from zero and infinity, the first sum inside the parenthesis in the right-hand side is bounded by a constant $c(p)>0$ independent of $N$ and $x$. By the upper bound on the conditioned expectation in Proposition~\ref{prop:UBCF} we can bound the second sum by $c \sqrt{\log N}$, hence choosing $q<2(1-\delta')/(2-\delta)$ and $\delta'<\delta/2$ we get \eqref{eq:defectharmonicityFull} also for $x \in \partial_{\mathrm{in}}V_N$.
\end{proof}

\section{Entropic repulsion: Pathwise characterization}\label{sec:ProofEntropicRepulsion}
In this section we characterize the pathwise behavior of the conditioned field under $\bbP^{\omega,+}_N$ proving Theorem~\ref{thm:MainTheorem}, iii). We end the section with a thorough discussion on the result and about possible extensions.

For the proof of the pathwise entropic repulsion we follow the strategy of the one of Proposition 2.1 in~\cite{deuschel1999entropic}. For convenience of the reader, in the following we go through the main steps and highlight that the most novel part comes in step 3.

\vspace{0.3\baselineskip}

\noindent \textbf{1.} \emph{Tightness of the conditioned and centered fields sequence}:

One defines the law of the conditioned and centered field
\begin{equation}\label{eq:LawCFRecentered}
    \hat{\bbP}^{\omega}_N = \bbP^{\omega,+}_N T_{\psi}^{-1}, \qquad \text{with }\, \psi = -\bbE^{\omega,+}_N[\varphi_\cdot] \,,
\end{equation}
and shows that the sequence $\{ \hat{\bbP}^{\omega}_N \}_{N \in \bbN}$ is tight, since the $\hat{\bbP}^{\omega}_N$-variance of the field is $\bbQ$-a.s.\ globally bounded uniformly in $N$: 
\begin{equation}
    \sup_{N \in \bbN} \sup_{x\in\bbZ^d} \hat{\bbE}^{\omega}_N[\varphi_x^2] \stackrel{\eqref{eq:BrascampLieb}}{\leq} \sup_{x\in\bbZ^d} \bbE^\omega[\varphi_x^2] \leq \overline{g} \,,
\end{equation}
$\hat{\bbE}^{\omega}_N$ being the expectation associated with $\hat{\bbP}^{\omega}_N$.
\vspace{0.3\baselineskip}

\noindent \textbf{2.} \emph{Any limit point of coincides with $\bbP^\omega$}:

One first shows that every limit point $\hat{\bbP}^{\omega}$ of subsequences of $\{ \hat{\bbP}^{\omega}_N \}_{N \in \bbN}$ satisfies
\begin{equation}\label{eq:Variance_ScalProd_ConditionedCenteredGFF}
    \lim_{n \to \infty} \hat{\bbE}^{\omega} [ \langle f_n , \varphi \rangle^2    ] = 0
\end{equation}
for every sequence $\{f_n\}_{n \in \bbN} \subseteq \bbR^{\bbZ^d} $ such that $\lVert f_n \rVert_1 \leq 1$ and $\lim_{n \to \infty} \lVert f_n \rVert_\infty = 0$, where $\hat{\bbE}^{\omega}$ is the expectation associated with $\hat{\bbP}^{\omega}$. The proof of~\eqref{eq:Variance_ScalProd_ConditionedCenteredGFF}, which is analogous to the one of~\cite[Lemma 2.2]{deuschel1999entropic}, uses in order Brascamp-Lieb~\eqref{eq:BrascampLieb}, H\"older and Minkowski inequality to get that for every $p,q \geq 1$ such that $1/p+1/q=1$
\begin{equation}
    \begin{aligned}
    \sup_{N \in \bbN} \hat{\bbE}^{\omega}_N [ \langle f_n , \varphi \rangle^2 ] &\leq \sum_{x,y \in \bbZ^d} f_n(x) g^\omega(x,y) f_n(y) \\&\leq  \lVert f_n \rVert_q \sum_{x\in \bbZ^d} \lVert g^\omega(x,\cdot) \rVert_p |f_n(x)|  \leq c \lVert f_n \rVert_q  \Vert f_n \rVert_1 \leq  c \lVert f_n \rVert_q\,,
    \end{aligned}
\end{equation}
having used that by~\eqref{eq:QuenchedGFEstimate} one can choose $p$ large enough such that $\sup_{x\in \bbZ^d}\lVert g^\omega(x,\cdot) \rVert_p < c$. Then, by the fact that $\lim_{n \to \infty} \lVert f_n \rVert_\infty = 0$ one gets by interpolation that $\lim_{n \to \infty} \lVert f_n \rVert_q =0$, which yields~\eqref{eq:Variance_ScalProd_ConditionedCenteredGFF}.

To have that any limit point $\hat{\bbP}^{\omega}$ in the weak topology of subsequences of $\{\hat{\bbP}^{\omega}_N\}_{N \in \bbN}$ coincides with $\bbP^\omega$, one first has to check that $\hat{\bbP}^{\omega}$ satisfies the same equations as $\bbP^\omega$ in~\eqref{eq:DLR2_GFF}. To show it, note that for $\hat{\bbP}^{\omega}_N$ it holds that
\begin{equation}\label{DLR_ConditionedRecenteredGFF}
    \begin{aligned}
        \hat{\bbP}^{\omega}_N(\De \phi | \cF_{\{x\}^c})(\varphi) &\propto \exp \bigg\{ - \frac{\mu^\omega_x}{2} [(\phi_x - m^\omega_x(\varphi) - \bbE^{\omega,+}_N[\varphi_x - m^\omega_x(\varphi)]]^2 \bigg\} \\
        &\quad\cdot\IND_{\{\varphi_x \geq -\bbE^{\omega,+}_N[\varphi_x] \}} \De \phi_x \prod_{y \neq x}\delta_{\varphi_y}(\De \phi_y) \,, \qquad \hat{\bbP}^{\omega}_N(\De \varphi)\text{-a.s.}
    \end{aligned}
\end{equation}
By Theorem~\ref{thm:MainTheorem}, ii), we have that $\lim_{N \to \infty}\bbE^{\omega,+}_N[\varphi_x] = \infty$ for every $x \in \bbZ^d$, $\bbQ$-a.s., hence one just has to show that for every $x \in \bbZ^d$, $\bbQ$-a.s.
\begin{equation}\label{eq:CondExpLimitForDLR}
    \lim_{N\to \infty} \bbE^{\omega,+}_N[\varphi_x - m^\omega_x(\varphi)] = \lim_{N\to \infty} (-\cL^\omega \bbE^{\omega,+}_N[\varphi_\cdot](x)) = 0 \,,
\end{equation}
which follows from Lemma~\ref{lem:ExtensionControlLu}.

Hence any convergent subsequence of $\{\hat{\bbP}^{\omega}_N\}_{N\in\bbN}$ has a limit $\hat{\bbP}^{\omega}$ that satisfies the same equations as $\bbP^\omega$ in~\eqref{eq:DLR2_GFF}. It is known that for the Gaussian free field every extremal Gibbs state can be written as $\bbP^\omega T_h^{-1} =: \bbG_h$ for some $\omega$-harmonic $h \in \bbR^{\bbZ^d}$ (see \cite[Section 13.2]{georgii1988gibbs}). Then by~\cite[Theorem 7.26]{georgii1988gibbs} there exists a measure $\nu$ on the space $\cH$ of $\omega$-harmonic functions such that $\hat{\bbP}^\omega = \int_\cH \bbG_h \, \De\nu(h)$. Let $\bbE^{\bbG_h}$ and $\bbV\mathrm{ar}^{\bbG_h}$ denote the expectation and the variance, respectively, associated with $\bbG_h$.
Now for $x \in \bbZ^d$ and considering the sequence of functions $f_n^x(y)=\bbP^\omega_{x}[X_{T_{B(x,n)}}=y]$, $y\in \bbZ^d$, $n \in\bbN$, one has by the $\omega$-harmonicity of $h$ that 
\begin{equation}
    \bbE^{\bbG_h}[\langle f_n^x,\varphi \rangle] = \bbE^\omega[\langle f_n^x,h \rangle] = h(x) \,.
\end{equation}
Hence by Fatou's lemma for every $x \in \bbZ^d$
\begin{equation}
    \begin{aligned}
        0 \stackrel{\eqref{eq:Variance_ScalProd_ConditionedCenteredGFF}}{=}\lim_{n \to \infty} \hat{\bbE}^\omega[\langle f_n^x,\varphi \rangle^2] &= \lim_{n \to \infty} \int_\cH \bbE^{\bbG_h}[\langle f_n^x,\varphi \rangle^2] \, \De\nu(h) \\
        & \geq \int_\cH \Big[ \lim_{n\to\infty}\bbV\mathrm{ar}^{\bbG_h}[\langle f_n^x,\varphi \rangle] + h(x)^2 \Big] \, \De\nu(h) \\
        & = \int_\cH h(x)^2 \, \De\nu(h) \,,
    \end{aligned}
\end{equation}
where $\lim_{n\to\infty} \bbV\mathrm{ar}^{\bbG_h}[\langle f_n^x,\varphi \rangle] = 0$ can be proven by just using the definition of $\bbG_h$ and \eqref{eq:Variance_ScalProd_ConditionedCenteredGFF} again.
Then $\nu$ is concentrated on the constant function $h = 0$, and thus $\hat{\bbP}^\omega = \bbP^\omega$.
\vspace{0.3\baselineskip}

\noindent \textbf{3.} \emph{By the regularity of the field, the recentering is constant over all vertices}:

Note that the results we have obtained in Section~\ref{sec:ProofEntropicRepulsion} until this point hold for almost every $\omega \in \Omega_{\lambda,\Lambda}$. 
If we show that for every $x \in \bbZ^d$
\begin{equation}\label{eq:limit}
    \lim_{N \to \infty} \bbE^{\omega,+}_N[\varphi_x - \varphi_0] = 0\,,
\end{equation}
then one can get that on a subset of $\Omega_{\lambda,\Lambda}$ of full $\bbQ$-measure it is possible to substitute the recentering $\bbE^{\omega,+}_N[\varphi_\cdot]$ by the constant recentering $a_N = \bbE^{\omega,+}_N[\varphi_0]$, which is such that
\begin{equation}
    \lim_{N \to \infty} \frac{a_N}{\sqrt{4 \overline{g} \log N}} = \mathscr{h}^{\mathrm{hom}}_V (0)\,.
\end{equation}
One can achieve this by showing the convergence on local functions and then using that the error $\bbE^{\omega,+}_N[\varphi_x -\varphi_0]$ vanishes in the limit as $N\to \infty$ for any $x$ fixed. This proves Theorem~\ref{thm:MainTheorem}, iii).
We are left proving~\eqref{eq:limit}. In fact, we can even prove a stronger regularity as stated in the following lemma. We remark that the result below does not assume that the conductances are i.i.d., only that Lemma~\ref{lem:ExtensionControlLu} holds (which is true under more mild algebraic mixing conditions on $\bbQ$, see Remark~\ref{rem:Discussion}, 1) below).
\begin{lemma}\label{lem:RegularityCF} There exists $\gamma\in (0,1)$ such that for every $\delta \in (0,1)$ we can find a constant $c_{12}=c_{12}(\delta)>0$ such that for $\bbQ$-a.a.\ $\omega \in \Omega_{\lambda,\Lambda}$ there exists $N_0(\omega,\delta) \in \bbN$ such that for every $N \geq N_0(\omega,\delta)$,
    \begin{equation}
        \big|\bbE^{\omega,+}_N[\varphi_x - \varphi_y]\big| \leq  c_{12}\frac{|x-y|_\infty^{\gamma}}{N^{\gamma-\delta}}\,,\qquad \forall x,\,y\in\bbZ^d.
    \end{equation}
\end{lemma}
\begin{proof} Fix $\delta\in (0,1)$. Clearly, if $x=y$ there is nothing to prove. Let $r>100$ and suppose first that $|x-y|_\infty < r$. Then, by arguing as in the proof of Proposition~\ref{prop:UBCF}, with the help of Lemma~\ref{lem:ExtensionControlLu}, the function $u^\omega_N = \bbE^{\omega, +}_N[\varphi_\cdot]$ is superharmonic while $u^\omega_N - c_{11} N^{-2+\delta} G^\omega_{B(x,2r)} \IND$ is subharmonic in $B(x,2r)$ for all $N$ large enough depending on $\omega$ and $\delta$. Furthermore, both functions are equal to $u_N^\omega$ on $\partial B(x,2r)$.
Therefore, we get 
\begin{equation}
    \begin{aligned}
        \big|u^\omega_N(x) - u^\omega_N(y)\big| &\leq  \Big |\sum_{z\in \partial B(x,2r)}  P^\omega_x[X_{T_{B(x,2r)}}=z] u_N^\omega(z) - \sum_{z\in \partial B(x,2r)}P^\omega_y[X_{T_{B(x,2r)}}=z] u^\omega_N(z) \Big |\\ &
        \quad + \frac{c_{11}}{N^{2-\delta}} \sup_{z\in B(x,r)} G^\omega_{B(x,2r)} \IND(z) \,.
    \end{aligned}
\end{equation}
On the one hand, we have the estimate $\sup_{B(x,r)} G^\omega_{B(x,2r)} \IND\leq cr^2$, because of~\eqref{eq:QuenchedGFEstimate}. On the other hand,
we note that the function $f_N^\omega(y) = \sum_{z\in \partial B(x,2r)} P^\omega_y[X_{T_{B(x,2r)}}=z] u_N^\omega(z)$ is harmonic in $B(x,2r)$ and also bounded by $c \sqrt{\log N}$ for all $N$ large enough by Proposition~\ref{prop:UBCF}. We can then use the H\"older regularity property~\eqref{eq:HoelderRegProperty} to estimate the first term in the right-hand side by
\begin{equation}
     c \Big(\frac{|x-y|_\infty}{r}\Big)^{\tau} \sqrt{\log N} \,.
\end{equation}
This leads to the estimate for all $N$ large enough depending on $\omega$ and $\delta$, and all $x,y \in \bbZ^d$ such that $|x-y|_\infty < r$
\begin{equation}\label{eq:estimate}
    \big|u^\omega_N(x) - u^\omega_N(y)\big| \leq c \frac{|x-y|_\infty^{\tau}}{r^{\tau}}\sqrt{\log N} + c \frac{ r^2}{N^{2-\delta}} \,.
\end{equation}
Note that~\eqref{eq:estimate} holds trivially also if $|x-y|_\infty \geq r$ by choosing a large numerical constant $c > 10 \,\overline{g}$ by Proposition~\ref{prop:UBCF}. Therefore, we can optimize the right-hand side of~\eqref{eq:estimate} by choosing $r = |x-y|_\infty^{\tau/(2+\tau)} N^{2/(2+\tau)}$. Letting $\gamma = 2\tau / (\tau + 2)$, this leads to the estimate, for all $N$ large enough depending on $\omega$ and $\delta$ and all $x,y \in \bbZ^d$
\begin{equation}
    \big|u^\omega_N(x) - u^\omega_N(y)\big| \leq c \frac{|x-y|^\gamma_\infty}{N^\gamma}\big( \sqrt{\log N} + N^{\delta}  \big) \,,
\end{equation}
which concludes the proof being the second factor bounded by $c\,N^{\delta}$ for all $N$ large enough.

\end{proof}

\begin{remark}\label{rem:RegularityCF}
    For i.i.d. conductances one can derive stronger controls on the regularity of the conditioned expectation $\bbE^{\omega,+}_N[\varphi_\cdot]$ analogously to the one in Proposition 3.1 of~\cite{deuschel1999entropic}. It is reasonable that this can be extended to the case of non i.i.d. conductances under sufficient mixing conditions on $\bbQ$.
    It is possible to show that for every $\delta \in (0,1)$ there exists a constant $c=c(\delta)>0$ such that for $\bbQ$-a.a.\ $\omega \in \Omega_{\lambda,\Lambda}$ there exists $N_0(\omega,\delta) \in \bbN$ such that for every $N \geq N_0(\omega,\delta)$
    \begin{equation}
        \big|\bbE^{\omega,+}_N[\varphi_x-\varphi_y] \big|\leq c \frac{|x-y|_\infty}{N^{1-\delta}} \,,\, \text{ for all } x,y \in \bbZ^d\,.
    \end{equation}
    The proof follows the strategy of the one of Proposition 3.1 in~\cite{deuschel1999entropic}, but we sketch it here for completeness.

    Setting 
    \begin{equation}
        A^\omega_N(x) = -\cL^\omega \bbE^{\omega,+}_N[\varphi_\cdot](x) \,, \qquad x \in \bbZ^d\,,
    \end{equation}
    we just have to prove that for $\bbQ$-a.a.\ $\omega \in \Omega_{\lambda,\Lambda}$
    \begin{equation}\label{eq:ClaimRegularityNablaLaplace}
        \sup_{x \in \bbZ^d} |(\nabla_i (-\cL^\omega)^{-1} A^\omega_N)(x)| \leq c \frac{1}{N^{1-\delta}} \,, \, \text{ for all } i \in \{1,\dots,d\} \,,
    \end{equation} 
    $\nabla_i$ being the discrete gradient in the $i$-th direction (i.e. $\nabla_i f(x) = f(x+e_i) - f(x)$ for $f \in \bbR^{\bbZ^d}$ and $e_i \in \bbZ^d$ whose $i$-th component is $1$ and all the others are $0$).
    Let $K^\omega_i(x,y) = g^\omega(x+e_i,y) - g^\omega(x,y)$ be the kernel of the operator $\nabla_i (-\cL^\omega)^{-1}$ on $\ell^2(\mu).$ Then the key point for the proof is to just have a control on such gradient of the Green function, namely to show that there exist two positive constants $s=s(\delta)$, $c_{13}=c_{13}(\delta)$ such that, setting $r=r(N)=c_{13}\,(2\log N)^{1/s}$, for $\bbQ$-a.a.\ $\omega \in \Omega_{\lambda,\Lambda}$ and for all $N$ large enough
    \begin{equation}\label{eq:ClaimKernel}
        \sup_{x \in \bbZ^d}\sup_{\substack{y \in V_{2N} \,:\,\\ |y-x| \geq r}} |K^\omega_i(x,y)| \leq \frac{c_{13}}{|x-y|^{d-1-\delta/2}}\,, \, \text{ for all } i \in \{1,\dots,d\} \,.
    \end{equation}
    Then, by Lemma~\ref{lem:ExtensionControlLu} and~\eqref{eq:ClaimKernel} it follows
    \begin{equation}
        \begin{aligned}
        \sup_{x\in\bbZ^d}|(&\nabla_i (-\cL^\omega)^{-1} A^\omega_N)(x)|  \leq \sum_{y \in V_N} |K^\omega_i(x,y)| |A^\omega_N(y)| \mu^\omega_y\\
        &\leq \frac{c_{11}}{N^{2-\delta'}} \Big\{ \sum_{y \in V_N \cap B(x,r) }[|g^\omega(x+e_i,y)|+|g^\omega(x,y)|]\mu^\omega_y  + \sum_{y \in V_N\setminus B(x,r)} \frac{c_{13}}{|x-y|^{d-1-\delta/2}} \Big\}\,,
        \end{aligned}
    \end{equation}
    having used in the first step that $A^\omega_N(y) = 0$ for $y \in V_N^c$ by Lemma~\ref{lem:SuperharmonicConditionedExpectation}.
    The first sum inside the curly bracket is less or equal to $c\bar{g}r^d$, $\bbQ$-a.s., while we can estimate the second one by $c\, N^{1+\delta/2}$, so that taking $\delta' < \delta/2$ we get
    \begin{equation}
        |(\nabla_i (-\cL^\omega)^{-1} A^\omega_N)(x)| \leq c\,N^{-2+\delta'}[r^d + N^{1+\delta/2}] \leq C N^{\delta-1}
    \end{equation}
    which is~\eqref{eq:ClaimRegularityNablaLaplace}.
    We now prove~\eqref{eq:ClaimKernel}.
    By Theorem 1.2 of~\cite{Dario2019QuantitativeHO}, we can find two constants $s,c_{13}>0$ such that for every $y \in \bbZ^d$ there exists a random variable $\cM_{\mathrm{ell},\delta/2}(y)$  such that $\bbQ[\cM_{\mathrm{ell},\delta/2}(y) > r] \leq c_{13}\exp(-r^s/c_{13})$, and for every $x\in \bbZ^d$ such that $|y-x|\geq \cM_{\mathrm{ell},\delta/2}(y)$ it holds that
    \begin{equation}\label{eq:homogenizationGreenFunction}
        |g^\omega(x,y)-\mathscr{g}(x-y)|\leq \frac{c_{13}}{|x-y|^{d-1-\delta/2}} \,,
    \end{equation}
    where $\mathscr{g}(\cdot)$ is the Green function for the Brownian motion, with deterministic covariance matrix, the scaled random walk $(N^{-1} X_{tN^2})_{t \geq 0}$ converges to. By setting $\cM = \max_{y\in V_{2N}} \cM_{\mathrm{ell},\delta/2}(y)$, it follows by a union bound that $\bbQ[\cM > r] \leq C (2N)^d \exp(-r^s/c_{13})$, hence by Borel-Cantelli lemma and the definition of $r$, for $\bbQ$-a.a.\ $\omega$ there exists $N_0(\omega,\delta)$ such that for every $N \geq N_0(\omega,\delta)$ we have~\eqref{eq:homogenizationGreenFunction} for all $y \in V_{2N}$ and $x \in \bbZ^d$ such that $|x-y| \geq r$. We can now conclude~\eqref{eq:ClaimKernel} by replacing in the definition of $K^\omega_i(x,y)$ the Green function $g^\omega(\cdot,\cdot)$ by $\mathscr{g}(\cdot-\cdot)$ plus the error term in~\eqref{eq:homogenizationGreenFunction}, and by using that $\nabla_i \mathscr{g}(x-y)$ is bounded by $c/|x-y|^{d-1}$ for some constant $c>0$.
\end{remark}

\begin{remark}\label{rem:corollaryglobal}
    With the same techniques as in~\cite[Corollary 3.2]{deuschel1999entropic}, in the case of i.i.d.\ conductances, one can use the regularity result in Remark \ref{rem:RegularityCF} to obtain strong concentration properties for the conditioned field $\bbP^{\omega,+}_N$.
    Formally,  by Remark~\ref{rem:RegularityCF}, the Brascamp-Lieb inequality~\eqref{eq:BrascampLieb} with $u(r)=\exp(r)$, and Chebyshev's inequality, one can get that for every $z\in \bbR^d$, $\beta < 1/\overline{g}$ and $\delta >0$ there exists $N_0=N_0(\omega,\delta) \in \bbN$ such that
    \begin{equation}
        \sup_{N \geq N_0}\sup_{x \,:\,|x-Nz| \leq N^{1-\delta}} \bbE^{\omega,+}_N \Big\{\exp\Big[\frac{\beta}{2}(\varphi_x-\bbE^{\omega,+}_N[\varphi_{\lfloor Nz \rfloor}])^2\Big]\Big\} < \infty \,.
    \end{equation}   
    It is also possible to obtain a similar result employing Lemma~\ref{lem:RegularityCF} instead of the stronger one in Remark~\ref{rem:RegularityCF}, but with a weaker control on the concentration scale.
\end{remark}

\begin{remark} \label{rem:Discussion}
\begin{enumerate} 
    \item  The assumption of independence of conductances can be relaxed. The only place where we use it is in the proof of Proposition~\ref{prop:PositivePortionOfHighVariance} to claim that $g^\omega_{B(y,R)}(y,y)$, $y \in B_{z,x}$, are mutually independent random variables for $x \in \square_\ell$ and $\ell=4R$. However, Proposition~\ref{prop:PositivePortionOfHighVariance} still holds assuming an $\alpha$-mixing condition for the conductances, with algebraic decorrelation rate given by $\Phi(x)=x^{-\alpha}$, for any $\alpha > d(d-1)/2$, namely assuming that
    \begin{equation}
        \begin{minipage}{0.6\linewidth}
                \label{eq:AlphaMixingCondition}
                $| \bbQ[ A \cap B] - \bbQ[A]\bbQ[B] | \leq \Phi(d_\infty(U,V)),$ \\
                for all Borel sets $U, V \subseteq \bbZ^d$ and every $A \in \cQ_U, B \in \cQ_V \,,$
        \end{minipage}
    \end{equation}
    where $\cQ_K = \sigma (\omega_{x,y} \, : \, x,y \in K \cup \partial K)$, $K \subseteq \bbZ^d$ (see Appendix~\ref{secapp:ProofPositivePortionOfHighVariance} for the proof). We do not claim optimality for this set of values for $\alpha$ and we do not pursue here a sharp control on the mixing properties.
    However, it is reasonable to expect, at least based on our method of proof, that Theorem~\ref{thm:MainTheorem} needs some decorrelation assumption on the conductances to quantify ergodicity of the environment. In fact, the repulsion is due to the fact that the maximum of the local fields $\psi_z^\omega$, $z\in \cC_N$ at scale $N^{2/d}$ is concentrated around $\sqrt{4 \overline{g} \log N}$, for a subcollection of points $\cS\subseteq \cC_N$ that need to be uniformly spread in $V_N$, (see \eqref{eq:CoveringProperty}).

    \item For general mixing environments, the constant $\overline{g}$ is not easily computable. However, when the conductances $\omega_{x,y}$ are i.i.d. and such that the essential infimum of their support is $\lambda$, then $\overline{g}$ can be evaluated and it holds that
    \begin{equation}\label{eq:ValueOverlineg}
        \overline{g} = \frac{g(0,0)}{\lambda}\,,
    \end{equation}
    where $g(\cdot,\cdot)$ is the Green function of the simple random walk on $\bbZ^d$, that is, the one defined in \eqref{eq:GreenFunction} with conductances all equal to one.  Interestingly, only the lower bound on the conductances enters in the value of $\overline{g}$. Let us quickly sketch~\eqref{eq:ValueOverlineg}. For given $N>0$ and $\delta>0$ there is a $\bbQ$-positive probability that all the conductances in $B(0,N)$ belong to $[\lambda,\lambda+\delta]$. As such, with positive $\bbQ$-probability,
    \begin{equation}\label{eq:bargiid}
        g^\omega(0,0) \geq g^{\omega}_{B(0,N)}(0,0) \geq \frac{g_{B(0,N)}(0,0)}{(\lambda+\delta)} \geq \frac{g(0,0)}{(\lambda+\delta)} - \frac{c}{N^{d-2}}\,,
    \end{equation}
    where we used Rayleigh's monotonicity principle in the second inequality and \eqref{eq:QuenchedGFEstimate} in the last one. By definition of essential supremum, it follows that $\overline{g} \geq g(0,0)/(\lambda+\delta) - c/N^{d-2}$ for every $\delta>0$ and all $N>0$, which yields $\overline{g} \geq g(0,0)/\lambda$. The reverse inequality follows directly from Rayleigh's monotonicity principle.

    Another interesting consequence of~\eqref{eq:ValueOverlineg} is that it is tempting to say that Theorem~\ref{thm:MainTheorem} might hold for possibly unbounded conductances, provided that they are i.i.d., bounded from below by a positive constant $\lambda$, and possibly satisfy some moment conditions. Together with the homogenization results in~\cite{barlow2010invariance}, we might expect that the proof of Theorem~\ref{thm:MainTheorem} could be adapted to this case. This will be the subject of future work.

    Strikingly, in view of~\eqref{eq:bargiid} if $\lambda = 0$, even keeping $\bbQ[\omega_{x,y} >0]=1$, $x\sim y$, then, despite the fact that a quenched invariance principle might still hold with a non-degenerate covariance matrix under mild moment conditions on the conductances (see for instance~\cite{andres2013invariance, andres2015invariance, andres2018quenched, barlow2010invariance, bella2020quenched, biskup2018limit}), the maximum of the field and the hard wall repulsion are expected to have a markedly different behavior. We discuss this in the next point for the manifestly non-elliptic case of a Gaussian free field on the supercritical percolation cluster.

    \item Recently, there has been a growing interest for the study of the Gaussian free field on the supercritical percolation cluster. For example, in~\cite{andres2025scaling} the authors prove a scaling limit for the field on the infinite cluster of supercritical Bernoulli bond percolation on $\bbZ^d$, $d\geq 2$, and~\cite{schweiger2024} look at the maximum of the field in dimension $d=2$. In the latter, the authors show that if $p$ is sufficiently close to one, then the maximum of the field with zero boundary conditions, properly centered and rescaled, converges to a randomly shifted Gumbel distribution for almost all realizations of the cluster. More precisely, if $V_N = [0,N-1]^2 \cap \bbZ^2$, $\cC_\infty$ is the infinite cluster of supercritical Bernoulli bond percolation on $\bbZ^2$ with parameter $p$ and law $\bbP_p$, and $\varphi^{\cC_\infty, N}_\cdot$ is the Gaussian free field with zero boundary conditions on $V_N\cap \cC_\infty$, then for $p$ sufficiently close to one there exists a constant $a_p$ such that for $\bbP_p$-a.e.\ realization of the cluster, the sequence of random variables
    \begin{equation}
       \max_{x\in V_N \cap \cC_\infty} \varphi^{\cC_\infty,N}_x - \sqrt{\frac{2}{\pi a_p}}\Big(\log N - \frac{3}{8}\log\log N\Big)
    \end{equation}
    converges in distribution to a deterministic limit, given by a randomly shifted Gumbel distribution. Notably, the constant $a_p$ is the effective diffusivity of the Brownian motion to which the random walk on the infinite cluster converges.
    
    When $d\geq3$, the situation is quite different to dimension $d=2$ but also to the case with uniformly elliptic conductances. On the one hand, the correlations still decay algebraically at mesoscopic scales, like in the uniformly elliptic case and one can define an infinite volume field. On the other hand the variances of the field are not uniformly bounded from above. It follows from~\cite{abe2015effective} that for $d\geq 3$ there exist constants $c_1,c_2>0$ such that $\bbP_p$-a.s.\ for all $N$ large enough
   \begin{equation}
        c_1 \log N \leq \sup_{x\in V_N\cap \cC_\infty} g^{\cC_\infty}(x,x) \leq c_2 \log N\,,
    \end{equation}
    which leads to the conclusion that the maximum of the field in $V_N \cap \cC_\infty$ is of order $\log N$ rather than of order $\sqrt{\log N}$ (see Corollary 1.7 in~\cite{abe2015effective} for a slightly different model). This is due to the substantial presence of open tubes of length $\log N$ in $V_N\cap \cC_\infty$. In light of this, it is natural to expect that the entropic repulsion for the field on the infinite cluster is of order $\log N$ rather than $\sqrt{\log N}$. More precisely, if we look at the hard wall event $\cW^+_{\cC_\infty,N} = \{\varphi_x\geq 0,\, \forall x\in V_N\cap \cC_\infty\}$ when $d\geq 3$, we expect that for some constants $\hat{c},\overline{c}>0$ the following asymptotic behavior hold for any $p>p_c$ and $\bbP_p$-a.e.\ realization of the cluster 
    \begin{equation}\label{eq:conjecture}
        e^{-\hat{c} N^{d-2} (\log N)^2 (1 + o(1))} \leq  \bbP^{\cC_\infty}[\cW^+_{\cC_\infty,N}] \leq e^{-\overline{c} N^{d-2} (\log N)^2(1+o(1))}\,.
    \end{equation}
    Proving that $\hat{c} = \overline{c}$, and from there getting a pathwise behavior for the conditioned field appears to be a more challenging task and it is the subject of ongoing investigation. Our method of proof, together with the results in~\cite{chiarini2025solidification} and the quantitative homogenization estimates in~\cite{Dario2019QuantitativeHO}, seems a good starting point to address the question. However, it is possible that sharp controls on the tails of $g^{\cC_\infty}(0,0)$ under $\bbP_p[\,\cdot\,|0\in \cC_\infty]$ are needed, which might be difficult to get in the whole supercritical regime $p\in(p_c,1]$. Some of this estimates are available for $d=2$ in \cite{schweiger2024}.

    \item The article~\cite{schweiger2024} also investigates the maximum of the Gaussian free field in dimension $d=2$ when the conductances are i.i.d.\ and uniformly bounded away from zero and infinity. The authors prove that the maximum of the field with zero boundary conditions, properly centered and rescaled, converges to a randomly shifted Gumbel distribution. To be more precise, if $V_N = [0,N-1]^2 \cap \bbZ^2$, and $\varphi^N_\cdot$ is the Gaussian free field with random conductances and with zero boundary conditions on $V_N$, then the sequence of random variables
  \begin{equation}
       \max_{x\in V_N } \varphi^N_x - \sqrt{\frac{2}{\pi \overline{a}}}\Big(\log N - \frac{3}{8}\log\log N\Big)
    \end{equation}
    converges in distribution to a deterministic limit, given by a randomly shifted Gumbel distribution. Also here the constant $\overline{a}$ is the effective diffusivity of the Brownian motion to which the random walk with i.i.d.\ uniformly elliptic conductances converges. In other words, the relevant proportionality constant is captured by the homogenized behavior of the random walk. This is in stark contrast with the first order behavior of the maximum in dimension $d\geq 3$ (see Remark~\ref{rem:MaxGFFErgodic}), where the relevant constant $\overline{g}$ is not captured by the homogenized behavior of the random walk. In particular, similarly to~\cite[Remark 3.7]{chiarini2021disconnection} and unlike what happens in $d=2$, two environments with the same effective diffusivity can have different values of $\overline{g}$ and therefore different first order behavior for the maximum of the field and of the entropic repulsion effect.
    For example, consider the special case in which the conductances $\omega\in \Omega_{\lambda,\Lambda}$ are i.i.d.\ under the environment measure $\bbQ$ and such that the infimum of the support of $\omega_{0,x}$, $x\sim 0$, is $\lambda$. It is well-known that the constant-speed random walk on $(\bbZ^d,\bbE_d ,\omega)$ with diffusive scaling has a $\bbQ$-a.s.\ scaling limit given by a Brownian motion with covariance matrix $a^{\rm hom} = \sigma^2 \mathrm{Id}$, where $\mathrm{Id}$ is the $(d \times d)$-identity matrix and $\sigma^2 > 0$. By considering constant conductances $\omega(\sigma) \equiv \frac{1}{2} \sigma^2$ on $(\bbZ^d,\bbE_d)$, it is straightforward to show that the corresponding random walk on the weighted graph $(\bbZ^d,\bbE_d,\omega(\sigma))$ with diffusive scaling has the same scaling limit. However, unless the random conductances under $\bbQ$ are deterministic, one has, (see  for instance~\cite[Section 4]{biskup2011recent})
    \begin{equation}
        \lambda < \frac{1}{\bbE[\omega_{0,x}^{-1}]}\leq \frac{\sigma^2}{2} \leq \bbE[\omega_{0,x}] < \Lambda\,,
    \end{equation}
    where $x\sim 0$.
    Thus, we get that $\overline{g}(\sigma) = 2g(0,0)/\sigma^2$ for the constant conductances, while in view of~\eqref{eq:ValueOverlineg}, for the random conductances $\overline{g} = g(0,0)/\lambda$  which is strictly larger than $2g(0,0)/\sigma^2$. Phenomenologically, this shows that while the scaling limit of the Gaussian free field is the same for the two environments (cf.\ \cite{andres2025scaling,chiariniRuszel2024stochastic}), the bond disorder leads to a larger value for the maximum and to a stronger entropic repulsion effect.
    \end{enumerate}
\end{remark}

\appendix
\section{Proof of Proposition~\ref{prop:PositivePortionOfHighVariance} with \texorpdfstring{$\alpha$}{}-mixing assumptions on conductances}
\label{secapp:ProofPositivePortionOfHighVariance}
\begin{prop}
    \label{prop:AlphaMixPositiveOfPortionHighVariance}
    Assume that $\bbQ$ satisfies the $\alpha$-mixing condition~\eqref{eq:AlphaMixingCondition} with $\alpha > d(d-1)/2$. Then, for every $\varepsilon >0$ there exists $q>0$ that for $\bbQ$-a.e.\ $\omega \in \Omega_{\lambda,\Lambda}$
    \begin{equation}
    \label{eq:AlphaMixPositivePortionOfHighVariance}
        \lim_{N \to \infty} \inf_{z \in \cC_N}\frac{1}{|B_z|}\sum_{y \in B_z} \IND_{\{ g^\omega_{U_z}(y,y) \geq \overline{g} - \varepsilon \}} \geq q \,.
    \end{equation}
\end{prop}
\begin{proof}
    Fix $\varepsilon >0$. As for Proposition~\ref{prop:PositivePortionOfHighVariance}, we fix $R=R(\varepsilon)$ such that $g^\omega_{B(y,R)}(y,y) \geq g^\omega(y,y) - \varepsilon/2$.
        For $y \in \bbZ^d$ let
    \begin{equation}
        \label{eq:DefqQ}
            q := \bbQ[g^\omega(y,y) \geq \overline{g} - \varepsilon/2 ]/2  \,, \qquad  Q := \bbQ[g^\omega_{B(y,R)}(y,y) \geq \overline{g}-\varepsilon] \,,
    \end{equation}
    which are independent of $y$ by the stationarity of $\bbQ$ and satisfy $0 < 2q \leq Q$.
    
    For any $4R < \ell  \leq L$ denote the box of size $\ell$ by $\square_\ell = [0,\ell)^d \cap \bbZ^d$, and for $z \in \cC_N$ we consider the grid $B_{z,x} = x+(\ell\bbZ^d \cap B_z)$. For $y \in \bbZ^d$ we define the centered variable $X_y = \IND_{\{ g^\omega_{B(y,R)}(y,y) \geq \overline{g} -\varepsilon \}} - Q$.  
    Then 
    \begin{equation}
      \sum_{y \in B_z} X_y = \sum_{x \in \square_\ell} \sum_{y \in B_{z,x}} X_y \,.
    \end{equation}
    As for Proposition~\ref{prop:PositivePortionOfHighVariance}, for $z \in \bbZ^d$ we define the events
    \begin{equation}
       J^\omega_N(z) = \Big\{ \frac{1}{|B_z|} \sum_{y \in B_z} \IND_{ \{g^\omega_{B(y,R)}(y,y) \geq \overline{g}-\varepsilon \} } < q  \Big\} \,,
    \end{equation}
    whose $\bbQ$-probability is independent of $z$. We then have that
    \begin{equation}
        \label{eq:EstimateQJWithX}
        \bbQ[J^\omega_N(z)] \leq \bbQ \Big[\sum_{y \in B_z}X_y < - \tfrac{Q}{2} |B_z| \Big] \,.
    \end{equation}
    Let $k \in \bbN$. Denoting by $\bbE$ the expectation associated with $\bbQ$, by Markov's and Jensen's inequalities
    \begin{equation}
        \label{eq:EstimateWithMoment}
        \bbQ \Big[\sum_{y \in B_z}X_y < -\tfrac{Q}{2} |B_z| \Big] \leq \frac{1}{{(QL^d/2)^{2k}}} \, \ell^{d(2k-1)} \sum_{x \in \square_\ell} \bbE\Big[ \Big(\sum_{y \in B_{z,x}}X_y \Big)^{2k} \Big] \,.
    \end{equation}
    We now want to estimate
    \begin{equation}
        \bbE\Big[ \Big(\sum_{y \in B_{z,x}}X_y \Big)^{2k} \Big] = \sum_{y_1, \dots, y_{2k} \in B_{z,x}} \bbE \Big[ \prod_{i=1}^{2k} X_{y_i} \Big] \,.
    \end{equation}
    We follow an analogous procedure to Section 3.2.4 of~\cite{Armstrong2022-ak}. For $z \in \cC_N$ and $x \in \square_\ell$ we define the set
    \begin{equation}
        H^{z,x}_j = \{ (y_1, \dots, y_{2k}) \in  B_{z,x}^{2k} \,:\, |\{ y_1, \dots, y_{2k} \}| = j \}
    \end{equation}
    composed of $2k$-tuples of points with exactly $j$ different entries, $j \in \bbN$. The set of $j$ distinct entries of $y \in H^{z,x}_j$ is denoted by $y_{\mathrm{dist}}$, and the multiplicity of $v \in y_{\mathrm{dist}}$ in $y$ is indicated by $q_y(v)$.
    The $\alpha$-mixing condition~\eqref{eq:AlphaMixingCondition} yields that for every $y =( y_1,\dots,y_{2k}) \in H^{x,z}_j$
    \begin{equation}
        \bbE \Big[ \prod_{i=1}^{2k} X_{y_i} \Big] = \bbE \Big[ \prod_{v \in y_{\mathrm{dist}}} X_v^{q_y(v)} \Big] \leq \prod_{v \in y_{\mathrm{dist}}} \bbE \Big[ X_v^{q_y(v)} \Big] +4(j-1) \Phi(\ell/2) \,,
    \end{equation}
    where the last step can be proven using induction on the covariance estimate $\bbE[Y_1 Y_2] - \bbE[Y_1]\bbE[Y_2] \leq 4 \Phi(\ell)$, which holds for every pair of $[-1,1]$-valued random variables $Y_1,Y_2$ that are measurable with respect to $\cQ_U$ and $\cQ_V$ respectively, with $d_\infty(U,V) \geq \ell$ (see Lemma A.1 of~\cite{Armstrong2016-vs}). In fact $X_y$ is $\cQ_{B(y,R)}$-measurable, and for $v,w \in B_{z,x}$ with $v \neq w$ we have $d_\infty(B(v,R),B(w,R)) \geq \ell - 2R \geq \ell/2$. 
    Then by the arguments on page 90 of~\cite{Armstrong2022-ak} we get
    \begin{equation}
        \bbE\Big[ \Big(\sum_{y \in B_{z,x}}X_y \Big)^{2k} \Big] \leq c_k \big(|B_{z,x}|^k + |B_{z,x}|^{2k} \Phi(\ell/2) \big) \,.
    \end{equation}
    Since for every $z \in \cC_N$ and $x \in \square_\ell$ we have $c (L/\ell)^d \leq |B_{z,x}| \leq c' (L/\ell)^d$, setting $n=L/\ell$ we get by~\eqref{eq:EstimateWithMoment} that
    \begin{equation}
        \bbQ \Big[\sum_{y \in B_z}X_y \leq - \tfrac{Q}{2} |B_z| \Big] \leq c'_{k}\big( n^{-dk} + \Phi(\ell/2) \big)\,.
    \end{equation}
    Then~\eqref{eq:EstimateQJWithX} yields
    \begin{equation}
        \sum_{N \in \bbN} \bbQ \Big[ \bigcup_{z \in \cC_N} J^\omega_N(z) \Big] \leq c'_{k} \sum_{N \in \bbN} |\cC_N|\big(n^{-dk} + \Phi(\ell/2)\big) \,.
    \end{equation}
    Taking $\ell= L^{1-\delta}$ for a $\delta \in (0,1)$, by~\eqref{eq:CardinalityCN} the above quantity is smaller or equal to
    \begin{equation}
        c''_{k}\sum_{N \in \bbN} N^d L^{-d(1 + k \delta)}+ c''_{k}\sum_{N \in \bbN} N^d L^{-d - (1-\delta)\alpha}
    \end{equation}
    and by~\eqref{eq:DefinitionL}, since $\alpha > d(d-1)/2$, the second term converges by taking $\delta$ small enough, while the first one converges by fixing $k$ (depending on $\delta$) large enough.
    Finally, by Borel-Cantelli lemma we get the claim~\eqref{eq:AlphaMixPositivePortionOfHighVariance}.
\end{proof}
\section{Proofs on the conditional expectation of the field}\label{secapp:ProofEstimatesExpectationCF}
In this section we give a proof of Lemma~\ref{lem:SuperharmonicConditionedExpectation} and of the estimate in~\eqref{eq:EstimateExpExpectationCF}, the latter being in Lemma~\ref{lem:EstimateExpExpectationCF} below. These proofs are an adaptation of the ones of Lemma 2.3 and Lemma 3.4 in~\cite{deuschel1999entropic}, respectively, to our inhomogeneous setting.
\begin{proof}[Proof of Lemma~\ref{lem:SuperharmonicConditionedExpectation}]
    For $\omega \in \Omega_{\lambda,\Lambda}$, $x \in \bbZ^d$ and $N \in \bbN$ we have
    \begin{equation}
        \begin{aligned}
            \bbE^{\omega,+}_N[ \varphi_x - m^\omega_x(\varphi) ] &= \bbE^{\omega,+}_N \Big[ \bbE^{\omega,+}_N[ \varphi_x - m^\omega_x(\varphi) | \cF_{\{x\}^c}] \Big] \\
            &= 
            \begin{cases}
                0 \,, & \mbox{for } x \in V_N^c \,, \\
                \bbE^{\omega,+}_N \Big[ \bbE^\omega [ (\varphi_x - m^\omega_x(\varphi))\IND_{\cW^+_N} | \cF_{\{x\}^c} ] / \bbP^\omega [\cW^+_N | \cF_{\{x\}^c}] \Big] \,, & \mbox{for }x \in V_N \,,
            \end{cases}
        \end{aligned}
    \end{equation}
    where the case for $x \in V_N^c$ follows from the fact that $m^\omega_x(\varphi)$ is $\cF_{\{x\}^c}$-measurable and from the equations~\eqref{eq:DLR2_GFF}. By the same equations, we have that for $x \in V_N$
    \begin{equation}
        \begin{aligned}
        \bbE^\omega [ (\varphi_x - m^\omega_x(\varphi))\IND_{\cW^+_N} | \cF_{\{x\}^c} ] &= \frac{1}{\sqrt{2\pi/\mu^\omega_x}} \int_0^\infty (\zeta - m^\omega_x(\varphi)) \exp\Big( -\frac{\mu^\omega_x}{2} (\zeta-m^\omega_x(\varphi))^2 \Big) \De\zeta \\
        &= \frac{1}{\sqrt{2\pi \mu^\omega_x}} \exp \Big( -\frac{\mu^\omega_x}{2}m^\omega_x(\varphi)^2 \Big)
        \end{aligned}
    \end{equation}
    and under $\bbP^{\omega,+}_N$
    \begin{equation}
        \bbP^\omega [\cW^+_N | \cF_{\{x\}^c}] = \bbP^\omega [\varphi_x - m^\omega_x(\varphi) \geq -m^\omega_x(\varphi)  | \cF_{\{x\}^c}] = \phi(m^\omega_x(\varphi)\sqrt{\mu^\omega_x}) \,.
    \end{equation}
\end{proof}
\begin{lemma}\label{lem:EstimateExpExpectationCF}
    There exists a constant $c_{10}>0$ such that for $\bbQ$-almost all $\omega \in \Omega_{\lambda,\Lambda}$, $N \in \bbN$ and $x \in V_N$
    \begin{equation}
        \bbE^{\omega,+}_N \Big[ \exp \Big( - \frac{\mu^\omega_x}{2} m^\omega_x(\varphi)^2 \Big) \Big] \leq \exp \Big( -\frac{1}{2\overline{g}} m^\omega_x(\overline{\varphi}^\omega)^2 + c_{10} |m^\omega_x(\overline{\varphi}^\omega)| \Big) \,,
    \end{equation}
    where $\overline{\varphi}^\omega_y = \bbE^{\omega,+}_N[\varphi_y]$, $y \in \bbZ^d$.
\end{lemma}
\begin{proof}
    By the definition of $\hat{\bbP}^\omega_N$ as in~\eqref{eq:LawCFRecentered}, we have that
    \begin{equation}
        \begin{aligned}
            &\bbE^{\omega,+}_N \Big[ \exp \Big( - \frac{\mu^\omega_x}{2} m^\omega_x(\varphi)^2 \Big) \Big] \\
            & \quad = \bbE^{\omega,+}_N \Big[ \exp \Big( - \frac{\mu^\omega_x}{2} m^\omega_x(\varphi - \overline{\varphi}^\omega)^2 - \mu^\omega_x m^\omega_x (\varphi - \overline{\varphi}^\omega) m^\omega_x (\overline{\varphi}^\omega) \Big) \Big] \exp \Big( -\frac{\mu^\omega_x}{2}m^\omega_x(\overline{\varphi}^\omega)^2 \Big) \\
            & \quad = \hat{\bbE}^\omega_N \Big[ \exp \Big( - \frac{\mu^\omega_x}{2}m^\omega_x(\varphi)^2 - \mu^\omega_x m^\omega_x(\varphi) m^\omega_x(\overline{\varphi}^\omega) \Big) \Big] \exp \Big( -\frac{\mu^\omega_x}{2} m^\omega_x
            (\overline{\varphi}^\omega)^2 \Big) \,.
        \end{aligned}
    \end{equation}
    Define $\widetilde{\bbP}^\omega_N (\De \varphi) \propto \exp ( -\mu^\omega_x m^\omega_x(\varphi)^2 /2 ) \hat{\bbP}^\omega_N(\De \varphi)$ and $\widetilde{\varphi}^\omega_\cdot = \widetilde{\bbE}^\omega_N [\varphi_\cdot]$, where $\widetilde{\bbE}^\omega_N$ is the expectation associated with $\widetilde{\bbP}^\omega_N$. Then
    \begin{equation}\label{eq:ExpCFEstimate}
        \begin{aligned}
            \bbE^{\omega,+}_N \Big[\exp \Big( -\frac{\mu^\omega_x}{2} m^\omega_x(\varphi)^2 \Big)\Big] &= \widetilde{\bbE}^\omega_N \Big[\exp (-\mu^\omega_x m^\omega_x(\overline{\varphi}^\omega) m^\omega_x(\varphi - \widetilde{\varphi}^\omega)) \Big] \hat{\bbE}^\omega_N \Big[ \exp \Big( - \frac{\mu^\omega_x}{2}m^\omega_x(\varphi)\Big)\Big]\\
            & \quad \cdot \exp \Big( -\frac{\mu^\omega_x}{2} m^\omega_x(\overline{\varphi}^\omega)^2 - \mu^\omega_x m^\omega_x(\widetilde{\varphi}^\omega) m^\omega_x(\overline{\varphi}^\omega) \Big) \,.
        \end{aligned}
    \end{equation}
    Let $\widetilde{\bbP}^\omega(\De \varphi) \propto \exp ( -\mu^\omega_x m^\omega_x(\varphi)^2 /2 ) \bbP^\omega(\De \varphi)$.
    By direct computation, for every $x \in \bbZ^d$ it holds that $\bbE^\omega[m^\omega_x(\varphi)^2] = g^\omega(x,x) - (\mu^\omega_x)^{-1}$. Hence, under $\widetilde{\bbP}^\omega$ the variable $m^\omega_x(\varphi)$ is centered Gaussian with variance
    \begin{equation}\label{eq:VarianceAvgFieldPTilde}
        \widetilde{\bbE}^\omega [m^\omega_x(\varphi)^2] = \Big( \frac{1}{\mu^\omega_x} \Big) \Big[\frac{\mu^\omega_x g^\omega (x,x)-1}{ \mu^\omega_x g^\omega(x,x) }\Big] \,.
    \end{equation}
    Then, by Jensen, H\"older and Brascamp-Lieb inequalities we get
    \begin{equation}
        \begin{aligned}
            |m^\omega_x(\widetilde{\varphi}^\omega)| &= \left| \frac{ \hat{\bbE}^\omega_N [m^\omega_x(\varphi) \exp (-\mu^\omega_x m^\omega_x(\varphi)^2 /2)] }{ \hat{\bbE}^\omega_N[ \exp (-\mu^\omega_x m^\omega_x(\varphi)^2 /2) ] } \right| \leq \hat{\bbE}^\omega_N [m^\omega_x(\varphi)^2]^{1/2} \exp \Big( \frac{\mu^\omega_x}{2}\hat{\bbE}^\omega_N[m^\omega_x(\varphi)^2] \Big)\\
            & \leq \bbE^\omega[m^\omega_x(\varphi)^2]^{1/2} \exp \Big( \frac{\mu^\omega_x}{2} \bbE^\omega[m^\omega_x(\varphi)^2] \Big) = \sqrt{g^\omega(x,x) - \frac{1}{\mu^\omega_x}} \exp \Big( \frac{\mu^\omega_x}{2} g^\omega(x,x) - \frac{1}{2} \Big)\\
            & \leq \sqrt{\,\overline{g}} \exp \Big( d\Lambda\overline{g} - \frac{1}{2} \Big) =: k \,.
        \end{aligned}
    \end{equation}
    Moreover, by Brascamp-Lieb inequality
    \begin{equation}
        \begin{aligned}
            \widetilde{\bbE}^\omega_N [\exp (-\mu^\omega_x m^\omega_x (\overline{\varphi}^\omega) m^\omega_x(\varphi- \widetilde{\varphi}^\omega))] &\leq \widetilde{\bbE}^\omega [\exp (-\mu^\omega_x m^\omega_x(\overline{\varphi}^\omega) m^\omega_x(\varphi))]\\
            &\!\! \stackrel{\eqref{eq:VarianceAvgFieldPTilde}}{=} \exp \Big( \frac{m^\omega_x(\overline{\varphi}^\omega)^2 (\mu^\omega_x g^\omega(x,x)-1)}{2 g^\omega(x,x)} \Big) \,.
        \end{aligned}
    \end{equation}
    Thus, by~\eqref{eq:ExpCFEstimate} 
    \begin{equation}
            \bbE^{\omega,+}_N \Big[ \exp \Big( - \frac{\mu^\omega_x}{2} m^\omega_x(\varphi)^2 \Big) \Big] \leq \exp \Big( - \frac{1}{2 g^\omega(x,x)} m^\omega_x(\overline{\varphi}^\omega)^2 + 2dk |m^\omega_x(\overline{\varphi}^\omega)| \Big) \,.
    \end{equation}
    Finally, since $g^\omega(x,x) \leq \overline{g}$, for $\bbQ$-almost all $\omega$, the claim follows.
\end{proof}

\paragraph{\textbf{Acknowledgements.}} While this work was written, the authors were associated with INdAM (Istituto
Nazionale di Alta Matematica ``Francesco Severi'') and the group GNAMPA.

\bibliographystyle{plain}
\bibliography{literature}

\end{document}